\documentclass[english,12pt]{article}
\usepackage{amsmath,amsfonts,amsthm,amssymb,color,mathrsfs}
\usepackage[T1]{fontenc}
\usepackage[utf8]{inputenc}
\usepackage[left=1in, right=1in, top=1.1in,bottom=1.1in]{geometry}
\usepackage{amsthm}
\usepackage{amsmath}
\usepackage{amssymb}
\usepackage{wasysym}
\usepackage{esint}
\usepackage{hyperref}
\usepackage{comment}

\usepackage[title]{appendix}

\hypersetup{
     colorlinks   = true,
     citecolor    = blue,
     linkcolor    = blue
}
\usepackage{pdfsync}
\usepackage{stmaryrd}

\makeatletter

\newcommand{\bch}{\bar{\mathcal{H}}}
\newcommand{\hb}[1]{\textcolor{blue}{#1}}

\def\mr{{\mathbb  R}}

\newcommand{\beq}{\begin{equation}}
\newcommand{\eeq}{\end{equation}}

\newcommand{\E}{\mathbb E}

\newcommand{\R}{\mathbb R}

\newcommand{\ca}{\mathcal A}

\newcommand{\cac}{\mathcal C}

\newcommand{\cf}{\mathcal F}
\newcommand{\ch}{\mathcal H}

\newcommand{\ck}{\mathcal K}
\newcommand{\cl}{\mathcal L}

\newcommand{\cs}{\mathcal S}

\newcommand{\al}{\alpha}
\newcommand{\ep}{\varepsilon}

\newcommand{\ga}{\gamma}

\newcommand{\oom}{\Omega}

\newcommand{\vp}{\varphi}

\newcommand{\lp}{\left(}
\newcommand{\rp}{\right)}
\newcommand{\lc}{\left[}
\newcommand{\rc}{\right]}
\newcommand{\lcl}{\left\{}
\newcommand{\rcl}{\right\}}

\numberwithin{equation}{section}
  \theoremstyle{plain}
  \newtheorem{thm}{\protect\theoremname}[section]
   \theoremstyle{plain}
  
  \theoremstyle{remark}
  \newtheorem{rem}[thm]{\protect\remarkname}
  \theoremstyle{definition}
  \newtheorem*{notation*}{\protect\notationname}
  \theoremstyle{plain}
  \newtheorem{prop}[thm]{\protect\propositionname}
    \theoremstyle{plain}
  \newtheorem{lem}[thm]{\protect\lemmaname}
  \theoremstyle{plain}
  \newtheorem{cor}[thm]{\protect\corollaryname}
  \theoremstyle{definition}
  \newtheorem{defn}[thm]{\protect\definitionname}
   \theoremstyle{plain}
  
   \theoremstyle{plain}
  \newtheorem*{uha*}{\protect\uhaname}
     \theoremstyle{definition}
      \newtheorem*{uell*}{\protect\uellname}
     \theoremstyle{definition}
  \newtheorem{conv}[thm]{\protect\convname}

  \makeatother

\usepackage{babel}
  \providecommand{\assumptionname}{Assumption}
  \providecommand{\definitionname}{Definition}
  \providecommand{\lemmaname}{Lemma}
  \providecommand{\notationname}{Notation}
  \providecommand{\propositionname}{Proposition}
  \providecommand{\remarkname}{Remark}
\providecommand{\corollaryname}{Corollary}
\providecommand{\theoremname}{Theorem}
\providecommand{\claimname}{Claim}
\providecommand{\uhaname}{Uniform Hypoellipticity Assumption}
\providecommand{\uellname}{Uniform Ellipticity Assumption}
\providecommand{\convname}{Convention}

\begin{document}

\title{Precise Local Estimates for Hypoelliptic Differential Equations driven by Fractional Brownian Motion}

\author{Xi Geng\thanks{School of Mathematics and Statistics, University of Melbourne, Melbourne, Australia. Email: xi.geng@unimelb.edu.au.} 
\and 
Cheng Ouyang\thanks{Department of Mathematics, Statistics and Computer Science, University of Illinois at Chicago, Chicago, United States. Email: couyang@math.uic.edu.} 
\and 
Samy Tindel\thanks{Department of Mathematics, Purdue University, West Lafayette, United States. Email: stindel@purdue.edu.} 
}
\date{}

\maketitle

\begin{abstract}
This article is concerned with stochastic differential equations driven by a $d$ dimensional fractional Brownian motion with Hurst parameter $H>1/4$, understood in the rough paths sense. Whenever the coefficients of the equation satisfy a uniform hypoellipticity condition, we establish a sharp local estimate on the associated control distance function and a sharp local lower estimate on the density of the solution. Our methodology relies heavily on the rough paths structure of the equation.
\end{abstract}

\tableofcontents

\section{Introduction.}

We will split our introduction into two parts. In Section \ref{sec:background}, we recall some background on the stochastic analysis of stochastic differential equations driven by a fractional Brownian motion. In Section \ref{sec: main results} we describe our main results.  Section \ref{sec: strategy} is then devoted to a brief explanation about the methodology we have used in order to obtain our main results.

\subsection{Background and motivation.}\label{sec:background}
One way to envision Malliavin calculus is to see it as a geometric and analytic framework on an infinite dimensional space (namely the Wiener space) equipped with a Gaussian measure. This is already apparent in Malliavin's seminal contribution \cite{Ma78} giving a probabilistic proof of H\"ormander's theorem. The same point of view has then been  pushed forward in the celebrated series of papers by Kusuoka and Stroock, which set up the basis for densities and probabilities expansions for diffusion processes within this framework.

On the other hand, the original perspective of Lyons' rough path theory (cf. \cite{Lyons94, Lyons98}) is quite different. Summarizing very briefly, it asserts that a reasonable differential calculus with respect to a noisy process $X$ can be achieved as long as one can define enough iterated integrals of $X$. One of the first processes to which the theory has been successfully applied is a fractional Brownian motion, and we shall focus on this process in the present paper. Namely a $\mathbb{R}^d$-valued fractional Brownian motion is a continuous centered Gaussian process $B=\{(B_t^1,\ldots,B_t^d); \, t\geq 0\}$ with independent coordinates, such that each $B^{j}$ satisfies
\begin{equation*}
\E\lc  (B_{t}^{j} - B_{s}^{j})^{2} \rc
=
|t-s|^{2H},
\qquad
s,t\geq 0,
\end{equation*}
for a given $H\in(0,1)$. The process $B$ can be seen as a natural generalization of Brownian motion allowing any kind of H\"older regularity (that is a H\"older exponent $H-\ep$ for an arbitrary small $\ep$, whenever $H$ is given). We are interested in the following differential equation driven by $B$:
\begin{equation}\label{eq: hypoelliptic SDE}
\begin{cases}
dX_{t}=\sum_{\alpha=1}^{d}V_{\alpha}(X_{t})dB_{t}^{\alpha}, & 0\leq t\leq1,\\
X_{0}=x\in\mathbb{R}^{N}.
\end{cases}
\end{equation}
Here the $V_\alpha$'s are $C_b^\infty$ vector fields, and the Hurst parameter is assumed to satisfy the condition $H>1/4$. In this setting, putting together the results contained in \cite{CQ} and \cite{Lyons98}, the stochastic differential equation \eqref{eq: hypoelliptic SDE} can be understood in the framework of rough path theory. Although we will give an account on the notion of rough path solution in Section \ref{sec:signature}, the simplest way of looking at equation \eqref{eq: hypoelliptic SDE} is the following. Let $B^{(n)}_t$ be a dyadic  linear interpolation of $B_t$. Let $X^{(n)}_t$ be the solution to equation \eqref{eq: hypoelliptic SDE} in which the driving process $B_t$ is replaced by $B^{(n)}_t$. From standard ODE theory,  $X^{(n)}_t$ is pathwisely well-defined. The solution to the SDE (\ref{eq: hypoelliptic SDE}) is then proved (cf. \cite{FV06} for instance) to be the limit of $X^{(n)}_t$ as $n\rightarrow \infty$.

With the solution of \eqref{eq: hypoelliptic SDE} in hand, a natural problem one can think of is the following: can we extend the aforementioned analytic studies on Wiener's space to the process $B$? In particular can we complete Kusuoka-Stroock's program in the fractional Brownian motion setting? This question has received a lot of attention in the recent years, and previous efforts along this line include H\"ormander type theorems for the process $X$ defined by \eqref{eq: hypoelliptic SDE} (cf. \cite{BH07,CF10,CHLT15}), some upper Gaussian bounds on the density $p(t,x,y)$ of $X_{t}$ (cf. \cite{BNOT16}), as well as Varadhan type estimates for $\log(p(t,x,y))$ in small time \cite{BOZ15}. One should stress at this point that the road from the Brownian to the fractional Brownian case is far from being trivial. This is essentially due to the lack of independence of the fBm increments and Markov property,  as well as to the technically demanding characterization of the Cameron-Martin space whenever $B$ is not a Brownian motion. We shall go back to those obstacles throughout the article.

Our contribution can be seen as a step in the direction mentioned above. More specifically, we shall obtain  a sharp local estimate on the associated control distance function and a sharp local estimates for the density of $X_{t}$ under hypoelliptic conditions on the vector fields $V_{\al}$. This will be achieved thanks to a combination of geometric and analytic tools which can also be understood as a mix of stochastic analysis and rough path theory. We describe our main results more precisely in the next subsection.

\subsection{Statement of main results.}\label{sec: main results}

Let us recall that equation \eqref{eq: hypoelliptic SDE} is our main object of concern.
We are typically interested in the degenerate case where the vector fields $V=\{V_1,\ldots,V_d\}$ satisfy the so-called \textit{uniform hypoellipticity} assumption to be defined shortly. This is a standard degenerate setting where one can expect that the solution of the SDE (\ref{eq: hypoelliptic SDE}) admits a smooth density with respect to the Lebesgue measure. As mentioned in Section \ref{sec:background}, we wish to obtain quantitative information for the density in this context.

We first formulate the uniform hypoellipticity condition which will be assumed throughout the rest of the paper. For $l\geq1,$ define $\mathcal{A}(l)$ to be the set of words over letters $\{1,\ldots,d\}$ with length at most $l$ (including the empty word), and $\mathcal{A}_1(l)\triangleq \mathcal{A}(l)\backslash\{\emptyset\}$. Denote $\mathcal{A}_1$ as the set of all non-empty words. Given a word $\alpha\in\ca_{1}$, we define the vector field $V_{[\alpha]}$ inductively by $V_{[i]}\triangleq V_i$ and $V_{[\alpha]}\triangleq[V_i,V_{[\beta]}]$ for $\alpha=(i,\beta)$ with $i$ being a letter and $\beta\in\ca_{1}$.
\begin{uha*}
The vector fields $(V_{1},\ldots,V_{d})$ are $C_{b}^{\infty}$, and there exists an integer $l_{0}\geq1$, such that
\begin{equation}\label{eq:unif-hypo-assumption}
\inf_{x\in\mathbb{R}^N}\inf_{\eta\in S^{N-1}}\left\{ \sum_{\alpha\in\mathcal{A}_1(l_0)}\langle V_{[\alpha]}(x),\eta\rangle_{\mathbb{R}^N}^{2}\right\}>0.
\end{equation}
The smallest such $l_0$ is called the hypoellipticity constant for the vector fields.
\end{uha*}

\begin{rem}
The uniform hypoellipticity assumption is a quantitative description of the standard uniform H\"ormander condition that the family of vectors $\{V_{[\alpha]}(x):\alpha\in\mathcal{A}_1(l_0)\}$ span the tangent space $T_x\mathbb{R}^N$ uniformly with respect to $x\in\mathbb{R}^N$.
\end{rem}

Under condition \eqref{eq:unif-hypo-assumption}, it was proved by Cass-Friz \cite{CF10} and Cass-Hairer-Litterer-Tindel \cite{CHLT15} that the solution to the SDE (\ref{eq: hypoelliptic SDE}) admits a smooth density $y\mapsto p(t,x,y)$ with respect to the Lebesgue measure on $\mathbb{R}^N$ for all $(t,x)\in (0,1]\times\mathbb{R}^N$. Our contribution aims at getting quantitative small time estimates for $p(t,x,y)$.

In order to describe our bounds on the density $p(t,x,y)$, let us recall that the small time behavior of $p(t,x,y)$ is closely related to the so-called control distance function associated with the vector fields. This fact was already revealed in the Varadhan-type asymptotics result proved by Baudoin-Ouyang-Zhang \cite{BOZ15}:
\begin{equation}\label{eq:varadhan-estimate}
\lim_{t\rightarrow0}t^{2H}\log p(t,x,y)=-\frac{1}{2}d(x,y)^{2}.
\end{equation}
The control distance function $d(x,y)$ in \eqref{eq:varadhan-estimate} is defined in the following way. For any continuous path $h:[0,1]\rightarrow\mathbb{R}^d$ with finite $q$-variation for some $1\leq q<2$, denote $\Phi_t(x;h)$ as the solution to the ODE 
\begin{equation}\label{eq: skeleton ODE}
\begin{cases}
dx_{t}=\sum_{\alpha=1}^{d}V_{\alpha}(x_{t})dh_{t}^{\alpha}, & 0\leq t\leq1,\\
x_{0}=x,
\end{cases}
\end{equation}which is well-posed in the sense of Young \cite{Young36} and Lyons \cite{Lyons94}. Let $\bar{\mathcal{H}}$ be the Cameron-Martin subspace for the fractional Brownian motion with Hurst parameter $H$, whose definition will be recalled in Section \ref{sec: prel.}. According to a variational embedding theorem due to Friz-Victoir \cite{FV06} (cf. Proposition \ref{prop: variational embedding} in Section \ref{sec: prel.}), every Cameron-Martin path $h\in\bar{\mathcal{H}}$ has finite $q$-variation for some $1\leq q<2$ so that the ODE (\ref{eq: skeleton ODE}) can be solved for all such $h$. With those preliminary considerations in hand, the distance $d$ in \eqref{eq:varadhan-estimate} is given in the next definition.
\begin{defn}\label{def: CM bridge}
For $x,y\in\mathbb{R}^N$ and $\Phi_t(x;h)$ defined as in \eqref{eq: skeleton ODE}, set 
\begin{equation}\label{eq: Pi_xy}
\Pi_{x,y}\triangleq\left\{ h\in\bar{\mathcal{H}}:\Phi_{1}(x;h)=y\right\} 
\end{equation}
to be the space of Cameron-Martin paths which \textit{join} $x$ \textit{to} $y$ \textit{in the sense of differential equations}. The  \textit{control distance function} $d(x,y)=d_{H}(x,y)$ is defined by 
\[
d(x,y)\triangleq\inf\left\{ \|h\|_{\bar{{\cal H}}}:h\in\Pi_{x,y}\right\} ,\ \ \ x,y\in\mathbb{R}^{N}.
\]
\end{defn}

According to \eqref{eq:varadhan-estimate}, one can clearly expect that the Cameron-Martin structure and the control distance function will play an important role in understanding the small time behavior of $p(t,x,y)$. However, unlike the diffusion case and due to the complexity of the Cameron-Martin structure, the function $d(x,y)$ is far from being a metric and its shape is not clear. Our first main result is thus concerned with the local behavior of $d(x,y)$. It establishes a comparison between $d$ and  the Euclidian distance. 

\begin{thm}\label{thm: local comparison}
Under the same assumptions as in Theorem \ref{thm:equiv-distances}, let $l_0$ be the hypoellipticity constant in assumption \eqref{eq:unif-hypo-assumption} and $d$ be the control distance given in Definition \ref{def: CM bridge}.
There exist constants $C_1,C_2,\delta>0$, where $C_1,C_2$ depend only on $H,l_0$ and the vector fields, and where $\delta$ depends only on $l_0$ and the vector fields, such that 
\begin{equation}\label{eq:local-comparison}
C_{1}|x-y|\leq d(x,y)\leq C_{2}|x-y|^{\frac{1}{l_{0}}} \, ,
\end{equation}
for all $x,y\in \mathbb{R}^N$ with $|x-y|<\delta$.
\end{thm}

We are in fact able to establish a stronger result, namely, the local equivalence of $d$ to the sub-Riemannian distance induced by the vector fields $\{V_1,..., V_d\}$. More specifically, let us write the distance given in Definition \ref{def: CM bridge} as $d_H(x,y)$, in order to emphasize the dependence on the Hurst parameter $H$. Our second main result asserts that all the distances $d_{H}$ are locally equivalent.

\begin{thm}\label{thm:equiv-distances}
Assume that the vector fields $(V_{1},\ldots,V_{d})$ satisfy the uniform hypoellipticity condition \eqref{eq:unif-hypo-assumption}. For $H\in (1/4,1)$, consider the distance $d_{H}$ given in Definition \ref{def: CM bridge}.
Then
for any $H_1, H_2\in (1/4,1)$, there exist  constants $C=C(H_1,H_2, V)>0$ and $\delta>0$ such that
\begin{equation}\label{eq:equiv-distances}
\frac{1}{C}d_{H_1}(x,y)\leq d_{H_2}(x,y)\leq Cd_{H_1}(x,y),
\end{equation}
for all $x,y\in\mathbb{R}^N$ with $|x-y|<\delta$.  In particular, all distances $d_H$ are locally equivalent to $d_{\textsc{BM}}\equiv d_{1/2}$, where $d_{\textsc{BM}}$ stands for the controlling distance of the system \eqref{eq: skeleton ODE} driven by a Brownian motion, i.e. the sub-Riemannian distance induced by the vector fields $\{V_1,...,V_d\}$.
\end{thm}



\begin{rem}\label{rmk:density-signature}
In the special case when \eqref{eq: hypoelliptic SDE} reads as 
$dX_t=X_t\otimes dB_t,$
that is, when $X_t$ is the truncated signature of $B$ up to order $l>0$, it is proved in \cite{BFO19} that all $d_H(x,y)$ are globally equivalent. The proof crucially depends on the fact that the signature of $B$ is homogeneous with respect to the dilation operator on $G^{(l)}(\mr^d)$, the free nilpotent Lip group over $\mr^d$ of order $l$. In the current general nonlinear case, the local equivalence is much more technically challenging.  In addition, we believe that the global equivalence of the distances $d_{H}$ does not hold.
\end{rem}

Our second main result asserts that the density $p(t,x,y)$ of $X_{t}$ is strictly positive everywhere whenever $t>0$. It generalizes for the first time the result of \cite[Theorem 1.4]{BNOT16} to a general hypoelliptic case, by affirming that Hypothesis 1.2 in that theorem is always verified under our assumption \eqref{eq:unif-hypo-assumption}. Recall that a \textit{distribution} over a differentiable manifold is a smooth choice of subspace of the tangent space at every point with constant dimension.  

\begin{thm}\label{prop: positivity-intro}
Let $\{V_{1},\ldots,V_{d}\}$ be a family of $C_{b}^{\infty}$-vector
fields on $\mathbb{R}^{N},$ which span a distribution over $\mathbb{R}^N$ and satisfy the uniform hypoellipticity assumption \eqref{eq:unif-hypo-assumption}.  Let $X_{t}^{x}$
be the solution to the stochastic differential equation  \eqref{eq: hypoelliptic SDE},
where $B_{t}$ is a $d$-dimensional fractional Brownian motion with
Hurst parameter $H>1/4$. Then for each $t\in(0,1]$, the density
of $X_{t}$ is everywhere strictly positive.
\end{thm}

As we will see in Section \ref{section: step 1}, the proof of the above result is  based on finite dimensional geometric arguments such as the classical Sard theorem, as well as a general positivity criteria for densities on the Wiener space. We believe that the positivity result in Theorem \ref{prop: positivity-intro} is non-trivial and interesting in its own right. 
 
Let us now turn to a description of our third main result. It establishes a sharp local lower estimate for the density function $p(t,x,y)$ of the solution to the SDE (\ref{eq: hypoelliptic SDE}) in small time. 

\begin{thm}\label{thm: local lower estimate}
Under the uniform hypoelliptic assumption (\ref{eq:unif-hypo-assumption}), let $p(t,x,y)$ be the density of the random variable $X_{t}$ defined by equation \eqref{eq: hypoelliptic SDE}.
There exist some constants $C,\tau>0$ depending only on $H,l_0$ and the vector fields $V_{\al}$, such that 
\begin{equation}\label{eq:local-density-bound}
p(t,x,y)\geq\frac{C}{|B_{d}(x,t^{H})|},
\end{equation}
for all $(t,x,y)\in(0,1]\times\mathbb{R}^N\times\mathbb{R}^N$ satisfying the following local condition involving the distance $d$ introduced in Definition \ref{def: CM bridge}:
\begin{equation*}
d(x,y)\leq t^H,
\quad\text{and}\quad
t<\tau .
\end{equation*}
In relation \eqref{eq:local-density-bound}, $B_{d}(x,t^{H})\triangleq\{ z\in\mathbb{R}^{N}:d(x,z)<t^{H}\}$ denotes the ball with respect to the distance $d$ and $|\cdot|$ stands for the Lebesgue measure.
\end{thm}

The sharpness of Theorem \ref{thm: local lower estimate} can be seen from the fractional Brownian motion case, i.e. when $N=d$ and $V=\mathrm{Id}$. As we will see, the technique we use to prove Theorem \ref{thm: local comparison} will be an essential ingredient for establishing Theorem \ref{thm: local lower estimate}. Theorem \ref{thm:equiv-distances} and Theorem \ref{prop: positivity-intro} will also be proved as byproducts along our path of proving Theorem \ref{thm: local lower estimate}.

\subsection{Strategy and outlook.}\label{sec: strategy}

Let us say a few words about the methodology we have used in order to obtain our main results. Although we will describe our overall strategy with more details in Section \ref{sec: local lower bound}, let us  mention here that it
 is based on the reduction of the problem to a finite dimensional one, plus some geometric type arguments.  

More specifically, the key point in our proofs is that the solution $X_{t}$ to \eqref{eq: hypoelliptic SDE} can be approximated by a simple enough function of the so-called truncated signature of order $l$ for the fractional Brownian motion $B$. This object is formally defined,  for a given $l\ge 1$, as the following $\oplus_{k=0}^{l}(\R^{d})^{\otimes k}$-valued process:
\begin{equation*}
\Gamma_t
=
1+\sum_{k=1}^{l}\int_{0<t_1<\cdots<t_k<t}dB_{t_1}\otimes \cdots\otimes dB_{t_k},
\end{equation*}
and  it enjoys some convenient algebraic and analytic properties. The truncated signature is the main building block of the rough path theory (see e.g \cite{Lyons98}), and was also used in \cite{KS87} in a Malliavin calculus context. Part of our challenge in the current contribution is to combine the properties of the process $\Gamma$, together with the Cameron-Martin space structure related to the fractional Brownian motion $B$, in order to achieve efficient bounds for the density of $X_{t}$.

As mentioned above, the truncated signature gives rise to a $l$-th order local approximation of $X_{t}$ in a neighborhood of its initial condition $x$. Namely if we set 
\begin{equation}\label{eq:def-Fl-intro}
F_{l}(\Gamma_{t},x)\triangleq
\sum_{k=1}^{l} \sum_{i_{1},\ldots,i_{k}=1}^{d}
V_{(i_{1},\ldots,i_{k})}(x)
\int_{0<t_1<\cdots<t_k<t}dB_{t_1}^{i_{1}} \cdots dB_{t_k}^{i_{k}},
\end{equation}
then classical rough paths considerations assert that $F_{l}(\Gamma_{t},x)$ is an approximation of order $t^{Hl}$ of $X_{t}$ for small $t$. In the sequel we will heavily rely on some non degeneracy properties of $F_{l}$ derived from the uniform hypoelliptic assumption \eqref{eq:unif-hypo-assumption}, in order to get the following information:

\noindent
\emph{(i)}
One can construct a path $h$ in the Cameron-Martin space of $B$ which joins $x$ and any point $y$ in a small enough neighborhood of $x$. This task is carried out thanks to a complex iteration procedure, whose building block is the non-degeneracy of the function $F_{l}$. It is detailed in Section \ref{sec:proof-thm-12}. In this context, observe that the computation of the Cameron-Martin norm of $h$ also requires a substantial effort. This will be the key step in order to prove Theorems \ref{thm:equiv-distances} and \ref{thm: local comparison} concerning the distance $d$ given in Definition \ref{def: CM bridge}.

\noindent
\emph{(ii)}
The proof of the lower bound given in Theorem \ref{thm: local lower estimate} also hinges heavily on the approximation $F_{l}$ given by \eqref{eq:def-Fl-intro}. Indeed the preliminary results about the density of $\Gamma_{t}$ (see Remark~\ref{rmk:density-signature} above), combined with the non-degeneracy of $F_{l}$, yield good properties for the density of $F_{l}(\Gamma_{t},x)$. One is then left with the task of showing that $F_{l}(\Gamma_{t},x)$ approximates $X_{t}$ properly at the density level.

In conclusion, although the steps performed in the remainder of the article might look technically and computationally involved, they rely on a natural combination of analytic and geometric bricks as well as a reduction to a finite dimensional problem. Let us also highlight the fact that our next challenge is to iterate the local estimates presented here in order to get Gaussian type lower bounds for the density $p(t,x,y)$ of $X_{t}$. This causes some further complications due to  the complex (non Markovian) dependence structure for the increments of the fractional Brownian motion $B$. We defer this project to a future publication.

\medskip

\noindent
\textbf{Organization of the present paper.} In Section \ref{sec:prelim}, we present some basic notions from the analysis of fractional Brownian motion and rough path theory. In Section \ref{sec: ellip.}, we give an independent discussion in the elliptic case in which the analysis is considerably simpler. In Section \ref{sec: control distance} and Section \ref{sec: local lower bound}, we develop the proofs of Theorem \ref{thm: local comparison} and Theorem \ref{thm: local lower estimate} respectively in the hypoelliptic case. Theorem \ref{thm:equiv-distances} and Theorem \ref{prop: positivity-intro} are proved in the steps towards proving Theorem \ref{thm: local lower estimate}.

\begin{notation*}
Throughout the rest of this paper, we use "$\mathrm{Letter}_{\mathrm{subscript}}$" to denote constants whose value depend only on objects specified in the "subscript" and may differ from line to line. For instance, $C_{H,V,l_0}$ denotes a constant depending only on the Hurst parameter $H$, the vector fields $V$ and the hypoellipticity constant $l_0$. Unless otherwise stated, a constant will implicitly depend on $H,V,l_0$. We will always omit the dependence on dimension. 
\end{notation*}

\section{Preliminary results.}\label{sec:prelim}

This section is devoted to some preliminary results on the Cameron-Martin space related to a fractional Brownian motion. We shall also recall some basic facts about rough paths solutions to noisy equations.

\subsection{The Cameron-Martin subspace of fractional Brownian motion.}\label{sec: prel.}

Let us start by recalling the definition of fractional Brownian motion.

\begin{defn}\label{def:fbm}
A $d$-dimensional \textit{fractional Brownian motion} with Hurst parameter $H\in(0,1)$ is an $\mathbb{R}^d$-valued continuous centered Gaussian process $B_t=(B_t^1,\ldots,B_t^d)$ whose covariance structure is given by 
\beq\label{eq:cov-fBm}
\mathbb{E}[B_{s}^{i}B_{t}^{j}]=\frac{1}{2}\left(s^{2H}+t^{2H}-|s-t|^{2H}\right)\delta_{ij}
\triangleq R(s,t) \delta_{ij}.
\eeq
\end{defn}

This process is defined and analyzed in numerous articles (cf. \cite{DU97,Nualart06,PT00} for instance), to which we refer for further details.
In this section, we mostly focus on a proper definition of the Cameron-Martin subspace related to $B$. We also prove two general lemmas about this space which are needed for our analysis of the density $p(t,x,y)$. 
Notice that we will frequently identify a Hilbert space with its dual in the canonical way without  further mentioning. 

In order to introduce the Hilbert spaces which will feature in the sequel, consider a one dimensional fractional Brownian motion $\{B_t:0\leq t\leq 1\}$ with Hurst parameter $H\in(0,1)$.   The discussion here can be easily adapted to the multidimensional setting with arbitrary time horizon $[0,T]$. 
Denote $W$ as the space of continuous paths $w:[0,1]\rightarrow\mathbb{R}^{1}$
with $w_{0}=0.$ Let $\mathbb{P}$ be the probability measure over $W$ under which the coordinate process $B_t(w)=w_t$ becomes a fractional Brownian motion.  Let ${\cal C}_{1}$ be the associated first order Wiener chaos,
i.e. ${\cal C}_{1}\triangleq\overline{\mathrm{Span}\{B_{t}:0\leq t\leq1\}}\ {\rm in}\ L^{2}(W,\mathbb{P})$.

\begin{defn}\label{def:bar-H}
Let $B$ be a one dimensional fractional Brownian motion as defined in \eqref{def:fbm}.
Define $\bar{{\cal H}}$ to be the space of elements $h\in W$ which
can be written as 
\beq\label{eq:def-h-in-CM}
h_{t}=\mathbb{E}[B_{t}Z],\ \ \ 0\leq t\leq1,
\eeq
where $Z\in{\cal C}_{1}.$ We equip $\bar{{\cal H}}$ with an inner product structure given by
\[
\langle h_{1},h_{2}\rangle_{\bar{{\cal H}}}\triangleq\mathbb{E}[Z_{1}Z_{2}],\ \ \ h_{1},h_{2}\in\bar{{\cal H}},
\]
whenever $h^{1},h^{2}$ are defined by \eqref{eq:def-h-in-CM} for two random variables $Z_{1},Z_{2}\in{\cal C}_{1}$.
The Hilbert space $(\bar{\mathcal{H}},\langle\cdot,\cdot\rangle_{\bar{\mathcal{H}}})$ is called the \textit{Cameron-Martin
subspace} of the fractional Brownian motion. 
\end{defn}

One of the advantages of working with fractional Brownian motion is that a convenient analytic description of $\bar{\mathcal{H}}$ in terms of fractional calculus is available (cf. \cite{DU97}).  Namely recall that given a function $f$ defined on $[a,b]$, the right and left \textit{fractional integrals} of $f$ of order $\alpha>0$ are respectively defined by
\beq\label{eq:def-frac-integral}
(I_{a^{+}}^{\alpha}f)(t)\triangleq\frac{1}{\Gamma(\alpha)}\int_{a}^{t}f(s)(t-s)^{\alpha-1}ds,
\quad\text{and}\quad 
(I_{b^{-}}^{\alpha}f)(t)\triangleq\frac{1}{\Gamma(\alpha)}\int_{t}^{b}f(s)(s-t)^{\alpha-1}ds.
\eeq
In the same way the right and left \textit{fractional derivatives} of $f$ of order $\alpha>0$ are respectively defined by  
\beq\label{eq:def-frac-deriv}
(D_{a^{+}}^{\alpha}f)(t)\triangleq\left(\frac{d}{dt}\right)^{[\alpha]+1}(I_{a^{+}}^{1-\{\alpha\}}f)(t),
\quad\text{and}\quad 
(D_{b^{-}}^{\alpha}f)(t)\triangleq\left(-\frac{d}{dt}\right)^{[\alpha]+1}(I_{b^{-}}^{1-\{\alpha\}}f)(t),
\eeq
where $[\alpha]$ is the integer part of $\alpha$ and $\{\alpha\}\triangleq\alpha-[\alpha]$ is the fractional part of $\alpha$. The following formula for $D_{a^+}^\alpha$ will be useful for us:
\begin{equation}\label{eq: formula for fractional derivatives}
(D_{a^{+}}^{\alpha}f)(t)=\frac{1}{\Gamma(1-\alpha)}\left(\frac{f(t)}{(t-a)^{\alpha}}+\alpha\int_{a}^{t}\frac{f(t)-f(s)}{(t-s)^{\alpha+1}}ds\right),\ \ \ t\in[a,b].
\end{equation}
The fractional integral and derivative operators are inverse to each other. For this and other properties of fractional derivatives, the reader is referred to \cite{KMS93}.

Let us now go back to the construction of the Cameron-Martin space for $B$, and proceed as in \cite{DU97}. Namely define an isomorphism $K$ between $L^{2}([0,1])$
and $I_{0+}^{H+1/2}(L^{2}([0,1]))$ in the following way:
\begin{equation}\label{eq: analytic expression of K}
K\varphi\triangleq\begin{cases}
C_{H}\cdot I_{0^{+}}^{1}\left(t^{H-\frac{1}{2}}\cdot I_{0^{+}}^{H-\frac{1}{2}}\left(s^{\frac{1}{2}-H}\varphi(s)\right)(t)\right), & H>\frac{1}{2};\\
C_{H}\cdot I_{0^{+}}^{2H}\left(t^{\frac{1}{2}-H}\cdot I_{0^{+}}^{\frac{1}{2}-H}\left(s^{H-\frac{1}{2}}\varphi(s)\right)(t)\right), & H\leq\frac{1}{2},
\end{cases}
\end{equation}
where $c_{H}$ is a universal constant depending only on $H.$ One
can easily compute $K^{-1}$ from the definition of $K$ in terms
of fractional derivatives. Moreover, the operator $K$ admits a kernel representation,
i.e. there exits a function $K(t,s)$ such that 
\[
(K\varphi)(t)=\int_{0}^{t}K(t,s)\varphi(s)ds,\ \ \ \varphi\in L^{2}([0,1]).
\]
The kernel $K(t,s)$ is defined for $s<t$ (taking zero value
otherwise). One can write down $K(t,s)$ explicitly thanks to the definitions \eqref{eq:def-frac-integral} and \eqref{eq:def-frac-deriv}, but this expression is not included here since it will not be used later in  our analysis. A crucial property for $K(t,s)$ is that 
\begin{equation}
R(t,s)=\int_{0}^{t\wedge s}K(t,r)K(s,r)dr,\label{eq: kernel representation of covariance function for fBM}
\end{equation} 
where $R(t,s)$ is the fractional Brownian motion covariance function introduced in \eqref{eq:cov-fBm}. This essential fact enables the following analytic characterization of the Cameron-Martin space in \cite[Theorem~3.1]{DU97}.

\begin{thm}\label{thm:bar-cal-H-as frac-integral}
Let $\bch$ be the space given in Definition \ref{def:bar-H}.
As a vector space we have $\bar{{\cal H}}=I_{0^{+}}^{H+1/2}(L^{2}([0,1])),$ and 
the Cameron-Martin norm is given by 
\begin{equation}
\|h\|_{\bar{{\cal H}}}=\|K^{-1}h\|_{L^{2}([0,1])}.\label{eq: inner product in terms of fractional integrals}
\end{equation}
\end{thm}

In order to define Wiener integrals with respect to $B$, it is also convenient to look at the Cameron-Martin subspace in terms of the covariance structure. Specifically, we define another space $\mathcal{H}$ as the completion of the space of simple step functions with inner product induced by 
\begin{equation}\label{eq:def-space-H}
\langle\mathbf{1}_{[0,s]},\mathbf{1}_{[0,t]}\rangle_{{\cal H}}\triangleq R(s,t).
\end{equation}
The space $\ch$ is easily related to $\bch$. Namely define the following operator 
\begin{equation}\label{eq:def-K-star}
{\cal K}^{*}:{\cal H}\rightarrow L^{2}([0,1]),
\quad\text{such that}\quad
\mathbf{1}_{[0,t]}\mapsto K(t,\cdot).
\end{equation}
We also set 
\begin{equation}\label{eq:IsoR}
\mathcal{R}\triangleq K\circ\mathcal{K}^{*}:{\cal H}\rightarrow\bar{\mathcal{H}},
\end{equation}where the operator $K$ is introduced in \eqref{eq: analytic expression of K}. Then it can be proved that $\mathcal{R}$ is an isometric isomorphism (cf. Lemma \ref{lem: surjectivity of K*} below for the surjectivity of $\mathcal{K}^*$). In addition, under this identification, $\mathcal{K}^*$ is the adjoint of $K$, i.e. $\mathcal{K}^{*}=K^{*}\circ\mathcal{R}.$ This can be seen by acting on indicator functions and then passing limit. 
As mentioned above, one advantage about the space $\mathcal{H}$ is that the fractional Wiener integral operator $I:{\mathcal{H}}\rightarrow{\mathcal{C}}_{1}$ induced by  
${\bf 1}_{[0,t]}\mapsto B_{t}$ is an isometric isomorphism. According to relation~\eqref{eq: kernel representation of covariance function for fBM},
$B_{t}$ admits a Wiener integral representation with respect to an underlying Wiener process $W$:
\beq\label{eq:B-as-Wiener-integral}
B_{t}=\int_{0}^{t}K(t,s)dW_{s}  .
\eeq
Moreover, the process $W$ in \eqref{eq:B-as-Wiener-integral} can be expressed as a Wiener integral with respect to $B$, that is
$W_{s}=I(({\cal K}^{*})^{-1}{\bf 1}_{[0,s]})$  (cf. \cite[relation (5.15)]{Nualart06}).

Let us also mention the following useful formula for the natural pairing between $\mathcal{H}$ and $\bar{\mathcal{H}}$.

\begin{lem}

Let $\ch$ be the space defined as the completion of the indicator functions with respect to the inner product \eqref{eq:def-space-H}. Also recall that $\bch$ is introduced in Definition \ref{def:bar-H}. Then through the isometric isomorphism $\mathcal{R}$ defined by (\ref{eq:IsoR}), the natural pairing between $\ch$ and $\bch$ is given by
\begin{equation}\label{eq: H-barH pairing}
_{\mathcal{H}}\langle f,h\rangle_{\bar{\mathcal{H}}}=\int_{0}^{1}f_{s} \, dh_{s}.
\end{equation}
\end{lem}

\begin{proof}
First of all, let $h\in\bch$ and $g\in\mathcal{H}$ be such that $\mathcal{R}(g)=h$.  It is easy to see that $g$ can be constructed in the following way. According to Definition \ref{def:bar-H}, there exists a random variable $Z$ in the first chaos $\cac_{1}$ such that $h_t=\mathbb{E}[B_tZ]$. 
The element $g\in\mathcal{H}$ is then given via
the Wiener integral isomorphism between $\ch$ and $\cac_{1}$, that is, the element $g\in\ch$ such that $Z=I(g)$. Also note that we have $h_t=\mathbb{E}[B_t \, I(g)]$.

Now consider  $f\in\ch$. The natural pairing between $f$ and $h$ is thus given by
\begin{equation*}
_{\mathcal{H}}\langle f,h\rangle_{\bar{\mathcal{H}}}
= \,
_{\mathcal{H}}\langle f,g\rangle_{\mathcal{H}}
=
\mathbb{E}[Z\cdot I(f)].
\end{equation*}
A direct application of Fubini's theorem then yields:
\[
_{\mathcal{H}}\langle f,h\rangle_{\bar{\mathcal{H}}}
=\mathbb{E}[Z\cdot I(f)]
=\mathbb{E}\left[Z\cdot\int_{0}^{1}f_{s}dB_{s}\right]
=\int_{0}^{1}f_{s} \, \mathbb{E}[ZdB_{s}]=\int_{0}^{1}f_{s}dh_{s}.
\]
\end{proof}

The space $\mathcal{H}$ can also be described in terms of fractional calculus (cf. \cite{PT00}), since the operator $\mathcal{K}^*$ defined by \eqref{eq:def-K-star} can be expressed as
\begin{equation}\label{eq: analytic expression for K*}
(\mathcal{K}^{*}f)(t)=\begin{cases}
C_{H}\cdot t^{\frac{1}{2}-H}\cdot\left(I_{1^{-}}^{H-\frac{1}{2}}\left(s^{H-\frac{1}{2}}f(s)\right)\right)(t), & H>\frac{1}{2};\\
C_{H}\cdot t^{\frac{1}{2}-H}\cdot\left(D_{1^{-}}^{\frac{1}{2}-H}\left(s^{H-\frac{1}{2}}f(s)\right)\right)(t), & H\leq\frac{1}{2}.
\end{cases}
\end{equation} 
Starting from this expression, it is readily checked  that when  $H>1/2$ the space $\mathcal{H}$ coincides with the following subspace of the Schwartz distributions $\cs'$:
\begin{equation}\label{eq:ch-case-H-large}
\ch 
=
\lcl f\in\cs' ; \,  t^{1/2-H}\cdot(I_{1^{-}}^{H-1/2}(s^{H-1/2}f(s)))(t) \text{ is an element of } L^{2}([0,1])  \rcl.
\end{equation}
 In the case $H\leq1/2$, we simply have 
 \begin{equation}\label{eq:ch-case-H-small}
\mathcal{H}=I_{1^{-}}^{1/2-H}(L^{2}([0,1])).
\end{equation}

\begin{rem}
As the Hurst parameter $H$ increases, $\mathcal{H}$ gets larger (and contains distributions when $H>1/2$) while $\bar{\mathcal{H}}$ gets smaller. This fact is apparent from Theorem \ref{thm:bar-cal-H-as frac-integral} and relations~\eqref{eq:ch-case-H-large}-\eqref{eq:ch-case-H-small}.
When $H=1/2$, the process $B_t$ coincides with the usual Brownian motion. In this case, we have $\mathcal{H}=L^2([0,1])$ and $\bar{\mathcal{H}}=W_0^{1,2}$, the space of absolutely continuous paths starting at the origin with square integrable derivative.
\end{rem}

Next we mention a variational embedding theorem for the Cameron-Martin subspace $\bar{\mathcal{H}}$ which will be used in a crucial way. The case when $H>1/2$ is a simple exercise starting from the definition \eqref{eq:def-h-in-CM} of $\bar{\mathcal{H}}$ and invoking the Cauchy-Schwarz inequality. The case when $H\leq 1/2$ was treated in \cite{FV06}. From a pathwise point of view, this allows us to integrate a fractional Brownian path against a Cameron-Martin path or vice versa (cf. \cite{Young36}), and to make sense of ordinary differential equations driven by a Cameron-Martin path (cf. \cite{Lyons94}).

\begin{prop}\label{prop: variational embedding}If $H>\frac{1}{2}$, then
$\bar{\mathcal{H}}\subseteq C_{0}^{H}([0,1];\mathbb{R}^{d})$, the space of H-H\"older continuous paths. If $H\leq\frac{1}{2}$,
then for any $q>\left(H+1/2\right)^{-1}$, we have 
$\bar{\mathcal{H}}\subseteq C_{0}^{q\text{-}\mathrm{var}}([0,1];\mathbb{R}^{d})$, the space of continuous paths with finite $q$-variation. In addition, the above inclusions are continuous embeddings.
\end{prop}

Finally, we prove two general lemmas on the Cameron-Martin subspace that are needed later on. These properties do not seem to be contained in the literature and they require some care based on fractional calculus.  The first one claims the surjectivity of $\ck^{*}$ on properly defined spaces.

\begin{lem}\label{lem: surjectivity of K*}
Let $H\in(0,1)$, and consider 
the operator $\mathcal{K}^*:\mathcal{H}\rightarrow L^2([0,1])$ defined by~\eqref{eq:def-K-star}. Then $\ck^{*}$ is surjective.
\end{lem}

\begin{proof}
If $H>1/2$, we know that the image of $\mathcal{K}^*$ contains all indicator functions (cf. \cite[Equation (5.14)]{Nualart06}). Therefore, $\mathcal{K}^*$ is surjective. 

If $H<1/2$, we first claim that the image of $\mathcal{K}^*$ contains functions of the form $t^{1/2-H}p(t)$ where $p(t)$ is a polynomial. Indeed, given an arbitrary $\beta\geq0,$ consider the function
$$f_{\beta}(t)\triangleq t^{\frac{1}{2}-H}(1-t)^{\beta+\frac{1}{2}-H}.$$ It is readily checked that $D_{1^{-}}^{\frac{1}{2}-H}f_{\beta}\in L^{2}([0,1])$, and hence $f_{\beta}\in I_{1^{-}}^{\frac{1}{2}-H}(L^{2}([0,1]))=\mathcal{H}$. Using the analytic expression (\ref{eq: analytic expression for K*}) for $\mathcal{K}^*$, we can compute $\mathcal{K}^*f_\beta$ explicitly (cf. \cite[Chapter 2, Equation (2.45)]{KMS93}) as\[
(\mathcal{K}^{*}f_{\beta})(t)=C_{H}\frac{\Gamma\left(\beta+\frac{3}{2}-H\right)}{\Gamma(\beta+1)}t^{\frac{1}{2}-H}(1-t)^{\beta}.
\] Since $\beta$ is arbitrary and $\mathcal{K}^*$ is linear, the claim follows.

Now it remains to show that the space of  functions of the form $t^{\frac{1}{2}-H}p(t)$ with $p(t)$ being a polynomial is dense in $L^2([0,1])$. To this end, let $\varphi\in C_c^\infty((0,1))$. Then $\psi(t)\triangleq t^{-(1/2-H)}\varphi(t)\in C_{c}^{\infty}((0,1)).$ According to Bernstein's approximation theorem, for any $\varepsilon>0,$ there exists a polynomial $p(t)$ such that \[
\|\psi-p\|_{\infty}<\varepsilon,
\]and thus\[
\sup_{0\leq t\leq1}|\varphi(t)-t^{\frac{1}{2}-H}p(t)|<\varepsilon.
\]Therefore, functions in $C_c^\infty((0,1))$ (and thus in $L^2([0,1])$) can be approximated by functions of the desired form.
\end{proof}

Our second lemma gives some continuous embedding properties for $\ch$ and $\bch$ in the irregular case $H<1/2$.

\begin{lem}\label{lem: continuous embedding when H<1/2}
For $H<1/2$, the inclusions $\mathcal{H}\subseteq L^2([0,1])$ and $W_0^{1,2}\subseteq \bar{\mathcal{H}}$ are continuous embeddings.
\end{lem}

\begin{proof}

For the first assertion, let $f\in\mathcal{H}$. We wish to prove that
\begin{equation}\label{eq: continuity of H to L^2}
\|f\|_{L^{2}([0,1])}\leq C_{H}\|f\|_{\mathcal{H}}.
\end{equation}
Towards this aim, define $\varphi\triangleq \mathcal{K}^*f$, where $\ck^{*}$ is defined by \eqref{eq:def-K-star}. Observe that $\ck^{*}:\ch\to L^2([0,1])$ and thus $f\in L^2([0,1])$. By solving $f$ in terms of $\varphi$ using the analytic expression~\eqref{eq: analytic expression for K*} for $\mathcal{K}^*$, we have 
\beq\label{a1}
f(t)=C_H t^{\frac{1}{2}-H}\left(I_{1-}^{\frac{1}{2}-H}\left(s^{H-\frac{1}{2}}\varphi(s)\right)\right)(t).
\eeq
We now bound the right hand side of \eqref{a1}. Our first step in this direction is to notice that according to the definition \eqref{eq:def-frac-integral} of fractional integral we have
\begin{align*}
 \left|\left(I_{1^{-}}^{\frac{1}{2}-H}(s^{H-\frac{1}{2}}\varphi(s))\right)(t)\right|
 & =C_{H}\left|\int_{t}^{1}(s-t)^{-\frac{1}{2}-H}s^{H-\frac{1}{2}}\varphi(s)ds\right|\\
 & \leq C_{H}\int_{t}^{1}(s-t)^{-\frac{1}{2}-H}s^{H-\frac{1}{2}}|\varphi(s)|ds\\
 & =C_{H}\int_{t}^{1}(s-t)^{-\frac{1}{4}-\frac{H}{2}} \left((s-t)^{-\frac{1}{4}-\frac{H}{2}}s^{H-\frac{1}{2}}|\varphi(s)|\right)ds .
 \end{align*}
Hence a direct application of Cauchy-Schwarz inequality gives
\begin{align}\label{a2}
  \left|\left(I_{1^{-}}^{\frac{1}{2}-H}(s^{H-\frac{1}{2}}\varphi(s))\right)(t)\right|
& \leq C_{H}\left(\int_{t}^{1}(s-t)^{-\frac{1}{2}-H}ds\right)^{\frac{1}{2}}
 \left(\int_{t}^{1}(s-t)^{-\frac{1}{2}-H}s^{2H-1}|\varphi(s)|^{2}ds\right)^{\frac{1}{2}} \notag\\
 & =C_{H} (1-t)^{\frac{1}{2}\lp\frac{1}{2}-H\rp} 
 \left(\int_{t}^{1}(s-t)^{-\frac{1}{2}-H}s^{2H-1}|\varphi(s)|^{2}ds\right)^{\frac{1}{2}},
\end{align}
where we recall that $C_{H}$ is a positive constant which can change from line to line.
Therefore, plugging \eqref{a2} into \eqref{a1} we obtain
\begin{equation*}
\|f\|_{L^{2}([0,1])}^{2}  
\leq 
C_{H}\int_{0}^{1}t^{1-2H}(1-t)^{\frac{1}{2}-H}\int_{t}^{1}(s-t)^{-\frac{1}{2}-H}s^{2H-1}|\varphi(s)|^{2}ds.
\end{equation*}
We now bound all the terms of the form $s^{\beta}$ with $\beta>0$ by 1. This gives
\begin{align*}
\|f\|_{L^{2}([0,1])}^{2} 
 & \leq C_{H}\int_{0}^{1}dt\int_{t}^{1}(s-t)^{-\frac{1}{2}-H}|\varphi(s)|^{2}ds
  =C_{H}\int_{0}^{1}|\varphi(s)|^{2}ds\int_{0}^{s}(s-t)^{-\frac{1}{2}-H}dt \\
 & =C_{H}\int_{0}^{1}s^{\frac{1}{2}-H}  |\varphi(s)|^{2}ds
  \leq C_{H}\|\varphi\|_{L^{2}([0,1])}^{2}
  =C_{H}\|f\|_{\mathcal{H}}^{2},
\end{align*}
which is our claim \eqref{eq: continuity of H to L^2}.

For the second assertion about the embedding of $W_0^{1,2}$ in $\bar{\mathcal{H}}$, let $h\in W_0^{1,2}$. We thus also have  $h\in\bch$ and we can write $h=K\varphi$ for some $\varphi\in L^2([0,1])$. We first claim that 
\begin{equation}\label{eq: duality}
\int_{0}^{1}f(s)dh(s)=\int_{0}^{1}\mathcal{K}^{*}f(s)\varphi(s)ds
\end{equation}
for all $f\in\mathcal{H}$. This assertion can be reduced in the following way: since $\ch\hookrightarrow L^{2}([0,1])$ continuously and $\ck^{*}:\ch\to L^{2}([0,1])$ is continuous, one can take limits along indicator functions in \eqref{eq: duality}. Thus it is sufficient to consider $f=\mathbf{1}_{[0,t]}$ in \eqref{eq: duality}. In addition, relation~\eqref{eq: duality} can be checked easily for $f=\mathbf{1}_{[0,t]}$. Namely we have
\[
\int_{0}^{1}\mathbf{1}_{[0,t]}(s)dh(s)=h(t)=\int_{0}^{t}K(t,s)\varphi(s)ds=\int_{0}^{1}\left(\mathcal{K}^{*}\mathbf{1}_{[0,t]}\right)(s)\varphi(s)ds.
\]
Therefore, our claim  (\ref{eq: duality}) holds true. Now from Lemma \ref{lem: surjectivity of K*}, if $\vp\in L^{2}([0,1])$ there exists $f\in\mathcal{H}$ such that $\varphi=\mathcal{K}^*f$. For this particular $f$, invoking relation~\eqref{eq: duality} we get
\beq\label{a21}
\int_{0}^{1}f(s)dh(s)=\|\varphi\|_{L^{2}([0,1])}^{2}.
\eeq
But we also know that
\beq\label{a3}
\|\varphi\|_{L^{2}([0,1])}=\|h\|_{\bar{\mathcal{H}}}=\|f\|_{\mathcal{H}},
\quad\text{and thus}\quad
\|\varphi\|_{L^{2}([0,1])}^{2} = \|h\|_{\bar{\mathcal{H}}}  \|f\|_{\mathcal{H}}.
\eeq
In addition recall that the $W^{1,2}$ norm can be written as
\begin{equation*}
\|h\|_{W^{1,2}}  =\sup_{\psi\in L^{2}([0,1])}\frac{\left|\int_{0}^{1}\psi(s)dh(s)\right|}{\|\psi\|_{L^{2}([0,1])}}
\end{equation*}
Owing to \eqref{a21} and \eqref{a3} we thus get
\begin{align*}
\|h\|_{W^{1,2}} & 
\geq\frac{\int_{0}^{1}f(s)dh(s)}{\|f\|_{L^{2}([0,1])}}
 =\frac{\|h\|_{\bar{\mathcal{H}}}\|f\|_{\mathcal{H}}}{\|f\|_{L^{2}([0,1])}}\geq C_{H}\|h\|_{\bar{\mathcal{H}}},
\end{align*}
where the last step stems from (\ref{eq: continuity of H to L^2}). The continuous embedding $W_0^{1,2}\subseteq \bar{\mathcal{H}}$ follows.
\end{proof}

\subsection{Free nilpotent groups.}\label{sec:free-nilpotent}

In this section we introduce some geometrical objects which are of fundamental importance for the definition of equation \eqref{eq: hypoelliptic SDE} as well as our study on density estimates.

\begin{defn}\label{def:truncated-algebra}
For $l\in\mathbb{N}$, the \textit{truncated tensor algebra} $T^{(l)}$ of order $l$ is defined by 
\begin{equation*}
T^{(l)}=\bigoplus_{n=0}^{l}(\mathbb{R}^{d})^{\otimes n},
\end{equation*}
with the convention $(\mathbb{R}^{d})^{\otimes 0}=\mathbb{R}$. The set $T^{(l)}$ is equipped with a straightforward
vector space structure, plus an operation $\otimes$ defined by
\beq\label{eq:def-product-in-T-l}
\lc g\otimes h\rc^{n}=\sum_{k=0}^{n}g^{n-k}\otimes h^{k},\qquad g,h\in
T^{(l)},
\eeq
where $g^{n}$ designates the projection onto the $n$-th degree component of $g$ for $n\le l$. 
\end{defn}

\noindent
Notice that $T^{(l)}$ should be denoted $T^{(l)}(\R^{d})$. We have dropped the dependence on $\R^{d}$ for notational simplicity. Also observe that with Definition \ref{def:truncated-algebra} in hand,
$(T^{(l)},+,\otimes)$ is an associative algebra with unit element $\mathbf{1} \in (\mathbb{R}^{d})^{\otimes 0}$. The polynomial terms in the expansions which will be considered later on are contained in a subspace of $T^{(l)}$ that we proceed to define now.

\begin{defn}\label{def:free-algebra}
The \textit{free nilpotent Lie algebra} $\mathfrak{g}^{(l)}$ of order $l$ is defined to be the graded sum \[
\mathfrak{g}^{(l)}\triangleq\bigoplus_{k=1}^{l}{\cal L}_{k}\subseteq T^{(l)}.
\]Here $\mathcal{L}_k$ is the space of homogeneous Lie polynomials of degree $k$ defined inductively by $\mathcal{L}_1\triangleq\mathbb{R}^d$ and $\mathcal{L}_k\triangleq [\mathbb{R}^d,\mathcal{L}_{k-1}]$, where the Lie bracket is defined to be the commutator of the tensor product. \end{defn}

We now define some groups related to the algebras given in Definitions \ref{def:truncated-algebra} and \ref{def:free-algebra}. To this aim, introduce the subspace $T_0^{(l)}\subseteq T^{(l)}$ of tensors whose scalar component is zero and recall that $\mathbf{1}\triangleq(1,0,\ldots,0)$. For $u\in T_0^{(l)}$, one can define the inverse $(\mathbf{1}+u)^{-1}$, the exponential $\exp(u)$ and the logarithm $\log(\mathbf{1}+u)$ in $T^{(l)}$ by using the standard Taylor expansion formula with respect to the tensor product. For instance, 
\beq\label{eq:def-exp-on-T-l}
\exp({a})\triangleq\sum_{k=0}^{\infty}\frac{1}{k!} \, {a}^{\otimes k}\in T^{(l)},
\eeq
where the sum is indeed locally finite and hence well-defined. We can now introduce the following group.
\begin{defn}\label{def:free-group}
The \textit{free nilpotent Lie group} $G^{(l)}$ of order $l$ is defined by 
$$G^{(l)}\triangleq\exp(\mathfrak{g}^{(l)})\subseteq T^{(l)}.
$$ 
The exponential function is a diffeomorphism under which $\mathfrak{g}^{(l)}$ in Definition \ref{def:free-algebra} is the Lie algebra of $G^{(l)}$.
\end{defn}

\begin{rem}
One can also include the case of $l=\infty$ in the definitions of $T^{(l)}$, $\mathfrak{g}^{(l)}$ and $G^{(l)}$. But in this case we need to be careful that the direct sums should all be understood as formal series instead of polynomials.
Also all the spaces we have mentioned are defined over the given underlying vector space (which is $\mathbb{R}^d$ in our case), and recall that we have omitted such dependence in the notation for simplicity. 
\end{rem}

It will be useful in the sequel to have some basis available for the algebras introduced above. We shall resort to the following families:
for each word $\alpha=(i_{1},\ldots,i_{r})\in\ca_{1}(l)$, set
\beq\label{eq:def-basis-algebras}
{\rm e}_{(\alpha)}\triangleq{\rm e}_{i_{1}}\otimes\cdots\otimes{\rm e}_{i_{r}},
\quad\text{and}\quad
{\rm e}_{[\alpha]}\triangleq[{\rm e}_{i_{1}},\cdots[{\rm e}_{i_{r-2}},[{\rm e}_{i_{r-1}},{\rm e}_{i_{r}}]]],
\eeq
where $\{{\rm e}_{1},\ldots,{\rm e}_{d}\}$ denotes the canonical basis of $\mathbb{R}^{d}
$. 
Then it can be shown that $\{{\rm e}_{(\alpha)}:\alpha\in{\cal A}(l)\}$ is the canonical
basis of $T^{(l)},$ and we also have $\mathfrak{g}^{(l)}={\rm Span}\{{\rm e}_{[\alpha]}:\alpha\in{\cal A}_{1}(l)\}$.

As a closed subspace, $\mathfrak{g}^{(l)}$ induces a canonical Hilbert structure from $T^{(l)}$ which makes it into a flat Riemannian manifold. The associated volume measure $du$ (the Lebesgue measure) on $\mathfrak{g}^{(l)}$ is left invariant with respect to the product induced from the group structure on $G^{(l)}$ through the exponential diffeomorphism. In addition, for each $\lambda>0$, there is a dilation operation $\delta_\lambda: T^{(l)}\rightarrow T^{(l)}$ induced by $\delta_\lambda(a)\triangleq \lambda^k a$ if $a\in(\mathbb{R}^d)^{\otimes k}$, which satisfies the relation 
$\delta_\lambda\circ\exp=\exp\circ\,\delta_\lambda$ when restricted on $\mathfrak{g}^{(l)}$. Thanks to the fact that $\delta_\lambda(a)= \lambda^k a$ for any $a\in(\R^{d})^{\otimes k}$,
one can easily show that 
\begin{equation}\label{eq:dilation-lebesgue-on-cal-G}
du\circ\delta_\lambda^{-1}=\lambda^{-\nu}du,
\quad\text{where}\quad
\nu\triangleq\sum_{k=1}^{l}k\dim (\mathcal{L}_k).
\end{equation}

We always fix the Euclidean norm on $\mathbb{R}^d$ in the remainder of the paper. As far as the free nilpotent group $G^{(l)}$ is concerned, there are several useful metric structures. Among them we will use an extrinsic metric $\rho_{\textsc{HS}}$ which can be defined easily due to the fact that $G^{(l)}$ is a subspace of $T^{(l)}$. Namely for $g_{1},g_{2}\in G^{(l)}$ we set:
\beq\label{eq:def-HS-distance}
\rho_{\textsc{HS}}(g_{1},g_{2})\triangleq\|g_{2}-g_{1}\|_{\textsc{HS}},\ \ \ g_{1},g_{2}\in G^{(l)},
\eeq
where the right hand side is induced from the Hilbert-Schmidt norm on $T^{(l)}$.

\subsection{Path signatures and the fractional Brownian rough path.}\label{sec:signature}

The stochastic differential equation \eqref{eq: hypoelliptic SDE} governed by a fractional Brownian motion $B$ is standardly solved in the rough paths sense. In this section we recall some basic facts about this notion of solution. We will also give some elements of rough paths expansions, which are at the heart of our methodology in order to obtain lower bounds for the density.

The link between free nilpotent groups and noisy equations like \eqref{eq: hypoelliptic SDE} is made through the notion of signature. Recall that a continuous map
$\mathbf{x}:\{(s,t)\in[0,1]^{2}: \, s\le t\}\rightarrow T^{(l)}$ is called a
\textit{multiplicative functional} if for $s<u<t$ one has $\mathbf{x}_{s,t}
=\mathbf{x}_{s,u}\otimes\mathbf{x}_{u,t}$. A particular occurrence of this kind of map is given when one considers a path $w$ with finite
variation and sets for $s\le t$,
\begin{equation}\label{eq:def-iterated-intg}
\mathbf{w}_{s,t}^{n}=
 \int_{s<u_{1}<\cdots<u_{n}<t}dw_{u_{1}}\otimes\cdots\otimes dw_{u_{n}} .
\end{equation}
Then the so-called \textit{truncated signature path} of order $l$ associated with $w$ is defined by the following object:
\begin{equation}\label{eq:signature-smooth-x}
S_{l}(w)_{\cdot,\cdot}: \{(s,t)\in[0,1]^{2}; \, s\le t\} \rightarrow T^{(l)} ,
\qquad
(s,t)\mapsto
S_{l}(w)_{s,t}:=1+\sum_{n=1}^{l}\mathbf{w}_{s,t}^{n}.
\end{equation}
It can be shown that the functional $S_{l}(w)_{\cdot,\cdot}$ is multiplicative and takes values in the free nilpotent group $G^{(l)}$. The \textit{truncated signature} of order $l$ for $w$ is the tensor element $S_l(w)_{0,1}\in G^{(l)}$. It is simply denoted as $S_l(w)$.

A rough path can be seen as a generalization of the truncated signature path \eqref{eq:signature-smooth-x} to the non-smooth situation. Specifically, the definition of H\"older rough paths can be summarized as follows.
 \begin{defn}\label{def:RP}
Let $\ga\in(0,1)$. The space of \textit{weakly geometric $\ga$-H\"older rough paths} 
is the set of multiplicative paths $\mathbf{x}:\{(s,t)\in[0,1]^{2}; \, s\le t\}\rightarrow G^{[1/\ga]}$ such that the following norm is finite:
\begin{equation}
\|\mathbf{x}\|_{\ga;\textsc{HS}}
= 
\sup_{0\le s < t \le 1} \frac{\|\mathbf{x}_{s,t}\|_{\textsc{HS}}}{|t-s|^{\ga}} .
\end{equation}
\end{defn}
An important subclass of weakly geometric $\ga$-H\"older rough paths is the set of \textit{geometric} $\ga$-H\"older rough paths. These are multiplicative paths $\mathbf{x}$ with values in $G^{\lfloor
1/\ga\rfloor}$ such that $\|\mathbf{x}\|_{\ga;\textsc{HS}}$ is finite and such that there exists a sequence 
$\{  x_{\ep}; \, \ep>0 \} $ with $x_{\ep}\in C^{\infty}([0,T];\mathbb{R}^{d})$ satisfying 
\begin{equation}\label{eq:closure-smooth-path}
\lim_{\ep\to 0} 
\|\mathbf{x} - S_{[1/\ga]}(x_{\ep})\|_{\ga;\textsc{HS}}
 =0.
\end{equation}


The notion of signature allows to define a more intrinsic distance (with respect to the HS-distance given by \eqref{eq:def-HS-distance}) on the free group $G^{(l)}$. This metric is known as the \textit{Carnot-Caratheodory metric} and given by 
\[
\rho_{\textsc{CC}}(g_{1},g_{2})\triangleq\|g_{1}^{-1}\otimes g_{2}\|_{\textsc{CC}},\ \ \ g_{1},g_{2}\in G^{(l)},
\]
where the CC-norm $\|\cdot\|_{\textsc{CC}}$ is defined by  
\beq\label{eq:def-norm-CC}
\|g\|_{\textsc{CC}}\triangleq\inf\left\{ \|w\|_{1-{\rm var}}:w\in C^{1-{\rm var}}([0,1];\mathbb{R}^{d})\ {\rm and}\ S_{l}(w)=g\right\}. 
\eeq
It can be shown that the infimum in \eqref{eq:def-norm-CC} is attainable. 

\begin{rem}\label{connecting by smooth path}
It is well-known that for any $g\in G^{(l)}$, one can find a piecewise linear path $w$ such that $S_l(w)=g$ (cf. \cite{FV10} for instance). Moreover, one can do better, and find a smooth path $w$ whose derivative is compactly supported such that $S_l(w)=g$. Indeed, this could be achieved by (i) reparametrizing the piecewise linear path so that the resulting path is smooth but the trajectory itself is still the same piecewise linear one; and (ii) adding trivial pieces to the beginning and the end of the path. This will not change the truncated signature as it is invariant under reparametrization. 
\end{rem}

The HS and CC metrics are equivalent as seen from the following so-called \textit{ball-box estimate}.
\begin{prop}\label{prop: ball-box estimate}
Let $\rho_{\textsc{HS}}$ and $\rho_{\textsc{CC}}$ be the distances on $G^{(l)}$ respectively defined by \eqref{eq:def-HS-distance} and \eqref{eq:def-norm-CC}.
For each $l\geq1$, there exists a constant $C=C_{l}>0$, such that 
\begin{equation}\label{eq:bnd-CC-to-HS}
\rho_{\textsc{CC}}(g_{1},g_{2})\leq C\max\left\{ \rho_{\textsc{HS}}(g_{1},g_{2}),\rho_{\textsc{HS}}(g_{1},g_{2})^{\frac{1}{l}}\cdot\max\left\{ 1,\|g_{1}\|_{\textsc{CC}}^{1-\frac{1}{l}}\right\} \right\} 
\end{equation}
and
\[
\rho_{\textsc{HS}}(g_{1},g_{2})\leq C\max\left\{ \rho_{\textsc{CC}}(g_{1},g_{2})^{l},\rho_{\textsc{CC}}(g_{1},g_{2})\cdot\max\left\{ 1,\|g_{1}\|_{\textsc{CC}}^{l-1}\right\} \right\} 
\]
for all $g_{1},g_{2}\in G^{(l)}$. In particular, 
\[
\|g\|_{\textsc{CC}}\leq1\implies\|g-{\bf 1}\|_{\textsc{HS}}\leq C\|g\|_{\textsc{CC}}
\]
and
\[
\|g-{\bf 1}\|_{\textsc{HS}}\leq1\implies\|g\|_{\textsc{CC}}\leq C\|g-{\bf 1}\|_{\textsc{HS}}^{\frac{1}{l}}.
\]
\end{prop}

One of the main application of the abstract rough path theory is the ability to extend most stochastic calculus tools to a large class of Gaussian processes. The following result, borrowed from \cite{CQ,FV06}, establishes this link for fractional Brownian motion. 

\begin{prop}\label{prop:fbm-rough-path}
Let $B$ be a fractional Brownian motion with Hurst parameter  $H>1/4$. Then $B$ admits a lift $\mathbf{B}$ as a geometric rough path of order $[1/\ga]$ for any $\ga<H$.
\end{prop}

Let us now turn to the definition of rough differential equations. There are several equivalent ways to introduce this notion, among which we will choose to work with Taylor type expansions, since they are more consistent with our later developments. To this aim, let us first consider a bounded variation path $w$ and the following ordinary differential equation driven by $w$:
\begin{equation}\label{eq: ode}
dx_t=\sum_{\alpha=1}^d V_\alpha(x_t) \, dw_t^\alpha,
\end{equation}
where the $V_\alpha$'s are $C_b^\infty$ vector fields.
For any given word $\alpha=(i_1,\ldots,i_r)$ over the letters $\{1,\ldots,d\}$, we define the vector field $V_{(\alpha)}\triangleq(V_{i_{1}}\cdots(V_{i_{r-2}}(V_{i_{r-1}}V_{i_{r}})))$, 
where we have identified a vector field with a differential operator, so that $V_iV_j$ means differentiating $V_j$ along direction $V_i$. 
Classically, a \textit{formal Taylor expansion} of the solution $x_t$ to \eqref{eq: ode} is then given by
\begin{equation}\label{eq: formal Taylor expansion}
x_{s,t}\sim\sum_{k=1}^{\infty}\sum_{i_{1},\ldots,i_{k}=1}^{d}V_{(i_{1},\ldots,i_{k})}(x_{s})
\int_{s<u_{1}<\cdots<u_{k}<t}dw_{u_{1}}^{i_{1}}\cdots dw_{u_{k}}^{i_{k}},
\end{equation}
where we have set $x_{s,t}=x_{t}-x_{s}$.  
This expansion can be rephrased in more geometrical terms. Specifically, we define the following Taylor approximation function on $\mathfrak{g}^{(l)}$.
\begin{defn}\label{def: Taylor approximation function}
Let $\{V_\alpha; \, 1\le \al\le d\}$ be a family of $C_b^\infty$ vector fields on $\R^{N}$, and recall that the sets of words $\ca(l),\ca_{1}(l)$ are introduced at the beginning of Section \ref{sec: main results}.
For each $l\geq1$, we define the \textit{Taylor approximation function} $F_l: \mathfrak{g}^{(l)}\times\mathbb{R}^N\rightarrow\mathbb{R}^N$ of order $l$ associated with the ODE (\ref{eq: ode}) by
\[
F_{l}(u,x)\triangleq\sum_{\alpha\in{\cal A}_{1}(l)}V_{(\alpha)}(x)\cdot (\exp u)^\alpha,\ \ \ (u,x)\in\mathfrak{g}^{(l)}\times\mathbb{R}^{N},
\]
where the exponential function is defined on $T^{(l)}$ by \eqref{eq:def-exp-on-T-l} and $(\exp(u))^\alpha$ is the coefficient of $\exp(u)$ with respect to the tensor basis element $\rm e_{(\alpha)}$ (we recall that the notation $\rm e_{(\alpha)}$ is introduced in \eqref{eq:def-basis-algebras}). We also say that $u\in\mathfrak{g}^{(l)}$ \textit{joins} $x$ \textit{to} $y$ \textit{in the sense of Taylor approximation} if $y=x+F_l(u,x)$.
\end{defn}

With Definition \ref{def: Taylor approximation function} in hand, we can recast the formal expansion \eqref{eq: formal Taylor expansion} (truncated at an arbitrary degree $l$) in the following way:
\begin{equation}\label{eq: formal Taylor expansion 2}
x_{s,t}\sim F_{l}\lp \log\lp S(w)_{s,t} \rp,  x_{s}  \rp ,
\end{equation}
where the function $\log$ is the inverse of the exponential map for $G^{(l)}$ which can also be defined by a truncated Taylor's formula on $G^{(l)}$ similar to the exponential, and $S(w)_{s,t}$ is the truncated signature path of $w$ defined by \eqref{eq:signature-smooth-x}. In order to define rough differential equations, a natural idea is to extend this approximation scheme to rough paths. We get a definition which is stated below in the fractional Brownian motion case.

\begin{defn}\label{def:solution-rde}
Let $B$ be a fractional Brownian motion with Hurst parameter  $H>1/4$, and consider its rough path lift $\mathbf{B}$ as in Proposition \ref{prop:fbm-rough-path}. Let $\{V_\alpha; \, 1\le \al\le d\}$ be a family of $C_b^\infty$ vector fields on $\R^{N}$. We say that $X$ is a \textit{solution to the rough differential equation}~\eqref{eq: hypoelliptic SDE} if for all $(s,t)\in[0,1]^{2}$ such that $s<t$ we have
\begin{equation}\label{eq:def-rough-eq}
X_{s,t}
=
F_{[1/\gamma]-1}\lp \log\lp S(\mathbf{B})_{s,t} \rp,  X_{s}  \rp + R_{s,t} \, ,
\end{equation}
where $R_{s,t}$ is an $\R^{N}$-valued remainder such that there exists $\ep>0$ satisfying
\begin{equation*}
\sup_{0\le s < t \le 1} \frac{|R_{s,t}|}{|t-s|^{1+\ep}} < \infty.
\end{equation*}
\end{defn}

Roughly speaking, Definition \ref{def:solution-rde} says that the expansion of the solution $X$ to a rough differential equation should coincide with \eqref{eq: formal Taylor expansion} up to a remainder with H\"older regularity greater than 1. This approach goes back to Davie \cite{Da07}, and it can be shown to coincide with more classical notions of solutions.
We close this section by recalling an existence and uniqueness result which is fundamental in rough path theory. 

\begin{prop}
Under the same conditions as in Definition \ref{def:solution-rde}, there exists a unique solution to equation \eqref{eq: hypoelliptic SDE} considered in the sense of (\ref{eq:def-rough-eq}).
\end{prop}

\subsection{Malliavin calculus for fractional Brownian motion.}

In this section we review some basic aspects of Malliavin calculus. The reader is referred to~\cite{Nualart06} for further details.

We consider the fractional Brownian motion $B=(B^1,\ldots,B^d)$ as in Definition \eqref{def:fbm}, defined on a complete probability space $(\Omega, \cf, \mathbb{P})$. For sake of simplicity, we assume that $\cf$ is generated by $\{B_{t}; \, t\in[0,T]\}$. An $\mathcal{F}$-measurable real
valued random variable $F$ is said to be \textit{cylindrical} if it can be
written, with some $m\ge 1$, as
\begin{equation*}
F=f\lp  B_{t_1},\ldots,B_{t_m}\rp,
\quad\mbox{for}\quad
0\le t_1<\cdots<t_m \le 1,
\end{equation*}
where $f:\mathbb{R}^m \rightarrow \mathbb{R}$ is a $C_b^{\infty}$ function. The set of cylindrical random variables is denoted by~$\mathcal{S}$. 

\smallskip

The Malliavin derivative is defined as follows: for $F \in \mathcal{S}$, the derivative of $F$ in the direction $h\in\ch$ is given by
\[
\mathbf{D}_h F=\sum_{i=1}^{m}  \frac{\partial f}{\partial
x_i} \left( B_{t_1},\ldots,B_{t_m}  \right) \, h_{t_i}.
\]
More generally, we can introduce iterated derivatives. Namely, if $F \in
\mathcal{S}$, we set
\[
\mathbf{D}^k_{h_1,\ldots,h_k} F = \mathbf{D}_{h_1} \ldots\mathbf{D}_{h_k} F.
\]
For any $p \geq 1$, it can be checked that the operator $\mathbf{D}^k$ is closable from
$\mathcal{S}$ into $\mathbf{L}^p(\oom;\ch^{\otimes k})$. We denote by
$\mathbb{D}^{k,p}(\ch)$ the closure of the class of
cylindrical random variables with respect to the norm
\begin{equation} \label{eq:malliavin-sobolev-norm}
\left\| F\right\| _{k,p}=\left( \mathbb{E}\left[|F|^{p}\right]
+\sum_{j=1}^k \mathbb{E}\left[ \left\| \mathbf{D}^j F\right\|
_{\ch^{\otimes j}}^{p}\right] \right) ^{\frac{1}{p}},
\end{equation}
and we also set $\mathbb{D}^{\infty}(\ch)=\cap_{p \geq 1} \cap_{k\geq 1} \mathbb{D}^{k,p}(\ch)$. 

\smallskip

Estimates of Malliavin derivatives are crucial in order to get information about densities of random variables, and Malliavin covariance matrices as well as non-degenerate random variables will feature importantly in the sequel.
\begin{defn}\label{non-deg}
Let $F=(F^1,\ldots , F^n)$ be a random vector whose components are in $\mathbb{D}^\infty(\ch)$. Define the \textit{Malliavin covariance matrix} of $F$ by
\begin{equation} \label{malmat}
\gamma_F=(\langle \mathbf{D}F^i, \mathbf{D}F^j\rangle_{\ch})_{1\leq i,j\leq n}.
\end{equation}
Then $F$ is called  {\it non-degenerate} if $\gamma_F$ is invertible $a.s.$ and
$$(\det \gamma_F)^{-1}\in \cap_{p\geq1}L^p(\Omega).$$
\end{defn}
It is a classical result that the law of a non-degenerate random vector $F=(F^1, \ldots , F^n)$ admits a smooth density with respect to the Lebesgue measure on $\mr^n$. 

\section{The elliptic case.}\label{sec: ellip.}

In this section, for a better understanding of our overall strategy, we first prove Theorem~\ref{thm: local comparison} and Theorem~\ref{thm: local lower estimate} in the uniformly elliptic case. The analysis in this case is more explicit and straightforward, and our methodology might be more apparent. More precisely we still consider the SDE (\ref{eq: hypoelliptic SDE}), and we assume that the coefficients satisfy the following hypothesis:
\begin{uell*}
The $\cac_{b}^{\infty}$ vector fields $V=\{V_1,\ldots,V_d\}$ are such that
\beq\label{eq:hyp-elliptic}
\Lambda_{1}|\xi|^{2}\leq\xi^{*}V(x)V(x)^{*}\xi\leq\Lambda_{2}|\xi|^{2},\ \ \ \forall x,\xi\in\mathbb{R}^{N},
\eeq
with some constants $\Lambda_1,\Lambda_2>0$, where $(\cdot)^*$ denotes matrix transpose. 
\end{uell*}

\noindent
Notice that the condition \eqref{eq:hyp-elliptic} can be seen as a special case of the uniform hypoellipticity condition \eqref{eq:unif-hypo-assumption}, where $l_0=1$.

One of the major simplifications of the elliptic (vs. hypoelliptic) situation concerns the control distance $d$. Indeed, recall that for $x,y\in\mathbb{R}^N$, $\Pi_{x,y}$ is the set of Cameron-Martin paths that join $x$ to $y$ in the sense of differential equation (cf. (\ref{eq: Pi_xy})). Under our assumption \eqref{eq:hyp-elliptic} it is easy to construct an $h\in\bar{\mathcal{H}}\in\Pi_{x,y}$ explicitly, which will ease our computations later on.

\begin{lem}\label{lem: explicit construction of joining h}
Let $V=\{V_1,\ldots,V_d\}$ be vector fields satisfying the uniform elliptic assumption~\eqref{eq:hyp-elliptic}. Given $x,y\in\mathbb{R}^N$, define 
\beq\label{eq:def-ht-elliptic-case}
h_{t}\triangleq\int_{0}^{t}V^{*}(z_{s})\cdot(V(z_{s})V^{*}(z_{s}))^{-1}\cdot(y-x)ds ,
\eeq
where $z_{t}\triangleq(1-t)x+ty$ is the line segment from $x$ to $y$. Then $h\in\Pi_{x,y}$, where $\Pi_{x,y}$ is defined by relation \eqref{eq: Pi_xy}. 
\end{lem}

\begin{proof}
Since $\bar{\mathcal{H}}=I_{0^+}^{H+1/2}(L^2([0,1]))$ contains smooth paths, it is obvious that $h\in\bar{\mathcal{H}}$. As far as $z_{t}$ is concerned, the definition $z_{t}=(1-t)x+ty$ clearly implies that $z_0=x,z_1=y$ and $\dot{z}_{t}=y-x$. In addition, since $VV^{*}(\xi)$ is invertible for all $\xi\in\R^{N}$ under our condition \eqref{eq:hyp-elliptic}, we get
\[
\dot{z}_{t}=y-x=\left(VV^{*}(VV^{*})^{-1}\right)(z_{t})\cdot(y-x)=V(z_{t})\dot{h}_{t},
\] 
where the last identity stems from the definition \eqref{eq:def-ht-elliptic-case} of $h$.
Therefore $h\in\Pi_{x,y}$ according to our definition \eqref{eq: Pi_xy}.
\end{proof}

\begin{rem}
The intuition behind Lemma \ref{lem: explicit construction of joining h} is very simple. Indeed, given any smooth path $x_t$ with $x_0=x,x_1=y$, since the vector fields are elliptic, there exist smooth functions $\lambda^1(t),\ldots,\lambda^d (t)$, such that\[
\dot{x}_{t}=\sum_{\alpha=1}^{d}\lambda^{\alpha}(t)V_{\alpha}(x_{t}),\ \ \ 0\leq t\leq1.
\]In matrix notation, $\dot{x}_t=V(x_t)\cdot\lambda(t)$. A canonical way to construct $\lambda(t)$ is writing it as $\lambda(t)=V^*(x_t)\eta(t)$ so that from ellipticity we can solve for $\eta(t)$ as $$\eta(t)=(V(x_{t})V^{*}(x_{t}))^{-1}\dot{x}_{t}.$$
It follows that the path $h_t\triangleq\int_0^t \lambda(s)ds$ belongs to $\Pi_{x,y}$.
\end{rem}

Now we can prove the following result which asserts that the control distance function is locally comparable with the Euclidean metric, that is  Theorem \ref{thm: local lower estimate} under elliptic assumptions.

\begin{thm}\label{thm: local comparison in elliptic case}
Let $V=\{V_1,\ldots,V_d\}$ be vector fields satisfying the uniform elliptic assumption \eqref{eq:hyp-elliptic}. Consider  the control distance $d=d_{H}$ given in Definition \ref{def: CM bridge} for a given $H>\frac14$.
Then there exist constants $C_1,C_2>0$ depending only on $H$ and the vector fields, such that 
\beq\label{eq:local-euclid-d-elliptic}
C_{1}|x-y|\leq d(x,y)\leq C_{2}|x-y|
\eeq
for all $x,y\in\mathbb{R}^N$ with $|x-y|\leq1.$
\end{thm}

\begin{proof}

We first consider the case when $H\leq1/2$, which is simpler due to Lemma \ref{lem: continuous embedding when H<1/2}. Given $x,y\in\mathbb{R}^N$, define $h\in\Pi_{x,y}$ as in Lemma \ref{lem: explicit construction of joining h}. According to Lemma \ref{lem: continuous embedding when H<1/2} 
and Definition \ref{def: CM bridge} we have
\[
d(x,y)^{2}\leq\|h\|_{\bar{\mathcal{H}}}^{2}\leq C_{H}\|h\|_{W^{1,2}}^{2}.\]
Therefore, according to the definition (\ref{eq:def-ht-elliptic-case}) of $h$, we get
\[
d(x,y)^{2}\leq C_{H}\int_{0}^{1}|V^*(z_s)(V(z_{s})V^*(z_s))^{-1}\cdot(y-x)|^{2}ds\leq C_{H,V}|y-x|^{2},
\]
where the last inequality stems from the uniform ellipticity assumption \eqref{eq:hyp-elliptic} and the fact that $V^*$ is bounded. This proves the upper bound in \eqref{eq:local-euclid-d-elliptic}.

We now turn to the lower bound in \eqref{eq:local-euclid-d-elliptic}. To this aim, consider $h\in\Pi_{x,y}$. We assume (without loss of generality) in the sequel that
\begin{equation}\label{b1}
\|h\|_{\bar{\mathcal{H}}}
\leq
2 d(x,y)
\leq
2 C_{2},
\end{equation}
where the last inequality is due to the second part of inequality \eqref{eq:local-euclid-d-elliptic} and the fact that $|x-y|\leq1$.
Then recalling the definition \eqref{eq: Pi_xy} of $\Pi_{x,y}$ we have 
\[
y-x=\int_{0}^{1}V(\Phi_{t}(x;h))dh_{t}.
\]
According to Proposition \ref{prop: variational embedding} (specifically the embedding $\bar{\mathcal{H}}\subseteq C_{0}^{q-\mathrm{var}}([0,1];\mathbb{R}^{d})$ for $q>(H+1/2)^{-1}$) and the pathwise variational estimate given by \cite[Theorem 10.14]{FV10},   we have 
\begin{equation}\label{b2}
|y-x|  \leq C_{H,V} \lp \|h\|_{q-{\rm var}}\vee\|h\|_{q-{\rm var}}^{q} \rp
  \leq C_{H,V} \lp \|h\|_{\bar{\mathcal{H}}}\vee\|h\|_{\bar{\mathcal{H}}}^{q} \rp.
\end{equation}
Since $q\geq1$ and owing to \eqref{b1}, we conclude that 
\[
|y-x|\leq C_{H,V}\|h\|_{\bar{\mathcal{H}}}
\] for all $x,y$ with $|y-x|\leq1$. Since $h\in\Pi_{x,y}$ is arbitrary provided \eqref{b1} holds true,  the lower bound in \eqref{eq:local-euclid-d-elliptic} follows again by a direct application of Definition \ref{def: CM bridge}.

Next we consider the case when $H>1/2$. The lower bound in \eqref{eq:local-euclid-d-elliptic} can be proved with the same argument as in the case $H\leq1/2$, 
the only difference being that in \eqref{b2} we replace $\bar{\mathcal{H}}\subseteq C_{0}^{q-\mathrm{var}}([0,1];\mathbb{R}^{d})$ by $\bar{\mathcal{H}}\subseteq C_{0}^{H}([0,1];\mathbb{R}^{d})$ and the pathwise variational estimate of \cite[Theorem 10.14]{FV10} by a
H\"older estimate borrowed from  \cite[Proposition 8.1]{FH14}. 

For the upper bound in \eqref{eq:local-euclid-d-elliptic}, we again take $h\in\Pi_{x,y}$ as given by Lemma \ref{lem: explicit construction of joining h} and estimate its Cameron-Martin norm. Note that due to our uniform ellipticity assumption \eqref{eq:hyp-elliptic}, one can define the function 
\begin{align}\label{eq: def gamma}
\gamma_t\equiv\int_0^t (V^*(VV^*)^{-1})(z_s)ds=\int_0^t g((1-s)x+sy)ds,\end{align} 
where $g$ is a matrix-valued $C_b^\infty$ function. We will now prove that $\gamma$ can be written as $\gamma=K\varphi$ for $\varphi\in L^2([0,1])$. Indeed, one can solve for $\varphi$ in the analytic expression (\ref{eq: analytic expression of K}) for $H>1/2$ and get 
\begin{align*}
\varphi(t) & =C_{H}t^{H-\frac{1}{2}}\left(D_{0^{+}}^{H-\frac{1}{2}}\left(s^{\frac{1}{2}-H}\dot{\gamma}_{s}\right)\right)(t).
\end{align*}
We now use the expression \eqref{eq:def-frac-deriv} for $D^{H-1/2}_{0^+}$, which yield (after an elementary change of variable)
\begin{align*}
 \varphi(t)& =C_{H}t^{H-\frac{1}{2}}\frac{d}{dt}\int_{0}^{t}s^{\frac{1}{2}-H}(t-s)^{\frac{1}{2}-H}g((1-s)x+sy)ds\\
 & =C_{H}t^{H-\frac{1}{2}}\frac{d}{dt}\left(t^{2-2H}\int_{0}^{1}(u(1-u))^{\frac{1}{2}-H}g((1-tu)x+tuy)du\right)\\
 & =C_{H}t^{\frac{1}{2}-H}\int_{0}^{1}(u(1-u))^{\frac{1}{2}-H}g((1-tu)x+tuy)du\\
 & \ \ \ +C_{H}t^{\frac{3}{2}-H}\int_{0}^{1}(u(1-u))^{\frac{1}{2}-H}u\nabla g((1-tu)x+tuy)\cdot(y-x)du.
\end{align*}
Hence, thanks to the fact that $g$ and $\nabla g$ are bounded plus the fact that $t\leq1$, we get 
\[
|\varphi(t)|\leq C_{H,V}(t^{\frac{1}{2}-H}+|y-x|),
\]from which $\varphi$ is clearly an element of $L^2([0,1])$. Since $|y-x|\leq 1$, we conclude that \[
\|\gamma\|_{\bar{\mathcal{H}}}=\|\varphi\|_{L^{2}([0,1])}\leq C_{H,V}.
\] Therefore, recalling that $h$ is given by \eqref{eq:def-ht-elliptic-case} and $\gamma$ is defined by \eqref{eq: def gamma}, we end up with
\begin{align*}
d(x,y)\leq\|h\|_{\bar{\mathcal{H}}}&=\left\|\left(\int_{0}^{\cdot}(V^*(VV^*)^{-1})(z_{s})ds\right)\cdot(y-x)\right\|_{\bar{\mathcal{H}}}\\
&=\|\gamma\|_{\bar{\mathcal{H}}}|y-x|\leq C_{H,V}|y-x|.
\end{align*}
This concludes the proof.
\end{proof}


From Theorem \ref{thm: local comparison in elliptic case}, we know that $|B_d(x,t^H)|\asymp t^{NH}$ when $t$ is small. Therefore, the elliptic version of Theorem \ref{thm: local lower estimate} becomes the following result, which is consistent with the intuition that the density $p(t,x,y)$ of the solution to equation \eqref{eq: hypoelliptic SDE} should behave like the Gaussian kernel \[
p(t,x,y)\asymp\frac{C_{1}}{t^{NH}}\exp\left(-\frac{C_{2}|y-x|^{2}}{t^{2H}}\right).
\]

\begin{thm}\label{thm: local lower estimate in elliptic case}
Let $p(t,x,y)$ be the density of the solution $X_t$ to equation \eqref{eq: hypoelliptic SDE}. Under the uniform ellipticity assumption (\ref{eq:hyp-elliptic}), there exist constants $C_1,C_2,\tau>0$ depending only on $H$ and the vector fields $V$, such that \begin{align}\label{local lower estimate-elliptic}p(t,x,y)\geq\frac{C_1}{t^{NH}}\end{align} for all $(t,x,y)\in(0,1]\times\mathbb{R}^N\times\mathbb{R}^N$ satisfying $|x-y|\leq C_2 t^H$ and $t<\tau$.
\end{thm}

The main idea behind the proof of Theorem \ref{thm: local lower estimate in elliptic case} is to translate the small time  estimate in \eqref{local lower estimate-elliptic} into a large deviation estimate. To this aim, we will first recall some preliminary notions taken from \cite{BOZ15}. By a slight abuse of notation with respect to \eqref{eq: skeleton ODE}, we will call  $w\mapsto\Phi_t(x;w)$ the solution map of the SDE (\ref{eq: hypoelliptic SDE}) (or \eqref{eq:def-rough-eq}). From the scaling invariance of fractional Brownian motion, it is not hard to see that \begin{align}\label{scaling property of equation driven by fBm}
\Phi_{t}(x;B)\stackrel{{\rm law}}{=}\Phi_{1}(x;\varepsilon B),
\end{align} where $\varepsilon\triangleq t^H$. Therefore, since the random variable $\Phi_{t}(x;B)$ is nondegenerate under our standing assumption~\eqref{eq:hyp-elliptic}, the density $p(t,x,y)$ can be written as
\begin{equation}\label{b3}
p(t,x,y)=\mathbb{E}\left[\delta_{y}\left(\Phi_{1}(x;\varepsilon B)\right)\right].
\end{equation}

Starting from expression \eqref{b3}, we now label a proposition which gives a lower bound on $p(t,x,y)$ in terms of some conveniently chosen shifts on the Wiener space.
\begin{prop}\label{th: summary of 4}
In this proposition, $\Phi_t$ stands for  the solution map of equation \eqref{eq: hypoelliptic SDE}.  The vector fields $\{V_1,\ldots,V_d\}$ are supposed to satisfy the uniform elliptic assumption~\eqref{eq:hyp-elliptic}. Then the following holds true.\\
(i) Let $\Phi_t$ be the solution map of equation \eqref{eq: hypoelliptic SDE}, $h\in\bar{\mathcal{H}}$, and let
\begin{align}\label{def approx X(h)}
X^{\varepsilon}(h)\triangleq\frac{\Phi_{1}(x;\varepsilon B+h)-\Phi_{1}(x;h)}{\varepsilon}.
\end{align}
Then  $X^\ep(h)$ converges in $\mathbb{D}^\infty$ to X(h),
where $X(h)$ is a $\R^N$-valued centered Gaussian random variable whose covariance matrix will be specified below.\\
(ii) Let $\ep>0$ and consider $x,y\in\R^N$ such that $d(x,y)\leq\ep$, where $d(\cdot,\cdot)$ is the distance considered in Theorem \ref{thm: local comparison in elliptic case}.  Choose $h\in\Pi_{x,y}$ so that 
\begin{equation}\label{eq:local-cdt-on-h}
\|h\|_{\bar{\mathcal{H}}}\leq d(x,y)+\varepsilon\leq2\varepsilon.
\end{equation}
Then we have
\begin{align}\label{eq: CM theorem}
\mathbb{E}\left[\delta_{y}\left(\Phi_{1}(x;\varepsilon B)\right)\right]  \geq C\varepsilon^{-N}\cdot\mathbb{E}\left[\delta_{0}\left(X^\ep(h)\right){\rm e}^{-I\left(\frac{h}{\varepsilon}\right)}\right].
\end{align}
\end{prop}

\begin{proof}
The first statement is proved in \cite{BOZ15}. For the second statement, according to the Cameron-Martin theorem, we have 
\begin{equation*}
\mathbb{E}\left[\delta_{y}\left(\Phi_{1}(x;\varepsilon B)\right)\right]  
=
{\rm e}^{-\frac{\|h\|_{\bar{\mathcal{H}}}^{2}}{2\varepsilon^{2}}}
\mathbb{E}\left[\delta_{y}\left(\Phi_{1}(x;\varepsilon B+h)\right){\rm e}^{-I\left(\frac{h}{\varepsilon}\right)}\right],
\end{equation*}
where we have identified $\bar{\mathcal{H}}$ with $\mathcal{H}$ through $\mathcal{R}$ and recall that $I:\mathcal{H}\rightarrow\mathcal{C}_1$ is the Wiener integral operator introduced in Section \ref{sec: prel.}.
Therefore, thanks to inequality \eqref{eq:local-cdt-on-h}, we get
\begin{equation*}
\mathbb{E}\left[\delta_{y}\left(\Phi_{1}(x;\varepsilon B)\right)\right]  
\geq C\cdot\mathbb{E}\left[\delta_{y}\left(\Phi_{1}(x;\varepsilon B+h)\right){\rm e}^{-I\left(\frac{h}{\varepsilon}\right)}\right].
\end{equation*}
 In addition we have chosen $h\in\Pi_{x,y}$, which means that $\Phi_1(x;h)=y$. Thanks to the scaling property of the Dirac delta function in $\R^N$, we get
 \begin{align*}
p(t,x,y)=\mathbb{E}\left[\delta_{y}\left(\Phi_{1}(x;\varepsilon B)\right)\right]  
&\ge 
C\varepsilon^{-N}\cdot\mathbb{E}\left[\delta_{0}\left(\frac{\Phi_{1}(x;\varepsilon B+h)-\Phi_{1}(x;h)}{\varepsilon}\right){\rm e}^{-I\left(\frac{h}{\varepsilon}\right)}\right].
\end{align*}
Our claim  \eqref{eq: CM theorem} thus follows from the definition \eqref{def approx X(h)} of $X^{\ep}(h)$.
\end{proof}

Let us now describe the covariance matrix of $X(h)$ introduced in Proposition \ref{th: summary of 4}. For this, we recall again that $\Phi$ is the application defined on $\bar{\mathcal{H}}$ by \eqref{eq: skeleton ODE}. The Jacobian of $\Phi_t(\cdot\,;h): \R^N\to\R^N$ is denoted by $J(\cdot\,;h)$. Then it is standard (cf. \cite{BOZ15}) that the deterministic Malliavin differential of $\Phi$ satisfies
\begin{equation}\label{eq:derivative-Phi-with-Jacobian}
\langle D\Phi_{t}(x;h),l\rangle_{\bar{\mathcal{H}}}=J_{t}(x;h)\cdot\int_{0}^{t}J_{s}^{-1}(x;h)\cdot V(\Phi_{s}(x;h))dl_{s},\ \ \ \mathrm{for\ all}\  l\in\bar{\mathcal{H}}, 
\end{equation} 
where $D$ is the Malliavin derivative operator. According to the pairing (\ref{eq: H-barH pairing}), when viewed as an $\mathcal{H}$-valued functional, we have 
\begin{align}\label{eq: expression Malliavin derivative}(D\Phi_{t}^{i}(x;h))_{s}=\left(J_{t}(x;h)J_{s}^{-1}(x;h)V(\Phi_{s}(x;h))\right)^{i}\mathbf{1}_{[0,t]}(s),\ \ \ 1\leq i\leq N.
\end{align}
Then the $N\times N$ covariance matrix of $X(h)$ admits the following representation taken from the reference \cite{BOZ15}:
\begin{align}\label{Malliavin matrix X(h)}
\mathrm{Cov}(X(h))\equiv \Gamma_{\Phi_1(x;h)}=\langle D\Phi_1(x;h), D\Phi_1(x;h)\rangle_\mathcal{H}.
\end{align}


With \eqref{Malliavin matrix X(h)} in hand, a crucial point for proving Theorem \ref{thm: local lower estimate in elliptic case} is the fact that $\Gamma_{\Phi_1(x;h)}$ is uniformly non-degenerate with respect to all $h$. This is the content of the following result which is another special feature of ellipticity that fails in the hypoelliptic case. Its proof is an adaptation of the argument in \cite{BOZ15} to the deterministic context.

\begin{lem}\label{lem: uniform nondegeneracy of Malliavin matrix}
Let $M>0$ be a localizing constant. Consider the Malliavin covariance matrix $\Gamma_{\Phi_1(x;h)}$ defined by \eqref{Malliavin matrix X(h)}. Under the uniform ellipticity assumption (\ref{eq:hyp-elliptic}), there exist $C_1,C_2>0$ depending only on $H,M$ and the vector fields, such that 
\begin{align}\label{uniform bound for Malliavin matrix}C_{1}\leq\det \Gamma_{\Phi_{1}(x;h)}\leq C_{2}
\end{align}for all $x\in\mathbb{R}^N$ and $h\in\bar{\mathcal{H}}$ with $\|h\|_{\bar{\mathcal{H}}}\leq M$.
\end{lem}

\begin{proof}

We consider the cases of $H>1/2$ and $H\leq 1/2$ separately. We only study the lower bound of $\Gamma_{\Phi_1(x;h)}$ since the upper bound is standard from pathwise estimates by \eqref{eq: expression Malliavin derivative} and \eqref{Malliavin matrix X(h)}, plus the fact that $\|h\|_{\mathcal{H}}\leq M$.\\

\noindent
\emph{(i) Proof of the lower bound when  $H>1/2$.}
According to relation \eqref{Malliavin matrix X(h)} and the expression for the inner product in $\mathcal{H}$ given by \cite[equation (5.6)]{Nualart06}, we have \[
\Gamma_{\Phi_{1}(x;h)}=C_{H}\sum_{\alpha=1}^{d}\int_{[0,1]^{2}}J_{1}J_{s}^{-1}V_{\alpha}(\Phi_{s})V_{\alpha}^{*}(\Phi_{t})(J_{t}^{-1})^{*}J_{1}^{*}|t-s|^{2H-2}dsdt,
\]where we have omitted the dependence on $x$ and $h$ for $\Phi$ and $J$ inside the integral for notational simplicity. It follows that for any $z\in\mathbb{R}^N$, we have
\begin{align}\label{eq: inner product H>1/2}z^{*}\Gamma_{\Phi_{1}(x;h)}z=C_{H}\int_{[0,1]^{2}}\langle\xi_{s},\xi_{t}\rangle_{\mathbb{R}^{d}}|t-s|^{2H-2}dsdt,
\end{align}where
$\xi$ is the function in $\mathcal{H}$ defined by
\begin{align}\label{definition xi}\xi_{t}\triangleq V^{*}(\Phi_{t})(J_{t}^{-1})^{*}J_{1}^{*}z.\end{align}
According to an interpolation inequality proved by Baudoin-Hairer (cf. \cite[Proof of Lemma 4.4]{BH07}), given $\gamma>H-1/2$, we have 
\begin{equation}\label{eq: interpolation inequality}
\int_{[0,1]^{2}}\langle f_{s},f_{t}\rangle_{\mathbb{R}^{d}}|t-s|^{2H-2}dsdt\geq C_{\gamma}\frac{\left(\int_{0}^{1}v^{\gamma}(1-v)^{\gamma}|f_{v}|^{2}dv\right)^{2}}{\|f\|_{\gamma}^{2}}
\end{equation}for all $f\in C^\gamma([0,1];\mathbb{R}^d)$. 
Observe that, due to our uniform ellipticity assumption \eqref{eq:hyp-elliptic} and the non-degeneracy of $J_t$, we have
\begin{align}\label{lower bound xi}
\inf_{0\leq t\leq1}|\xi_{t}|^{2}\geq C_{H,V,M}|z|^{2}.
\end{align}
Furthermore, recall that $\Phi_t$ is defined in \eqref{eq: skeleton ODE} and is driven by $h\in\bar{\mathcal{H}}$. We have also seen that $\bar{\mathcal{H}}\hookrightarrow C^H_0$ whenever $H>1/2$. Thus for $H-1/2<\gamma<H$, we get $\|\Phi_t\|_\gamma\leq C_{H,V}\|h\|_\gamma$; and the same inequality holds true for the Jacobian $J$ in \eqref{definition xi}. Therefore, going back to equation \eqref{definition xi} again, we have
\begin{align}\label{upper bound xi}\|\xi\|_{\gamma}^{2}\leq C_{H,V,M}\|h\|_{\bar{\mathcal{H}}}\,|z|^2\leq C_{H,V,M}|z|^{2},
\end{align}
where the last inequality stems from our assumption $\|h\|_{\bar{\mathcal{H}}}\leq M$. Therefore, taking $f_t=\xi_t$ in (\ref{eq: interpolation inequality}), plugging inequalities \eqref{lower bound xi} and \eqref{upper bound xi} and recalling inequality \eqref{eq: inner product H>1/2}, we conclude that \[
z^{*}\Gamma_{\Phi_{1}(x;h)}z\geq C_{H,V,M}|z|^{2}
\]uniformly for $\|h\|_{\bar{\mathcal{H}}}\leq M$ and the result follows.\\

\noindent
\emph{(ii) Proof of the lower bound when  $H\leq1/2$.}
Recall again that \eqref{Malliavin matrix X(h)} yields
$$z^{*}\Gamma_{\Phi_{1}(x;h)}z  =\|z^{*}D\Phi_{1}(x;h)\|_{\mathcal{H}}^{2}.$$
Then owing to the continuous embedding $\mathcal{H}\subseteq L^2([0,1])$ proved in Lemma \ref{lem: continuous embedding when H<1/2}, and expression~\eqref{eq: expression Malliavin derivative} for $D\Phi_t$, we have for any $z\in\mathbb{R}^N$,
\begin{align*}
z^{*}\Gamma_{\Phi_{1}(x;h)}z  & \geq C_{H}\|z^{*}D\Phi_{1}(x;h)\|_{L^{2}([0,1])}^{2}\\
 & =C_{H}\int_{0}^{1}z^{*}J_{1}J_{t}^{-1}V(\Phi_{t})V^{*}(\Phi_{t})(J_{t}^{-1})^{*}J_{1}^{*}zdt.
 \end{align*}
We can  now invoke the uniform ellipticity assumption \eqref{eq:hyp-elliptic} and the non-degeneracy of $J_t$ in order to obtain
 \begin{align*}
z^{*}\Gamma_{\Phi_{1}(x;h)}z \geq C_{H,V,M}|z|^{2}
\end{align*}uniformly for $\|h\|_{\bar{\mathcal{H}}}\leq M$. Our claim \eqref{uniform bound for Malliavin matrix} now follows as in the case $H>1/2$.
\end{proof}

With the preliminary results of Proposition \ref{th: summary of 4} and Lemma \ref{lem: uniform nondegeneracy of Malliavin matrix} in hand,  we are now able to complete the proof of Theorem \ref{thm: local lower estimate in elliptic case}.

\begin{proof}[Proof of Theorem \ref{thm: local lower estimate in elliptic case}] 

Recall that $X^\ep(h)$ is defined by \eqref{def approx X(h)}. According to our preliminary bound (\ref{eq: CM theorem}), it remains to show that 
\begin{equation}\label{eq: uniform lower estimate for Y^epsilon}
\mathbb{E}\left[\delta_{0}\left(X^\varepsilon(h)\right){\rm e}^{-I\left(\frac{h}{\varepsilon}\right)}\right]\geq C_{H,V}
\end{equation}uniformly in $h$ for $\|h\|_{\bar{\mathcal{H}}}\leq 2\varepsilon$ when $\varepsilon$ is small enough. The proof of this fact consists of the following two steps:\\
\\
(i)  Prove that $\mathbb{E}[\delta_{0}(X(h)){\rm e}^{-I\left({h}/{\varepsilon}\right)}]\geq C_{H,V}$ for all $\varepsilon>0$ and $h\in\bar{\mathcal{H}}$ with $\|h\|_{\bar{\mathcal{H}}}\leq 1$;
\\
(ii) Upper bound the difference \[
\mathbb{E}\left[\delta_{0}\left(X^\varepsilon(h)\right){\rm e}^{-I\left(\frac{h}{\varepsilon}\right)}\right]-\mathbb{E}\left[\delta_{0}(X(h)){\rm e}^{-I\left(\frac{h}{\varepsilon}\right)}\right],
\]and show that it is small uniformly in $h$ for $\|h\|_{\bar{\mathcal{H}}}\leq 2\varepsilon $ when $\varepsilon$ is small.
We now treat the above two parts separately.

\noindent \emph{Proof of item (i):}  Recall that the first chaos $\mathcal{C}_1$ has been defined in Section 2.1. then observe that the random variable $X(h)=(X^1(h),...,X^N(h))$ introduced in Proposition \ref{th: summary of 4} sits in $\mathcal{C}_1$. We decompose the Wiener integral $I(h/\ep)$ as  $$I\left(h/\varepsilon\right)=G_{1}^{\varepsilon}+G_{2}^{\varepsilon},$$ where $G^\varepsilon_1$ and $G^\ep_2$ satisfy
$$G_1^\ep\in\mathrm{Span}\{X^i(h); 1\leq i\leq N\},\quad G^\ep_2\in \mathrm{Span}\{X^i(h); 1\leq i\leq N\}^\bot $$
where the orthogonal complement is considered in $\mathcal{C}_1$. With this decomposition in hand, we get
\begin{align*}
  \mathbb{E}\left[\delta_{0}(X(h)){\rm e}^{-I\left(\frac{h}{\varepsilon}\right)}\right]
 =\mathbb{E}\left[\delta_{0}(X(h)){\rm e}^{-G_{1}^{\varepsilon}}\right]\cdot\mathbb{E}\left[{\rm e}^{-G_{2}^{\varepsilon}}\right].
\end{align*}
Furthermore, $\mathbb{E}[\mathrm{e}^{G}]\geq 1$ for any centered Gaussian random variable $G$. Thus
\begin{align}\label{mid step-lower bound density}
  \mathbb{E}\left[\delta_{0}(X(h)){\rm e}^{-I\left(\frac{h}{\varepsilon}\right)}\right]
 \geq\mathbb{E}\left[\delta_{0}(X(h)){\rm e}^{-G_{1}^{\varepsilon}}\right].
\end{align} 
Next we approximate $\delta_0$ above by a sequence of function $\{\psi_n; n\geq 1\}$ compactly supported in $B(0,1/n)\subset\R^N$. Taking limits in the right hand-side of \eqref{mid step-lower bound density} and recalling that $G^\ep_1\in\mathrm{Span}\{X^i(h); 1\leq i\leq N\}$, we get
 \begin{align*}
  \mathbb{E}\left[\delta_{0}(X(h)){\rm e}^{-I\left(\frac{h}{\varepsilon}\right)}\right] \geq\mathbb{E}[\delta_{0}(X(h))].
\end{align*} 
We now resort to the fact that $X(h)$ is a Gaussian random variable with covariance matrix $\Gamma_{\Phi_1(x;h)}$ by \eqref{Malliavin matrix X(h)}, which satisfies relation \eqref{uniform bound for Malliavin matrix}. This yields
 \begin{align*}
  \mathbb{E}\left[\delta_{0}(X(h)){\rm e}^{-I\left(\frac{h}{\varepsilon}\right)}\right] \geq \frac{1}{(2\pi)^{\frac{N}{2}}\sqrt{\det \Gamma_{\Phi_1(x;h)}}}\geq C_{H,V},
\end{align*} 
uniformly for $\|h\|_{\bar{\mathcal{H}}}\leq1$. This ends the proof of item (i).

\bigskip

\noindent\emph{Proof of item (ii):} By using the integration by parts formula in Malliavin's calculus, the expectation $\mathbb{E}[\delta_0(X^\varepsilon(h))\mathrm{e}^{-I(h/\varepsilon)}]$ can be expressed in terms of the Malliavin derivatives of $I(h/\varepsilon)$, $X^\varepsilon(h)$ and the inverse Malliavin covariance matrix $M_{X^\varepsilon(h)}$ of $X^\varepsilon(h)$, and similarly for $\mathbb{E}[\delta_0(X(h))\mathrm{e}^{-I(h/\varepsilon)}]$. 
In addition, from standard argument (cf. \cite[Lemma 3.4]{BOZ15}), one can show that $\det M_{X^\varepsilon(h)}$ has negative moments of all orders uniformly for all $\varepsilon\in (0,1)$ and bounded $h\in\bar{\mathcal{H}}$. Together with the convergence $\mathbb{D}$\,-$\lim_{\ep\to0}X^\ep(h)=X(h)$ in Proposition \ref{th: summary of 4}, we conclude that \[
\det M_{X^{\varepsilon}(h)}^{-1}\stackrel{L^{p}}{\longrightarrow}\det M_{\Phi_{1}(x;h)}^{-1},\qquad {\rm as}\ \varepsilon\rightarrow0
\]uniformly for $\|h\|_{\bar{\mathcal{H}}}\leq 1$ for each $p\geq1.$ Therefore, the assertion of item (ii) holds.

Once item (i) and (ii) are proved, it is easy to obtain \eqref{eq: uniform lower estimate for Y^epsilon} and the details are omitted. This finishes te proof of Theorem \ref{thm: local lower estimate in elliptic case}. 
\end{proof}

\section{Hypoelliptic case: local estimate for the control distance function.}\label{sec: control distance} 

In this section, we prove Theorem \ref{thm: local comparison}  in the hypoelliptic case. In contrast to the elliptic case, it should be noticed that  one cannot explicitly construct a Cameron-Martin path joining two points in the sense of Definition \ref{def: CM bridge} in any easy way (i.e. no simple analogue of Lemma \ref{lem: explicit construction of joining h} is possible). The analysis of  Cameron-Martin norms also becomes more involved. We detail those steps below, starting with some technical lemmas.

\subsection{Preliminary results.}

As we mentioned above, it is quite difficult to explicitly construct a Cameron-Martin path joining $x$ to $y$ in the sense of differential equation in the hypoelliptic case. However, it is possible to find some $u\in\mathfrak{g}^{(l)}$ joining $x$ to $y$ \textit{in the sense of Taylor approximation}, i.e. $y=x+F_l(u,x)$ as introduced in Definition \ref{def: Taylor approximation function}. This is the content of the following fundamental lemma proved in \cite{KS87}, which will be crucial for us in the proofs of  both Theorem \ref{thm: local comparison} and Theorem \ref{thm: local lower estimate}. Recall that $l_0$ is the hypoellipticity constant in the assumption \eqref{eq:unif-hypo-assumption}.
 
 \begin{lem}\label{lem: the Psi function}
 For each $l\geq l_{0}$, there exist constants $r,A>0$ depending only
on $l$ and the vector fields, and a $C_{b}^{\infty}$-function 
\[
\Psi_l:\left\{ u\in\mathfrak{g}^{(l)}:\|u\|_{\textsc{HS}}<r\right\} \times\mathbb{R}^{N}\times\left\{ \eta\in\mathbb{R}^{N}:|\eta|<r\right\} \rightarrow\mathfrak{g}^{(l)},
\]
such that for all $u,x,\eta$ in the domain of $\Psi_l,$ we have:
\\
(i) $\Psi_l(u,x,0)=u$;
\\
(ii) $\|\Psi_l(u,x,\eta)-u\|_\textsc{HS}\leq A|\eta|;$
\\
(iii) $F_{l}(\Psi_l(u,x,\eta),x)=F_{l}(u,x)+\eta.$
 \end{lem}

The intuition behind the function $\Psi_l$ is the following. Let $y\triangleq x+F_l(u,x)$ so that  $u$ joins $x$ to $y$ in the sense of Taylor approximation. Then $v\triangleq\Psi_l(u,x,\eta)$ joins $x$ to $y+\eta$, i.e. $x+F_l(v,x)=y+\eta$. In particular, $\Psi_l(0,x,y-x)$ gives an element in $\mathfrak{g}^{(l)}$ which joins $x$ to $y$ in the sense of Taylor approximation, provided $|y-x|<r$. The proof of this lemma, for which we refer again to \cite{KS87},  is based on a non-degeneracy property of $F_l$ stated in Lemma~\ref{lem: nondegeneracy of F}  due to hypoellipticity, as well as a parametrized version of the classical inverse function theorem.

We begin with some easy preliminary steps toward the proof of Theorem \ref{thm: local comparison}, namely the lower bound on $d(\cdot,\cdot)$ and the upper bound for the case $H<1/2$.
\begin{lem}\label{lem: preliminary case}
Assume that the vector fields in equation \eqref{eq: hypoelliptic SDE} satisfy the uniform hypoellipticity assumption \eqref{eq:unif-hypo-assumption} with constant $l_0$. Let $d=d_{H}$ be the control distance introduced in Definition \ref{def: CM bridge}. Then the following bounds hold true.\\
(i) For all $H\in(1/4,1)$ and $x,y$ such that $|x-y|\leq 1$, we have
$$d(x,y)\geq C_1|x-y|.$$\\
(ii) Whenever $H\in(1/4,1/2)$ we have
$$d(x,y)\leq C_2|x-y|^\frac{1}{l_0}.$$
\end{lem}
\begin{proof}
Claim (i) follows from the exact same argument as in the proof of Theorem \ref{thm: local comparison in elliptic case}.
Claim (ii) stems from the fact that when $H<1/2$,
\begin{align}\label{distance upper bound}d(x,y)\leq C_H\, d_\mathrm{BM}(x,y)\end{align}
where $d_{BM}$ stands for the distance for the Brownian motion case. Note that \eqref{distance upper bound} can be easily justified by the fact that, according to Lemma \ref{lem: continuous embedding when H<1/2}, 
$$d(x,y)\leq \|h\|_{\bar{\mathcal{H}}}\leq C_H\|h\|_{W^{1,2}},$$
for any $h\in\Pi_{x,y}$. Then, with \eqref{distance upper bound} in hand, our claim (ii) follows from the Brownian hypoelliptic analysis \cite{KS87}.


\end{proof}

In the remainder of the section, we focus on the case $H>1/2$. It is not surprising that this is the hardest case since the Cameron-Martin subspace $\bch$ gets smaller as $H$ increases.  First, we need to make use of the following scaling property of Cameron-Martin norm. Namely denote $\bar{\mathcal{H}}([0,T])$ (respectively, $d_T(x,y)$) as the Cameron-Martin subspace (respectively, the control distance function) associated with fractional Brownian motion over $[0,T]$. Then the following property holds true.

\begin{lem}\label{lem: CM scaling}
Let $0<T_1<T_2$, and consider $H>1/2$. Given  $h\in\bar{\mathcal{H}}([0,T_1])$, define $\tilde{h}_t\triangleq h_{T_1 t/T_2}$ for $0\leq t\leq T_2$. Then $\tilde{h}\in\bar{\mathcal{H}}[0,T_2]$, and \begin{align}\label{bar-H norm rescaling}
\|\tilde{h}\|_{\bar{\mathcal{H}}([0,T_{2}])}=\left(\frac{T_{1}}{T_{2}}\right)^{H}\|h\|_{\bar{{\cal H}}([0,T_{1}])}.
\end{align}In particular, let $d_T$ be the distance introduced in Definition \ref{def: CM bridge} associated with a fractional Brownian motion over $[0,T]$. Then we have
\begin{align}\label{distance rescaling}
d_{1}(x,y)=T^{H}d_{T}(x,y),\ \ \ \forall T>0,\ x,y\in\mathbb{R}^{N}.
\end{align}
\end{lem}
\begin{proof}
Recall that, thanks to relation \eqref{eq: inner product in terms of fractional integrals}, we have
\begin{align}\label{bar-H norm}
\|\tilde{h}\|_{\bar{\mathcal{H}}([0,T_2])}=\|K^{-1}\tilde{h}\|_{L^2([0,T_2])}.
\end{align}
Moreover, invoking relation \eqref{eq: analytic expression of K} for $H>1/2$, we get
 \begin{align}\label{expression K inverse}
(K^{-1}h)_{t}=C_{H}\cdot t^{H-\frac{1}{2}}D_{0^{+}}^{H-\frac{1}{2}}\left(s^{\frac{1}{2}-H}\dot{h}_{s}\right)(t)
.\end{align}
Plugging \eqref{expression K inverse} into \eqref{bar-H norm}  and performing an elementary  change of variables, one ends up with \[
\|\tilde{h}\|_{\bar{{\cal H}}([0,T_{2}])}=\|K^{-1}\tilde{h}\|_{L^{2}([0,T_{2}])}=\left(\frac{T_{1}}{T_{2}}\right)^{H}\|K^{-1}h\|_{L^{2}([0,T_{1}])}=\left(\frac{T_{1}}{T_{2}}\right)^{H}\|h\|_{\bar{{\cal H}}([0,T_{1}])},
\]and the assertion \eqref{bar-H norm rescaling} follows. The second claim \eqref{distance rescaling} is now easily deduced. 
\end{proof}
\begin{rem}
In fact Lemma \ref{lem: CM scaling} also holds true for $H\leq1/2$. However, it will only be invoked for the case $H>1/2$.
\end{rem}

We also need the following lemma about the free nilpotent group $G^{(l)}$ which allows us to choose a "regular" path $\gamma$ with $S_l(\gamma)=u$  for all $u\in G^{(l)}$. 
 
\begin{lem}\label{lem: quasi-inverse of signature map}
Let $l\geq1$. For each $M>0$, there exists a constant $C=C_{l,M}>0$, such that for every $u\in G^{(l)}$ with $\|u\|_{\textsc{CC}}\leq M$, one can find a smooth path $\gamma:[0,1]\rightarrow\mathbb{R}^d$ which satisfies: \\
(i) $S_l(\gamma)=u$;
\\
(ii) $\dot{\gamma}$ is supported on $[1/3,2/3]$;
\\
(iii) $\|\ddot{\gamma}\|_{\infty;[0,1]}\leq C$.
\end{lem}

\begin{proof}
We first prove the claim for a generic element $u\in\exp(\mathcal{L}_{k})$, seen as an element of $G^{(k)}$. Let $\{a_{1},\ldots,a_{d_{k}}\}$ be a basis of $\mathcal{L}_{k}$
where $d_{k}\triangleq\dim\mathcal{L}_{k}$\,. Given $u\in\exp(\mathcal{L}_{k})$,  we can write $u=\exp(a)$ with
\begin{align}\label{linear combination}a=\lambda_{1}a_{1}+\cdots+\lambda_{d_{k}}a_{d_{k}}\in\mathcal{L}_{k}\end{align}
for some $\lambda_{1},\ldots,\lambda_{d_{k}}\in\mathbb{R}.$ Since
we assume that $\|u\|_{\text{CC}}\leq M,$ according to the ball-box
estimate (cf. Proposition \ref{prop: ball-box estimate}) and the fact that $a\in\mathcal{L}_{k}$, we have 
\begin{align}\label{distance bound for a}
\|a\|_{\mathrm{HS}}=\|u-\mathbf{1}\|_{\mathrm{HS}}\leq C_{1,l,M}.
\end{align}
Moreover, $\mathcal{L}_{k}$ is a finite dimensional vector space, on which all norms are equivalent.  Thus relation \eqref{distance bound for a} yields 
\begin{equation}\label{eq: boundedness of lambda_i}
\max_{1\leq i\leq d_{k}}|\lambda_{i}|\leq C_{2,l,M}.
\end{equation}

Now recall from Remark \ref{connecting by smooth path} that for each $a_{i}$ in \eqref{linear combination} one can  choose
a smooth path $\alpha_i:[0,1]\rightarrow \mathbb{R}^d$ such
that $S_{k}(\alpha_{i})=\mathrm{exp}(a_{i})$ and $\dot{\alpha}_i$ is supported on $[1/3,2/3]$.  Set 
\[R_{k}\triangleq\max\left\{ \|\ddot{\alpha}_{i}\|_{\infty;[0,1]}:1\leq i\leq d_{k}\right\}.\]Note that $R_{k}$ is a constant depending only on $k$.
We construct a smooth path $\gamma:[0,d_{k}]\rightarrow\mathbb{R}^d$ by
\begin{align}\label{definition gamma}
\gamma\triangleq\left(|\lambda_{1}|^{\frac{1}{k}}\alpha_{1}^{\mathrm{sgn}(\lambda_{1})}\right)\sqcup\cdots\sqcup\left(|\lambda_{d_{k}}|^{\frac{1}{k}}\alpha_{d_{k}}^{\mathrm{sgn}(\lambda_{d_{k}})}\right),
\end{align}
where $\alpha_{i}^{-1}$ denotes the reverse of $\alpha_{i},$ and $\sqcup$
denotes path concatenation. Then $\dot{\gamma}$ is obviously compactly supported, and we also claim that $S_{k}(\gamma)=u$.  Indeed, it follows from \eqref{definition gamma} that
\begin{align}\label{c0}
S_{k}(\gamma) & =S_{k}\left(|\lambda_{1}|^{\frac{1}{k}}\alpha_{1}^{{\rm sgn}(\lambda_{1})}\right)\otimes\cdots\otimes S_{k}\left(|\lambda_{d_{k}}|^{\frac{1}{k}}\alpha_{d_{k}}^{{\rm sgn}(\lambda_{d_{k}})}\right)
\notag\\
 & =\delta_{|\lambda_{1}|^{\frac{1}{k}}}\left(S_{k}\left(\alpha_{1}^{{\rm sgn}(\lambda_{1})}\right)\right)\otimes\cdots\otimes\delta_{|\lambda_{d_{k}}|^{\frac{1}{k}}}\left(S_{k}\left(\alpha_{d_{k}}^{{\rm sgn}(\lambda_{d_{k}})}\right)\right)\notag\\
 & =\delta_{|\lambda_{1}|^{\frac{1}{k}}}\left(\exp({\rm sgn}(\lambda_{1})a_{1})\right)\otimes\cdots\otimes\delta_{|\lambda_{d_{k}}|^{\frac{1}{k}}}\left(\exp({\rm sgn}(\lambda_{d_{k}})a_{d_{k}})\right),
\end{align}
where we have used the properties of the dilation, recalled in Section \ref{sec:free-nilpotent}, and the relation between signatures and $G^{(l)}$ given in \eqref{eq:def-iterated-intg} -- \eqref{eq:signature-smooth-x}. 
In addition, since each element $\exp(\lambda_{i}a_{i})$ above sits in $\exp(\cl_{k})$, the tensor product in $G^{(k)}$ is reduced to
\begin{equation}\label{c1}
S_{k}(\gamma)
=
\exp(\lambda_{1}a_{1})\otimes\cdots\otimes\exp(\lambda_{d_{k}}a_{d_{k}})
 =\exp(a)=u.
\end{equation}
We have thus found a path $\gamma$ with compactly supported derivative such that $S_{k}(\gamma)=u$.
In addition, from the definition of $R_{k}$ and (\ref{eq: boundedness of lambda_i}), we have
\[
\|\ddot{\gamma}\|_{\infty;[0,d_{k}]}\leq R_{k}\cdot\left(\max_{1\leq i\leq d_{k}}|\lambda_{i}|\right)^{\frac{1}{k}}\leq C_{3,l,M}.
\]
By suitable rescaling and adding trivial pieces on both ends if necessary,
we may assume that $\gamma$ is defined on $[0,1]$ and $\dot{\gamma}$ is supported on $[1/3,2/3]$. In this way, we have$$\|\ddot{\gamma}\|_{\infty;[0,1]}\leq C_{k}\cdot C_{3,k,M}
\triangleq C_{4,k,M},$$where $C_{k}$ is the constant coming from the rescaling. Therefore, our assertion (i)--(iii)
holds for  $u$ which are elements of $\exp(\mathcal{L}_{k})$.

With the help of the previous special case, we now prove the lemma
by induction on $l$. The case when $l=1$ is obvious, as we can simply
choose $\gamma$ to be a straight line segment. Suppose now that the claim
is true on $G^{(l-1)}$. We let $M>0$ and $u\in G^{(l)}$ with $\|u\|_{\text{CC}}\leq M.$
Define $v\triangleq\pi^{(l-1)}(u)$ where $\pi^{(l-1)}:\ G^{(l)}\rightarrow G^{(l-1)}$
is the canonical projection. We obviously have 
\[
\|v\|_{\text{CC}}\leq\|u\|_{\text{CC}}\leq M,
\]
where the CC-norm of $v$ is taken on the group $G^{(l-1)}$. According
to the induction hypothesis, there exists a constant $C_{l-1,M},$
such that we can find a smooth path $\alpha:[0,1]\rightarrow\mathbb{R}^d$ which satisfies (i)--(iii) in the assertion of Lemma \ref{lem: quasi-inverse of signature map}, for $v=S_{l-1}(\alpha)$ and constant $C_{l-1,M}$. Define 
\begin{align}\label{def: w}w\triangleq\left(S_{l}(\alpha)\right)^{-1}\otimes u,\end{align}
{where the tensor product is defined on $G^{(l)}$}. Then 
note that owing to the fact that $\|u\|_{\text{CC}}\leq M$, we have
\begin{align*}
\|w\|_{\textsc{CC}} & \leq\|S_{l}(\alpha)\|_{\textsc{CC}}+\|u\|_{\textsc{CC}}\leq\|\alpha\|_{1{\rm -var};[0,1]}+\|u\|_{\textsc{CC}}
  \leq\frac{1}{2}\|\ddot{\alpha}\|_{\infty;[0,1]}+M.
 \end{align*}
 Therefore, thanks to the induction procedure applied to $v=S_{l-1}(\alpha)$, we get
 \begin{align*}
 \|w\|_{\textsc{CC}}\leq \frac{1}{2}C_{l-1,M}+M\triangleq C_{5,l,M}.
\end{align*}

We claim that $w\in\exp(\mathcal{L}_{l})$. This can be proved in the following way. 

\noindent(i) Write $u=\exp(l_0+l_h)$, where $l_0\in\mathfrak{g}^{(l-1)}$ and $l_h\in\mathcal{L}_l$. Recall $v\triangleq\pi^{(l-1)}(u)$. We argue  that $v=\exp(l_0)\in G^{(l-1)}$ as follows:  since $l_h\in \mathcal{L}_l$, any product of the form $l^p_h\otimes l^q_0=0$ whenever $p,q>0$. Taking into account the definition \eqref{eq:def-exp-on-T-l} of the exponential function, we get that
\begin{align}\label{u imply v}
u=\exp(l_0+l_h)\ \Longrightarrow\ v=\exp(l_0)\in G^{(l-1)}.
\end{align}

\noindent(ii) Recall that our induction hypothesis asserts that $v=S_{l-1}(\alpha)$, thus according to \eqref{u imply v} we have $S_{l-1}(\alpha)=\exp(l_0)$.  Thanks to the same kind of argument as in (i), we get  $S_l(\alpha)=\exp(l_0+l_h')\in G^{(l)}$ for some $l_h'\in\mathcal{L}_l$. 

\noindent(iii) In order to conclude that $w\in\exp(\mathcal{L}_l)$, we go back to relation \eqref{def: w}, which can now be read as
$$w=\left(\exp(l_0+l'_h)\right)^{-1}\otimes\exp{(l_0+l_h)}.$$
According to Campbell-Baker-Hausdorff formula and taking into account the fact that
 \[
[l_{0},l_{0}]=[l_{0},l_{h}]=[l_{0},l_{h}']=[l_{h},l_{h}']=0\in\mathfrak{g}^{(l)},
\] we conclude that $w=\exp(l_h-l_h')$ and thus $w\in\exp(\mathcal{L}_l)$.

We are now ready to summarize our information and conclude our induction procedure. Namely, for $u\in G^{(l)}$, we can recast relation \eqref{def: w} as
\begin{align}\label{decomposition u}
u=S_l(\alpha)\otimes w,
\end{align}
and we have just proved that $w\in\exp(\mathcal{L}_l)$. Hence relation \eqref{c1} asserts that $w$ can be written as $w=S_l(\beta)$, where $\beta: [0,1]\to \mathbb{R}^d$ satisfying relation (i)-(iii) in Lemma \ref{lem: quasi-inverse of signature map} with $C=C_{6,l,M}$.
Now set $\gamma\triangleq\alpha\sqcup\beta$ and rescale it so that it is defined on $[0,1]$ and its derivative path is supported on $[1/3,2/3]$. Then, recalling our decomposition \eqref{decomposition u}, we have 
\[
S_{l}(\gamma)=S_{l}(\alpha)\otimes S_{l}(\beta)=S_{l}(\alpha)\otimes w=u,
\]
and,  moreover, the following upper bound holds true
\begin{align*}
\|\ddot{\gamma}\|_{\infty;[0,1]} & \leq36\max\left\{ \|\ddot{\alpha}\|_{\infty;[0,1]},\|\ddot{\beta}\|_{\infty;[0,1]}\right\} \leq C_{7,l,M}. 
\end{align*} Therefore our induction procedure is established, which finishes the proof.

\end{proof}

We conclude this subsection by  stating  a  convention on the group $G^{(l)}$ which will ease notation in our future computations. 

\begin{conv}\label{conv: convention}

Since $\mathfrak{g}^{(l)}$ is a finite dimensional vector space on which differential calculus is easier to manage, we will frequently identify $G^{(l)}$ with $\mathfrak{g}^{(l)}$ through the exponential diffeomorphism without further mention.  This is not too beneficial when proving Theorem \ref{thm: local comparison} but will be very convenient when proving Theorem \ref{thm: local lower estimate}. 
In this way, for instance, $S_l(w)=u$ means $S_l(w)=\exp(u)$ if $u\in\mathfrak{g}^{(l)}$. The same convention will apply to other similar relations when the meaning is clear from context.
For norms on $\mathfrak{g}^{(l)}$, we denote $\|u\|_\mathrm{CC}\triangleq\|\exp(u)\|_{\mathrm{CC}}$. As for the $\rm HS$-norm, note that \[
C_{1,l}\|u\|_{{\rm HS}}\leq\|\exp(u)-{\bf 1}\|_{{\rm HS}}\leq C_{2,l}\|u\|_{{\rm HS}}
\]for all $u\in\mathfrak{g}^{(l)}$ satisfying $\|\exp(u)-{\bf 1}\|_{{\rm HS}}\wedge\|u\|_{{\rm HS}}\leq1$. Therefore, up to a constant depending only on $l$, the notation $\|u\|_{\rm{HS}}$ can either mean the $\rm HS$-norm of $u$ or $\exp(u)-\bf 1$. This will not matter because we are only concerned with local estimates. The same convention applies to the distance functions $\rho_\mathrm{CC}$ and $\rho_\mathrm{HS}$.
\end{conv}

\subsection{Proof of Theorem \ref{thm: local comparison}.}
\label{sec:proof-thm-12}

In this section we give the details in order  to complete the proof of Theorem \ref{thm: local comparison}. 
Notice that thanks to our preliminary Lemma \ref{lem: preliminary case}, we only focus on the upper bound on the distance $d$ for $H>1/2$.

Recall that $\Psi_{l}(u,x,\eta)$ is the function given by Lemma \ref{lem: the Psi function}.
This function allows us to construct elements in $\mathfrak{g}^{(l)}$ joining
two points in the sense of Taylor approximation locally. In the following,
we take $l=l_{0}$ (where $l_0$ stands for the hypoellipticity constant) and we will omit the subscript $l$ for simplicity (e.g., $F=F_l$ and $\Psi=\Psi_l$) .
We will also identify $G^{(l)}$ with $\mathfrak{g}^{(l)}$ in the
way mentioned in Convention \ref{conv: convention}. We now divide our proof in several steps.

\medskip

\noindent{\it Step 1: Construction of an approximating sequence.}
Let $\delta<r$ be a constant to be chosen later on, where $r$ is
the constant appearing in the domain of $\Psi$ in Lemma \ref{lem: the Psi function}. Consider $x,y\in\mathbb{R}^{N}$ with $|x-y|<\delta.$

We are going to construct three sequences $\{x_{m}\}\subseteq\mathbb{R}^{N}$,
$\{u_{m}\}\subseteq\mathfrak{g}^{(l_{0})}$, $\{h_{m}\}\subseteq C^{\infty}([0,1];\mathbb{R}^{d})$
inductively. We start with $x_{1}\triangleq x$ and define the rest
of them by the following general procedure in the order $$u_{1}\rightarrow h_{1}\rightarrow x_{2}\rightarrow u_{2}\rightarrow h_{2}\rightarrow x_{3}\rightarrow\cdots.$$
To this aim, suppose we have already defined $x_{m}.$ Set 
\begin{align}\label{def: um}
u_{m}\triangleq\Psi(0,x_{m},y-x_{m}),\quad\text{and}\quad \bar{u}_{m}\triangleq\delta_{\|u_{m}\|_{\text{CC}}^{-1}}u_{m}.
\end{align}
By Lemma \ref{lem: the Psi function}, the first condition in \eqref{def: um} states that $u_m$ is an element of $\mathfrak{g}^{(l_0)}$ such that \begin{align}\label{xm and y}x_m+F(u_m,x_m)=y,\end{align} while the second condition in \eqref{def: um} ensures that  $\|\bar{u}_{m}\|_{\text{CC}}=1$. Once $u_m$ is defined, we construct $h_m$ in the following way: let $\bar{h}_{m}:[0,1]\rightarrow\mathbb{R}^{d}$
be the smooth path given by Lemma~\ref{lem: quasi-inverse of signature map} such that $S_{l_{0}}(\bar{h}_{m})=\bar{u}_{m}$,
$\dot{\bar{h}}_{m}$ is supported on $[1/3,2/3],$ and $\|\ddot{\bar{h}}_{m}\|_{\infty;[0,1]}\leq C_{l_{0}}$.
Define 
\begin{align}\label{def: hm}
h_{m}\triangleq\|u_{m}\|_{\text{CC}}\bar{h}_{m},
\end{align}
so that the truncated signature of $h_m$ is exactly $u_m$ (here recall the Convention \ref{conv: convention}). More specifically, we have:
\[
S_{l_{0}}(h_{m})=S_{l_{0}}(\|u_{m}\|_{\text{CC}}\cdot\bar{h}_{m})=\delta_{\|u_{m}\|_{\text{CC}}}(S_{l_{0}}(\bar{h}_{m}))=\delta_{\|u_{m}\|_{\text{CC}}}(\overline{u}_{m})=u_{m}.
\]
Taking into account  the definition \eqref{eq:def-norm-CC} of the $\textsc{CC}$-norm, it is immediate that
\begin{equation}
\|u_{m}\|_{\text{CC}}\leq\|h_{m}\|_{1\text{-var};[0,1]}\leq\|u_m\|_{\text{CC}}\|\bar{h}_m\|_{1\text{-var};[0,1]}\leq C_{l_{0}}\|u_{m}\|_{\text{CC}}\, ,\label{eq: controlling h_m in terms of u_m CC}
\end{equation}where the last inequality stems from the fact that $\bar{h}_m$ has a bounded second derivative.
Eventually we define 
\begin{align}\label{def: xm}
x_{m+1}\triangleq\Phi_{1}(x_{m};h_{m}),
\end{align}
where recall that $\Phi_{t}(x;h)$ is the solution flow of the ODE
(\ref{eq: ode}) driven by $h$ over $[0,1].$

\medskip
\noindent{\it Step 2: Checking the condition $|y-x_m|< r$.}
Recall that in Lemma \ref{lem: the Psi function} we have to impose $\|u\|_{\text{HS}}<r$ and $|\eta|<r$ in order to apply $\Psi$. In the context of \eqref{def: um} it means that we should make sure that 
\begin{align}\label{xm close to y}|y-x_m|<r, \quad\text{for\ all} \ m.\end{align}
We will now choose $\delta_1$ small enough such that if $|y-x|<\delta_1$, then \eqref{xm close to y} is satisfied.
This will guarantee that $u_{m}$
is well-defined  by Lemma \ref{lem: the Psi function} and we will also be able to write down several useful estimates
for $x_{m}$ and $u_{m}.$ Our first condition on $\delta_1$ is that $\delta_1\leq r$, so that if $|x-y|<\delta_1$, we can define $u_1$ by a direct application of Lemma \ref{lem: the Psi function}. We will now prove by induction that if $\delta_1$ is chosen small enough, then condition \eqref{xm close to y} is satisfied. To this aim, assume that $|x_m-y|<\delta_1$. Then one can apply Lemma \ref{lem: the Psi function} in order to define $u_m, h_m$ and $x_{m+1}$. We also get the following estimate: 
\begin{align}\label{um bound by xm-y}
\|u_{m}\|_{\textsc{HS}}\leq A|x_{m}-y|<A\delta_{1},
\end{align}
where $A$ is the constant appearing in Lemma \ref{lem: the Psi function}. In addition, let us require $\delta_{1}\leq1/A$ so that
$\|u_{m}\|_{\textsc{HS}}\leq1.$ Recalling relations \eqref{xm and y} and \eqref{def: xm} we get
\begin{align*}
|x_{m+1}-y|=|\Phi_1(x_m,h_m)-x_m-F(S_{l_0}(h_m),x_m)|.
\end{align*}
Thus applying successively the Taylor type estimate of \cite[Proposition 10.3]{FV10} and relation~\eqref{eq: controlling h_m in terms of u_m CC} we end up with
$$|x_{m+1}-y| \leq C_{V,l_{0}}\|h_{m}\|_{{\rm 1}-{\rm var};[0,1]}^{1+l_{0}}\leq C_{V,l_{0}}\|u_{m}\|_{\textsc{CC}}^{1+l_{0}}.$$
The quantity $\|u_m\|_{\textsc{CC}}$ above can be bounded thanks to the ball-box estimate of Proposition~\ref{prop: ball-box estimate}, for which we observe that the dominating term in \eqref{eq:bnd-CC-to-HS} is $\rho_{\text{HS}}(g_1,g_2)^{1/l_0}$ since our element $u_m$ is bounded by one in HS-norm. We get 
\begin{align*}
 |x_{m+1}-y|\leq  C_{V, l_0}\|u_m\|_{\textsc{CC}}^{1+l_0}  \leq C_{V,l_{0}}\|u_{m}\|_{\textsc{HS}}^{1+\frac{1}{l_{0}}}  \leq C_{V,l_{0}}A^{1+\frac{1}{l_{0}}}|x_{m}-y|^{1+\frac{1}{l_{0}}}.
\end{align*}
Summarizing our considerations so far, we have obtained the estimate
\begin{equation}
|x_{m+1}-y|\leq C_{1,V,l_{0}}\|u_{m}\|_{\textsc{CC}}^{1+l_{0}}\leq C_{2,V,l_{0}}|x_{m}-y|^{1+\frac{1}{l_{0}}}.\label{eq: x-u recursive estimate}
\end{equation}
On top of the inequalities $\delta_1<r$ and $\delta_1\leq 1/A$ imposed previously, we will also assume that  
$
C_{2,V,l_{0}}\delta_{1}^{1/l_{0}}\leq{1}/{2},
$
which easily yields the relation
\begin{equation}
|x_{m+1}-y|\leq\frac{1}{2}|x_{m}-y|<\frac{1}{2}\delta_{1}<\delta_{1}.\label{eq: x_m recursive estimate}
\end{equation}
For our future computations we will thus set 
$$\delta_{1}\triangleq r\wedge A^{-1}\wedge(2C_{2,V,l_{0}})^{-l_{0}}.
$$
According to our bound \eqref{eq: x_m recursive estimate}, we can guarantee that if $|x-y|<\delta_1$, then $|x_m-y|<\delta_1<r$ for all $m$. In addition, an easy induction procedure performed on inequality \eqref{eq: x_m recursive estimate} leads to the following relation, valid for all $m\geq 1$:

\begin{equation}
|x_{m}-y|\leq2^{-(m-1)}|x-y|.\label{eq: exponential decay of x_m-y}
\end{equation}
Together with the second inequality of (\ref{eq: x-u recursive estimate}),
we obtain that 
\begin{equation}
\|u_{m}\|_{\textsc{CC}}\leq C_{3,V,l_{0}}2^{-\frac{m}{l_{0}}}|x-y|^{\frac{1}{l_{0}}},\ \ \ \forall m\geq1.\label{eq: exponential decay of u_m}
\end{equation}

We will now choose a constant $\delta_2\leq\delta_1$ such that the sequence $\{\|u_m\|_{\text{CC}}; m\geq 1\}$ is decreasing with $m$ when $|x-y|<\delta_2$. This property will be useful for our future considerations. Towards this aim, observe that applying successively \eqref{eq:bnd-CC-to-HS}, \eqref{um bound by xm-y} and \eqref{eq: x-u recursive estimate} we get
\begin{equation}\label{d1}
\|u_{m+1}\|_{\textsc{CC}}\leq C_{l_0}\|u_{m+1}\|_{\text{HS}}^{\frac{1}{l_0}}\leq C_{4,V,l_0}\|u_m\|_{\textsc{CC}}^{1+\frac{1}{l_0}}.
\end{equation}
Hence involving the second inequality in (\ref{eq: x-u recursive estimate}) we have
\begin{equation}\label{eq: recursive estimate for u_m}
\|u_{m+1}\|_{\text{CC}}\leq  C_{5,V,l_{0}}|x-y|^{\frac{1}{l^2_{0}}}\|u_{m}\|_{\textsc{CC}}.
\end{equation}
Therefore, let us consider a new constant $\delta_{2}>0$ such that 
\begin{equation*}
C_{5,V,l_0}\delta_2^{\frac{1}{l^2_0}}<1.
\end{equation*}
If we choose $|x-y|<\delta$ with  $\delta\triangleq\delta_{1}\wedge\delta_{2},$  equation \eqref{eq: recursive estimate for u_m} can be recast as
\begin{align}\label{monotone um}
\|u_{m+1}\|_{\text{CC}}\leq\|u_m\|_{\text{CC}}.
\end{align}
Note that $\delta=\delta_1\wedge\delta_2$
depends only on $l_0$ and the vector fields, but not on the Hurst parameter $H$. We have thus shown that the application of Lemma \ref{lem: the Psi function} is valid in our context. 

\medskip
\noindent{\it Step 3: Construction of a path joining $x$ and $y$ in the sense of differential equation.} Our next aim is to obtain a path $\tilde{h}$ joining $x$ and $y$ along the flow of equation \eqref{eq: ode}. Our first step in this direction is to rescale $h_{m}$ in a suitable way. Namely, set $a_{1}\triangleq0,$
and for $m\geq1$, define recursively the following sequence:
\[
a_{m+1}\triangleq\sum_{k=1}^{m}\|u_m\|_\textsc{CC},\ \ I_{m}\triangleq[a_{m},a_{m+1}],\ \ I\triangleq\overline{\bigcup_{m=1}^{\infty}I_{m}}.
\]
It is clear that $|I_{m}|=\|u_m\|_\textsc{CC},$ and $I$
is a compact interval since the sequence $\{\|u_m\|_\textsc{CC}; \, m\ge 1\}$ is summable
according to (\ref{eq: exponential decay of u_m}).
We also define a family of function $\{\tilde{h}_m, m\geq1\}$ by
\begin{align}\label{def: tilde hm}
\tilde{h}_{m}(t) & \triangleq h_{m}\left(\frac{t-a_{m}}{a_{m+1}-a_{m}}\right),\ \ t\in I_{m},
\end{align}
and the concatenation of the first $\tilde{h}_m$'s is
\begin{align}\label{def: concatenation hm}
\tilde{h}^{(m)}\triangleq\tilde{h}_{1}\sqcup\cdots\sqcup\tilde{h}_{m}:\ [0,a_{m+1}]\rightarrow\mathbb{R}^{d}.
\end{align}
We will now bound the derivative  of $\tilde{h}_{m}.$ Specifically, we first use equation \eqref{def: tilde hm} to get
\[
\sup_{m\geq1}\|\dot{\tilde{h}}^{(m)}\|_{\infty;[0,a_{m+1}]}=\sup_{m\geq1}\|\dot{\tilde{h}}_{m}\|_{\infty;I_{m}}=\sup_{m\geq1}\frac{1}{|I_{m}|}\cdot\|\dot{{h}}_{m}\|_{\infty;[0,1]}.
\]
Then resort to relation \eqref{def: hm}, which yields
\[
\sup_{m\geq1}\|\dot{\tilde{h}}^{(m)}\|_{\infty;[0,a_{m+1}]}=\sup_{m\geq1}\left\{ \frac{\|u_{m}\|_{\text{CC}}}{|I_{m}|}\cdot\|\dot{\bar{h}}_{m}\|_{\infty;[0,1]}\right\}.
\]
Since $\|u_m\|_{\textsc{CC}}=|I_m|$ we end up with
\begin{align}\label{bound derivative tilde hm}
\sup_{m\geq1}\|\dot{\tilde{h}}^{(m)}\|_{\infty;[0,a_{m+1}]}=\sup_{m\geq1}\left\{\|\dot{\bar{h}}_m\|_{\infty; [0,1]}\right\}\leq C_{l_0},
\end{align}
where the last inequality stems from the fact that $\|\ddot{\bar{h}}_m\|_{\infty;[0,1]}\leq C_{l_0}.$ 

We can now proceed to the construction of the announced path joining $x$ and $y$. Namely, set
\begin{align}\label{def: tilde h}
\tilde{h}\triangleq\sqcup_{m=1}^{\infty}\tilde{h}_{m}:I\rightarrow\mathbb{R}^{d}.
\end{align}
Then according to \eqref{bound derivative tilde hm} we have that $\tilde{h}$ is a smooth function from $I$ to $\mr^d$. We also claim that $\Phi_1(x;\tilde{h})=y$, where $\Phi$ has to be understood in the sense of equation \eqref{eq: skeleton ODE}. Indeed, set $$z_{t}=\Phi_t(x; \tilde{h}),\quad t\in I.$$
From the construction of $x_{m}$ in \eqref{def: xm} and the fact that
$\tilde{h}|_{[0,a_{m+1}]}=\tilde{h}^{(m)}$ asserted in \eqref{def: tilde h}, we have
\begin{align}\label{zt at am}
x_{m+1}=x+\sum_{\alpha=1}^{d}\int_{0}^{a_{m+1}}V_{\alpha}(z_{t})d\tilde{h}_{t}^{\alpha}.
\end{align}
Since $x_{m+1}\rightarrow y$ as $m\rightarrow\infty,$ which can
be easily seen from (\ref{eq: exponential decay of x_m-y}), one can take
limits in~\eqref{zt at am} and we conclude that 
\[
y=x+\sum_{\alpha=1}^{d}\int_{0}^{|I|}V_{\alpha}(z_{t})d\tilde{h}_{t}^{\alpha}.
\]
We have thus proved that $\tilde{h}$ is a smooth path joining $x$ and $y$ in the sense of differential equations.

\medskip

\noindent{\it Step 4: Strategy for the upper bound.}  Let us recall that the family of distances $\{d_T; T>0\}$ has been introduced in Lemma \ref{lem: CM scaling}, and that they satisfy the scaling property \eqref{distance rescaling}.  Therefore we get
\begin{align}
d(x,y) & =|I|^{H}d_{|I|}(x,y)
\leq|I|^{H}\|\tilde{h}\|_{\bar{{\cal H}}([0,|I|])}\nonumber \\
 & =\lim_{m\rightarrow\infty}\left(\left(\sum_{k=1}^{m}|I_{k}|\right)^{H}\|\tilde{h}^{(m)}\|_{\bar{{\cal H}}([0,a_{m+1}])}\right),\label{eq: estimating the control metric}
\end{align}
where the last relation stems from the definition \eqref{def: tilde h} of $\tilde{h}$.

In order to estimate the right hand-side of \eqref{eq: estimating the control metric}, we use the definition \eqref{eq: inner product in terms of fractional integrals} of the Cameron-Martin norm to get
$$\|\tilde{h}^{(m)}\|^2_{\bch([0,a_{m+1}])}=\|K^{-1}\tilde{h}^{(m)}\|^2_{L^2([0,a_{m+1}]; dt)}.$$
We now invoke the formula \eqref{eq: analytic expression of K} for $K$, from which a formula for $K^{-1}$ is easily deduced. We end up with
$$\|\tilde{h}^{(m)}\|^2_{\bch([0,a_{m+1}])}=C_H\int_0^{a_{m+1}}\left|t^{H-\frac{1}{2}}D_{0+}^{H-\frac{1}{2}}(s^{\frac{1}{2}-H}\dot{\tilde{h}}^{(m)}(s))(t)\right|^{2}dt.$$
Taking into account formula (\ref{eq: formula for fractional derivatives}) for the fractional derivative, this yields
\begin{align*}
\|\tilde{h}^{(m)}\|_{\bar{{\cal H}}([0,a_{m+1}])}^{2} 
 & =C_{H}\cdot\int_{0}^{a_{m+1}}\left|t^{H-\frac{1}{2}}\left(t^{1-2H}\dot{\tilde{h}}^{(m)}(t)\right.\right.\\
 & \ \ \ \left.\left.+\left(H-\frac{1}{2}\right)\int_{0}^{t}\frac{t^{\frac{1}{2}-H}\dot{\tilde{h}}^{(m)}(t)-s^{\frac{1}{2}-H}\dot{\tilde{h}}^{(m)}(s)}{(t-s)^{H+\frac{1}{2}}}ds\right)\right|^{2}dt.
\end{align*}

We now split the interval $[0,a_{m+1}]$ as $[0, a_{m+1}]=\cup_{k=0}^mI_k$ and use the elementary inequality $(a+b+c)^2\leq 3(a^2+b^2+c^2)$ in order to get
\begin{align}\label{decomposition to Q}\|\tilde{h}^{(m)}\|_{\bar{{\cal H}}([0,a_{m+1}])}^{2}\leq Q_1+Q_2+Q_3,\end{align}
with
\begin{align}
Q_{1} & \triangleq\sum_{k=1}^{m}\int_{I_{k}}\left|t^{H-\frac{1}{2}}\left(t^{1-2H}\dot{\tilde{h}}_{k}(t)\right)\right|^{2}dt\triangleq\sum_{k=1}^m Q_{1,k},\label{Q1}\\
Q_{2} & \triangleq\sum_{k=1}^{m}\int_{I_{k}}\left|t^{H-\frac{1}{2}}\sum_{l=1}^{k-1}\int_{I_{l}}\frac{t^{\frac{1}{2}-H}\dot{\tilde{h}}_{k}(t)-s^{\frac{1}{2}-H}\dot{\tilde{h}}_{l}(s)}{(t-s)^{H+\frac{1}{2}}}ds\right|^{2}dt\triangleq\sum_{k=1}^mQ_{2,k},\label{Q2}\\
Q_{3} & \triangleq\sum_{k=1}^{m}\int_{I_{k}}\left|t^{H-\frac{1}{2}}\int_{a_{k}}^{t}\frac{t^{\frac{1}{2}-H}\dot{\tilde{h}}_{k}(t)-s^{\frac{1}{2}-H}\dot{\tilde{h}}_{k}(s)}{(t-s)^{H+\frac{1}{2}}}ds\right|^{2}dt\triangleq\sum_{k=1}^mQ_{3,k}.\label{Q3}
\end{align}
We now bound the above three terms separately. 

\medskip
\noindent{\it Step 5: Bound for $Q_1$.} In order to bound each $Q_{1,k}$ in the definition \eqref{Q1} of $Q_1$, we just resort to \eqref{def: tilde hm} which allows to write
\begin{align*}
Q_{1,k}&=\int_{I_k}t^{1-2H}\left|\dot{\tilde{h}}_{k}(t)\right|^{2}dt
=\int_{I_{k}}t^{1-2H}\left|\frac{1}{|I_{k}|}\dot{h}_{k}\left(\frac{t-a_{k}}{a_{k+1}-a_{k}}\right)\right|^{2}dt.
\end{align*}
Then the elementary change of variable
$$v=(t-a_k)/(a_{k+1}-a_k)$$
yields
$$Q_{1,k}=\frac{1}{|I_{k}|}\int_{0}^{1}(a_{k}+v|I_{k}|)^{1-2H}\left|\dot{h}_{k}(v)\right|^{2}dv.$$
We now wish to express $Q_1$ in terms of $\bar{h}_m$. To this aim, recall from \eqref{def: hm} that we have
$$Q_{1,k}=\frac{\|u_{k}\|_{\text{CC}}^{2}}{|I_{k}|}\int_{0}^{1}|\varphi_k(v)|^2dv,$$
where we have set
\begin{align}\label{def: phik}
\varphi_k(v)=(a_{k}+v|I_{k}|)^{\frac{1}{2}-H}\dot{\bar{h}}_{k}(v).
\end{align}
We also recall that each $\bar{h}_k$ has bounded second derivative and is supported in $[1/3,2/3]$. In addition, we have seen previously that $\|u_k\|_{\textsc{CC}}=|I_k|$. Thus we have
$$Q_{1,k}=|I_k|\int_{1/3}^{2/3}|\varphi_k(v)|^2dv$$
We now bound the terms $a_k+v|I_k|$ in the definition of $\varphi_k(v)$ uniformly by $\sum_{j=1}^k|I_j|$. We obtain 
\begin{align}
Q_{1,k}\leq C_{H}|I_k|
\left(\sum_{j=1}^k|I_{k}|\right)^{1-2H} \|\dot{\bar{h}}_{k}\|_{\infty;[0,1]}^{2}
\leq C_{H,l_0}\frac{|I_k|}{\left(\sum_{j=1}^m|I_{k}|\right)^{2H-1}},\label{bound for Q1k}
\end{align}
where in the last step we have used the fact that $\dot{\bar{h}}_k$ is bounded.
Therefore, summing relation~\eqref{bound for Q1k} over $k$ and recalling that $Q_1=\sum_{k=1}^mQ_{1,k}$ we get
\begin{align}\label{bound for Q1}
Q_{1} \leq  C_{H,l_0}\frac{\sum_{k=1}^m|I_k|}{\left(\sum_{j=1}^m|I_{k}|\right)^{2H-1}}
  \leq C_{H,l_{0}}\sum_{k=1}^{m}|I_{k}|^{2(1-H)}
\leq C_{H,l_{0}}\lp \sum_{k=1}^{m}|I_{k}| \rp^{2(1-H)}  ,
\end{align}
where we use the relation $2(1-H)<1$ for the last step. This concludes our estimate for the term $Q_1$.

\medskip
\noindent{\it Step 6: Bound for $Q_3$.}  As in the previous step, we first upper bound each term $Q_{3,k}$ separately. To this aim, we perform the same elementary change of variable as for $Q_{1,k}$ above, which allow to express $Q_{3,k}$ in terms of $\bar{h}_k$ instead of $\tilde{h}_k$. We let the patient reader check that we have
\begin{align}
Q_{3,k}&=|I_{k}|^{2(1-H)}\int_{0}^{1}\left|(a_{k}+v_1|I_{k}|)^{H-\frac{1}{2}}
 \cdot\int_{0}^{v_1}\frac{\varphi_k(v_1)-\varphi_k(v_2)}{\left(v_1-v_2\right)^{H+\frac{1}{2}}}dv_2\right|^{2}dv_1,\label{express Q3k}
\end{align}
where we recall that the function $\varphi_k$ has been introduced in \eqref{def: phik}.


Next we express the derivative of each $\varphi_k$ in the following way,
\begin{align}\label{derivative phik}
\frac{d\varphi_{k}}{du}
=
\frac{\ddot{\bar{h}}_{k}(u)}{(a_{k}+u|I_{k}|)^{H-\frac{1}{2}}}
+\left(\frac{1}{2}-H\right)\cdot\frac{|I_{k}| \, \dot{\bar{h}}_{k}(u)}{(a_{k}+u|I_{k}|)^{H+\frac{1}{2}}}.
\end{align}
Hence, since $\bar{h}_{k}$ is supported on $[1/3,2/3]$ and $\|\ddot{\overline{h}}_{k}\|_{\infty;[0,1]}\leq C_{l_{0}},$
it is readily checked from~\eqref{derivative phik} that
\[
\left|\frac{d\varphi_{k}}{du}\right|\leq\frac{C_{H,l_{0}}}{(\sum_{j=1}^k|I_{j}|)^{H-\frac{1}{2}}}.
\]
Plugging this information into \eqref{express Q3k} and bounding all the terms $a_k+v|I_k|$ uniformly by $\sum_{j=1}^k|I_j|$, we end up with
\begin{align*}
Q_{3,k} & \leq C_{H,l_{0}}|I_{k}|^{2(1-H)}\int_{0}^{1}\left(\sum_{j=1}^k|I_{j}|\right)^{2H-1}
 \cdot\left|\int_{0}^{v_1}\frac{dv_2}{(\sum_{j=1}^k|I_{j}|)^{H-\frac{1}{2}}(v_1-v_2)^{H-\frac{1}{2}}}\right|^{2}dv_1\\
 & \leq C_{H,l_{0}}|I_{k}|^{2(1-H)}.
\end{align*}
As for relation \eqref{bound for Q1}, we can now sum the previous bounds over $k$, which yields the following estimate for $Q_3$\,,
\begin{align}\label{bound for Q3}
Q_{3}   \leq C_{H,l_{0}}\sum_{k=1}^{m}|I_{k}|^{2(1-H)}.
\end{align}

\medskip
\noindent{\it Step 7: Bound for $Q_2$.} 
We now turn to the estimation of $Q_{2}$, which is more involved than $Q_{1}$
and $Q_{3}$. We adopt the same strategy as in the previous steps, that is, we handle each $Q_{2,k}$ in \eqref{Q2} separately and we resort to the elementary change of variables
\[
u\triangleq\frac{t-a_{k}}{a_{k+1}-a_{k}},\ \quad\text{and}\ \quad v\triangleq\frac{s-a_{l}}{a_{l+1}-a_{l}}.
\]
We also express the terms $\dot{\tilde{h}}_k$ in \eqref{Q2} in terms of $\dot{\bar{h}}_k$. Thanks to some easy algebraic manipulations, we get
\begin{align}\label{express Q2k}
Q_{2,k}=\int_{0}^{1}\left|\sum_{l=1}^{k-1}\int_{0}^{1}\frac{\frac{\dot{h}_{k}(u)}{|I_{k}|}-\left(\frac{a_{k}+u|I_{k}|}{a_{l}+v|I_{l}|}\right)^{H-\frac{1}{2}}\cdot\frac{\dot{h}_{l}(v)}{|I_{l}|}}{(a_{k}+u|I_{k}|-a_{l}-v|I_{l}|)^{H+\frac{1}{2}}}|I_{l}|dv\right|^{2}|I_{k}|du.
\end{align}
In the expression above, notice that for $l\leq k-1$ we have
$$a_k+u|I_k|-a_l-v|I_l|=q_{k,l}(u,v),$$
where
\begin{equation}\label{d11}
q_{k,l}(u,v)=(1-v)|I_l|+|I_{l+1}|+\cdots+|I_{k-1}|+u|I_k|.
\end{equation}
Therefore, invoking the trivial bounds $a_k+u|I_k|\leq \sum_{j_1=1}^k|I_{j_1}|$ and $a_l+v|I_l|\geq\sum_{j_2=1}^{l-1}|I_{j_2}|$, and bounding trivially the differences by sums, we obtain
\begin{align}\label{bound Q2k}
Q_{2,k}\leq C_{H}\int_{0}^{1}\left|\sum_{l=1}^{k-1}\int_{0}^{1}\frac{\left|\dot{\bar{h}}_{k}(u)\right|+\left(\frac{\sum_{j_1=1}^k|I_{j_1}|}{\sum_{j_2=1}^{l-1}|I_{j_2}|}\right)^{H-\frac{1}{2}}
\cdot\left|\dot{\bar{h}}_{l}(v)\right|}{|q_{k,l}(u,v)|^{H+\frac{1}{2}}}|I_{l}|dv\right|^{2}|I_{k}|du.
\end{align}
In order to obtain a sharp estimate in \eqref{bound Q2k}, we want to take advantage of the fact that $\dot{\bar{h}}_l$ is supported on $[1/3,2/3]$ and therefore avoids the singularities in $u, v$ close to $0$ and $1$. We thus introduce the intervals
\[
J_{1}\triangleq[0,1/3],\ J_{2}\triangleq[1/3,2/3],\ J_{3}\triangleq[2/3,1]
\]
and decompose the expression \eqref{bound Q2k} as follows,
$$Q_{2,k}\leq C_H\sum_{p,q=1}^{3} L_{k,p,q},$$
where the quantity $L_{k,p,q}$ is defined by
\begin{align}\label{def: Lkpq}
L_{k,p,q}\triangleq\int_{J_{p}}\left|\sum_{l=1}^{k-1}\int_{J_{q}}
\frac{\left|\dot{\bar{h}}_{k}(u)\right|+\left(\frac{|I_{1}|+\cdots+|I_{k}|}{|I_{1}|+\cdots+|I_{l-1}|}\right)^{H-\frac{1}{2}}\cdot\left|\dot{\bar{h}}_{l}(v)\right|}
{|q_{k,l}(u,v)|^{H+\frac{1}{2}}}|I_{l}|dv\right|^{2}|I_{k}|du,
\end{align}
for all $p,q=1,2,3.$  Notice again that since all the $\dot{\bar{h}}_k$ are  supported
on $[1/3,2/3]$, the only non-vanishing $L_{k,p,q}$'s are those for which  $p=2$ or $q=2$. Let us show how to handle the terms $L_{k,p,q}$ given by \eqref{def: Lkpq}, according to $q=1, 2$ and $q=3$.

Whenever 
 $q=1$ or $q=2$, regardless of the value of $p,$ it is easily seen from \eqref{d11} that we can  bound $q_{k,l}(u,v)$ from below uniformly by $C\sum_{j=l}^{k-1}|I_j|$. Thanks again to the fact that $\dot{\bar{h}}_k$ is uniformly bounded for all $k$, we obtain
\begin{align*}
 & \frac{\left|\dot{\bar{h}}_{k}(u)\right|+\left(\frac{|I_{1}|+\cdots+|I_{k}|}{|I_{1}|+\cdots+|I_{l-1}|}\right)^{H-\frac{1}{2}}\cdot\left|\dot{\bar{h}}_{l}(v)\right|}{q_{k,l}(u,v)^{H+\frac{1}{2}}}\\
 & \leq\frac{C_{H,l_{0}}}{(|I_{l}|+\cdots+|I_{k-1}|)^{H+\frac{1}{2}}}\cdot\left(\frac{|I_{1}|+\cdots+|I_{k}|}{|I_{1}|+\cdots+|I_{l-1}|}\right)^{H-\frac{1}{2}}.
\end{align*}
Summing the above quantity over $l$ and integrating over $[0,1]$, we end up with
\begin{align*}
L_{k,p,q}\leq C_{H,l_{0}}|I_{k}|\cdot\left(\sum_{l=1}^{k-1}\frac{|I_{l}|}{(|I_{l}|+\cdots+|I_{k-1}|)^{H+\frac{1}{2}}}\cdot\left(\frac{|I_{1}|+\cdots+|I_{k}|}{|I_{1}|+\cdots+|I_{l-1}|}\right)^{H-\frac{1}{2}}\right)^{2}.
\end{align*}
By lower bounding the quantity $|I_1|+\cdots+| I_{l-1}|$ above uniformly by $|I_1|$, we get
\begin{align}\label{bound Lkpq}
L_{k,p,q}\leq C_{H,l_{0}}|I_{k}|\cdot\left(\frac{\sum_{j=1}^k|I_{j}|}{|I_{1}|}\right)^{2H-1}\cdot\left(\sum_{l=1}^{k-1}|I_{l}|^{\frac{1}{2}-H}\right)^{2}.
\end{align}
Recall that we have shown in \eqref{monotone um} that $m\mapsto\|u_m\|_{\text{CC}}$ is a decreasing sequence. Since $\|u_m\|_{\textsc{CC}}=|I_m|$ we can bound uniformly $\sum_{l=1}^{k-1}|I_l|^{1/2-H}$ by $k|I_1|^{1/2-H}$ and $|I_1|^{-1}\sum_{j=1}^k|I_j|$ by $k$. Plugging this information into \eqref{bound Lkpq} we obtain,
\begin{align}\label{Lkpq q12}L_{k,p,q}\leq C_{H,l_0}k^{2H+1}|I_k|^{2(1-H)},\end{align}
which is our bound for $L_{k,p,q}$ when $q\in\{1,2\}$.


Let us now bound $L_{k,p,q}$ for  $q=3$ and $p=2.$ In this case, going back to the definition~\eqref{def: Lkpq} of $L_{k,p,q}$, we have that $\dot{\bar{h}}_l(v)=0$ for $v\in J_q$. Thus we get
\begin{align}
L_{k,2,3} & =\int_{J_{2}}\left|\sum_{l=1}^{k-1}\int_{J_{3}}\frac{\left|\dot{\bar{h}}_{k}(u)\right||I_{l}|dv}{((1-v)|I_{l}|+|I_{l+1}|+\cdots+|I_{k-1}|+u|I_{k}|)^{H+\frac{1}{2}}}\right|^{2}|I_{k}|du\nonumber\\
 & \leq C_{H,l_{0}}\left|\sum_{l=1}^{k-1}\int_{J_{3}}\frac{|I_{l}|dv}{((1-v)|I_{l}|+|I_{l+1}|+\cdots+|I_{k}|)^{H+\frac{1}{2}}}\right|^{2}\cdot|I_{k}|,\label{Lk23}
\end{align}
where we have used the boundedness of $\dot{\bar{h}}_k$ for the second inequality.
We can now evaluate the above $v$-integral  explicitly, which yields
\begin{align*}
 & \int_{J_{3}}\frac{|I_{l}|dv}{((1-v)|I_{l}|+|I_{l+1}|+\cdots+|I_{k}|)^{H+\frac{1}{2}}}\\
= & \frac{1}{\left(H-\frac{1}{2}\right)}\left(\frac{1}{(|I_{l+1}|+\cdots+|I_{k}|)^{H-\frac{1}{2}}}-\frac{1}{\left(\frac{1}{3}|I_{l}|+|I_{l+1}|+\cdots+|I_{k}|\right)^{H-\frac{1}{2}}}\right)
\leq  \frac{C_{H}}{|I_{k}|^{H-\frac{1}{2}}},
\end{align*}
where the second inequality is obtained by lower bounding trivially $|I_{l+1}+\cdots+|I_k|$ by $|I_k|$.
Summing this inequality over $l$ and plugging this information into \eqref{Lk23}, we get
\begin{align}
L_{k,2,3} & \leq C_{H,l_{0}}|I_{k}|\left(\frac{k}{|I_{k}|^{H-\frac{1}{2}}}\right)^{2}\leq C_{H,l_{0}}k^{2H+1}|I_{k}|^{2(1-H)}.\label{eq: I_23 case}
\end{align}
Summarizing our considerations in this step, we have handled the cases $q=1,2$ and $(q,p)=(3,2)$ in \eqref{Lkpq q12} and \eqref{eq: I_23 case} respectively. 
Therefore, we obtain
\begin{align}\label{bound for Q2}
Q_{2}\leq\sum_{k=1}^{m}k^{2H+1}|I_{k}|^{2(1-H)}.
\end{align}

\medskip
\noindent{\it Step 8: Conclusion.} Let us go back to the decomposition \eqref{decomposition to Q} and plug our bounds \eqref{bound for Q1}, \eqref{bound for Q2} and \eqref{bound for Q3} on $Q_1, Q_2$ and $Q_3$.
We get\begin{align*}
\|\tilde{h}^{(m)}\|_{\bar{{\cal H}}([0,a_{m+1}])}^{2} & \leq C_{H}(Q_{1}+Q_{2}+Q_{3})\leq C_{H,l_{0}}\sum_{k=1}^{m}k^{2H+1}|I_{k}|^{2(1-H)}.
\end{align*}
In addition, we have  $|I_{k}|=\|u_{k}\|_{\textsc{CC}}$ and relation (\ref{eq: exponential decay of u_m}) asserts that $k\mapsto\|u_k\|_{\textsc{CC}}$ decays exponentially. Thus we get

\begin{eqnarray}
 &  & \left(\sum_{k=1}^m|I_k|\right)^{2H}\|\tilde{h}^{(m)}\|_{\bar{{\cal H}}([0,a_{m+1}])}^{2}
  \leq C_{H,l_{0}}\left(\sum_{k=1}^{m}|I_{k}|\right)^{2H}
  \left(\sum_{k=1}^{m}k^{2H+1}|I_{k}|^{2(1-H)}\right)\label{eq: estimating d by I_m} \nonumber\\
 &  & \leq C_{H,V,l_{0}}\left(\sum_{k=1}^{m}2^{-\frac{k}{l_{0}}}\right)^{2H}\cdot\left(\sum_{k=1}^{m}k^{2H+1}2^{-\frac{2(1-H)}{l_{0}}k}\right)\cdot|x-y|^{\frac{2}{l_{0}}}\\
 &  & \leq C_{H,V,l_0}|x-y|^{\frac{2}{l_0}},\nonumber
\end{eqnarray}
where we have trivially bounded the partial geometric series for the last step.  Hence the left hand-side of \eqref{eq: estimating d by I_m} converges to a quantity which is lower bounded by $d^2(x,y)$ as $m\to\infty$, thanks to \eqref{eq: estimating the control metric}.  Therefore, letting $m\to\infty$ in \eqref{eq: estimating d by I_m} we have obtained
\begin{equation}\label{d2}
d(x,y)^{2}\leq C_{H,V,l_{0}}|x-y|^{\frac{2}{l_{0}}},
\end{equation}
which concludes our proof of Theorem \ref{thm: local comparison}.

\section{Hypoelliptic case: local lower estimate for the density of solution.}\label{sec: local lower bound}

In this section, we develop the proof of Theorem \ref{thm: local lower estimate} under the uniform hypoellipticity assumption (\ref{eq:unif-hypo-assumption}). As in Section \ref{sec: control distance}, one faces a much more complex situation than in the elliptic case. More specifically, the deterministic Malliavin covariance matrix of $X_{t}^{x}$  will not be uniformly non-degenerate (i.e. Lemma \ref{lem: uniform nondegeneracy of Malliavin matrix} is no longer true). Without this key ingredient, the whole elliptic argument will break down and one needs new approaches. Our strategy follows the main philosophy of Kusuoka-Stroock~\cite{KS87} in the diffusion case. However, as we will see when we develop the analysis, there are several non-trivial challenges in several key steps for the fractional Brownian setting, which require new ideas and methods. In particular, we shall see how to marry Kusuoka-Stroock's approach and the rough paths formalism.

To increase readability, we first summarize the main strategy of the proof. 
Our analysis starts from the existence of the truncated signature of order $l$ for the fractional Brownian motion, as asserted in Proposition \ref{prop:fbm-rough-path}. Specifically, with our notation~\eqref{eq:signature-smooth-x} in mind, we will write $\Gamma_t\equiv S_l(B)_{0,t}\in G^{(l)}$ as
\begin{align}\label{def: signature of B}
\Gamma_t=S_l(B)_{0,t}=1+\sum_{i=1}^l\int_{0<t_1<\cdots<t_i<t}dB_{t_1}\otimes \cdots\otimes dB_{t_i}.
\end{align}
In the sequel we will also use the truncated $\mathfrak{g}^{(l)}$-valued {\it log-signature} of $B$, defined by
\begin{align}\label{def: log signature of B}
U^{(l)}_t\triangleq\log S_l(B)_{0,t}.
\end{align}
Notice that $U^{(l)}_t$ features in relation \eqref{eq: formal Taylor expansion}, and more precisely the process
\begin{equation}\label{eq:def-X-l}
X_l(t,x)\triangleq x+F_l(U_{t}^{(l)},x)
\end{equation}
is the Taylor approximation of order $l$ for the solution of the rough equation \eqref{eq: hypoelliptic SDE} in small time (cf. relation \eqref{eq: formal Taylor expansion 2}).

With those preliminary notation in hand, we decompose the strategy towards the proof of Theorem \ref{thm: local lower estimate}  into three major steps.

\noindent
\textit{Step One}. According to the scaling property of fractional Brownian motion, a precise local lower estimate on the density of $U^{(l)}_t$ can be easily obtained from a general positivity property.

\noindent
\textit{Step Two}. When $l\geq l_0$, the hypoellipticity of the vector fields allows us to obtain a precise local lower estimate on the density of the process $X_l(t,x)$ defined by \eqref{eq:def-X-l} from the estimate on $U^{(l)}_t$ derived in step one.

\noindent
\textit{Step Three}. When $t$ is small, the density of $X_l(t,x)$ is close to the density of the actual solution in a reasonable sense, and the latter inherits the lower estimate obtained in step two. 

The above philosophy was first proposed by Kusuoka-Stroock \cite{KS87} in the diffusion case. However, in the fractional Brownian setting, there are several difficulties when implementing these steps precisely. Conceptually the main challenge arises from the need of respecting the fractional Brownian scaling and the Cameron-Martin structure in each step in order to obtain sharp estimates. More specifically, for Step 1 we need a new idea to prove the positivity for the density of $U_t^{(l)}$ when the Markov property is not available. For Step 2 we rely on the technique we used for proving Theorem \ref{thm: local comparison} in Section \ref{sec: control distance}, which yields sharp estimates for the density of $X_l(t,x)$. In Step 3,  a new ingredient is needed to prove uniformity for an upper estimate for the density of $X_l(t,x)$ with respect to the degree $l$ of expansion. In the following, we develop the above three steps mathematically.

\subsection{Step one: local lower estimate for the signature density of fractional Brownian motion.}\label{section: step 1}

We fix $l\geq1$. Recall that the truncated signature $\Gamma$ is defined by \eqref{def: signature of B}. We will now write $\Gamma$ as the solution of a simple enough rough differential equation. To this aim, let $\{{\rm e}_{1},\ldots,{\rm e}_{d}\}$ be
the standard basis of $\mathbb{R}^{d}.$ By viewing this family as vectors
in $\mathfrak{g}^{(l)}\cong T_{\mathbf{1}}G^{(l)}$, we denote the
associated left invariant vector fields on $G^{(l)}$ by $\{\tilde{W}_{1},\ldots,\tilde{W}_{d}\}$. It is standard
(cf. \cite[Remark 7.43]{FV10}) that $\Gamma_{t}$ satisfies
the following intrinsic stochastic differential equation on $G^{(l)}$:
\begin{equation}\label{eq: SDE for signature}
\begin{cases}
d\Gamma_{t}=\sum_{\alpha=1}^{d}\tilde{W}_{\alpha}(\Gamma_{t})dB_{t}^{\alpha},\\
\Gamma_{0}=\mathbf{1}.
\end{cases}
\end{equation}
Let $U_{t}\triangleq\log\Gamma_{t}\in\mathfrak{g}^{(l)}$ be the truncated
log-signature path, as defined in \eqref{def: log signature of B}. Since $\{\tilde{W}_{1},\ldots,\tilde{W}_{d}\}$ satisfies Hörmander's
condition by the definition of $\mathfrak{g}^{(l)}$, we know that
$U_{t}$ admits a smooth density with respect to the Lebesgue measure
$du$ on $\mathfrak{g}^{(l)}$. Denote this density by $\rho_{t}(u).$ 

Next we show that the density function $\rho_t$ is everywhere strictly positive. This fact will be important for us. In the Brownian case, this was proved in \cite{KS87} using support theorem and the semigroup property (or the Markov property). In the fractional Brownian setting, the argument breaks down although general support theorems for Gaussian rough paths are still available. It turns out that there is a simple neat proof based on Sard's theorem and a general positivity criteria of Baudoin-Nualart-Ouyang-Tindel \cite{BNOT16}. We mention that Baudoin-Feng-Ouyang \cite{BFO19} also has an independent proof of this fact. 

We first recall the classical Sard's theorem, and we refer the reader to \cite{Milnor97} for a beautiful presentation. Let $f:M\rightarrow N$ be a smooth map between two finite dimensional differentiable manifolds $M$ and $N$. A point $x\in M$ is said to be a \textit{critical point} of $f$ if the differential $df_x:T_x M\rightarrow T_{f(x)}N$ is not surjective. A \textit{critical value} of $f$ in $N$ is the image of a critical point in $M$. Also recall that a subset $E\subseteq N$ is a \textit{Lebesgue null set} if its intersection with any coordinate chart has zero Lebesgue measure in the corresponding coordinate space. 

\begin{thm}[Sard's theorem]\label{thm: Sard's theorem}
Let $f:M\rightarrow N$ be a smooth map between two finite dimensional differentiable manifolds. Then the set of critical values of $f$ is a Lebesgue null set in $N$.
\end{thm}
  
We now prove the positivity result announced above, which will be important for our future considerations.

\begin{lem}\label{lem: positivity of rho}
For each $t>0$, the density $\rho_t$ of the truncated signature path $U_t$ is everywhere strictly positive.
\end{lem}
\begin{proof}
We only consider the case when $t=1$. The general case follows from the scaling property (\ref{eq: formula for rho_t}) below. 
Our strategy relies on the fact that $\Gamma_t=\exp(U_t)$ solves equation \eqref{eq: SDE for signature}. In addition, recall our Convention \ref{conv: convention} about the identification of $\mathfrak{g}^{(l)}$ and $G^{(l)}$. Therefore we can get the desired positivity by applying \cite[Theorem 1.4]{BNOT16}. To this aim, recall that the standing assumptions in \cite[Theorem 1.4]{BNOT16} are the following:\\
\\
(i) The Malliavin covariance matrix of $U_t$ is invertible with inverse in $L^p(\Omega)$ for all $p>1$;\\
(ii) The skeleton of equation \eqref{eq: SDE for signature}, defined similarly to \eqref{eq: skeleton ODE}, generates a submersion. More specifically, we need to show that for any $u\in\mathfrak{g}^{(l)}$, there exists $h\in\bch$ such that $\log S_l(h)=u$ and
\begin{align}\label{nondegenerate map}
(d\log S_l)_h: \bch\to \mathfrak{g}^{(l)}\quad \text{is\ surjective,}
\end{align}
where $S_l(h)\triangleq S_l(h)_{0,1}$ is the truncated map.\\
\\
Notice that item (i) is proved in \cite{BFO19}. We will thus focus on condition (ii) in the remainder of the proof.

In order to prove relation \eqref{nondegenerate map} in item (ii) above, let us introduce some additional notation. First we shall write $G\triangleq G^{(l)}$ for the sake of simplicity. Then for all $n\geq 1$ we introduce a linear map
$
H_{n}:(\mathbb{R}^{d})^{n}\rightarrow\bar{{\cal H}}
$ in the following way. Given $y=(y_1,\ldots,y_n)$, the function $H_n(y)$ is defined to be the piecewise linear path obtained by concatenating the vectors $y_1,\ldots,y_n$ successively. We also define a set $\bch_0$ of piecewise linear paths by
\[
\bar{{\cal H}}_{0}\triangleq\bigcup_{n=1}^{\infty}H_{n}\left((\mathbb{R}^{d})^{n}\right)\subseteq\bar{\cal H}.
\]Note that $\bar{\mathcal{H}}_0$ is closed under concatenation, and $S_l(\bar{\cal H}_0)=G$ by the Chow-Rashevskii theorem (cf. Remark \ref{connecting by smooth path}). Now we claim that:\\
\\
(\textbf{P}) For any $g\in G$, there exists $h\in\bar{\cal H}_0$ such that $S_l(h)=g$ and the differential $(dS_l)_h|_{\bar{\cal H}_0}:\bar{\cal H}_0\rightarrow T_g G$ is surjective. \\

Note that the property  (\textbf{P}) is clearly stronger than the original desired claim \eqref{nondegenerate map}. To prove (\textbf{P}), let $\cal P$ be the set of elements in $G$ which satisfy (\textbf{P}). We first show that $\cal P$ is either $\emptyset$ or $G$.  The main idea behind our strategy is that if there exists $g_0\in\cal{P}$, such that $(dS_l)_{h_0}$ is a submersion for some $h_0\in\bch_0$ satisfying $S_l(h_0)=g_0$, then one can obtain every point $g\in G$ by a left translation $L_a$, since $dL_a$ is an isomorphism. To be more precise, suppose that $g_{0}\in G$ is an element satisfying (\textbf{P}). By definition, there exists a path $h_{0}\in\bar{{\cal H}}_{0}$ such
that $S_{l}(h_{0})=g_{0}$ and $(dS_{l})_{h_{0}}|_{\bar{{\cal H}}_{0}}$
is surjective. Now pick a generic element $a\in G$ and choose a path
$\alpha\in\bar{{\cal H}}_{0}$ so that $S_{l}(\alpha)=a$. Then $S_{l}(\alpha\sqcup h_{0})=a\otimes g_{0}.$
We want to show that $(dS_{l})_{\alpha\sqcup h_{0}}:\bar{{\cal H}}_{0}\rightarrow T_{a\otimes g_{0}}G$
is surjective. For this purpose, let $\xi\in T_{a\otimes g_{0}}G$
and set 
\[
\xi_{0}\triangleq dL_{a^{-1}}(\xi)\in T_{g_{0}}G.
\]
By the surjectivity of $(dS_{l})_{h_{0}}|_{\bar{{\cal H}}_{0}},$
there exists $\gamma\in\bar{{\cal H}}_{0}$ such that $(dS_{l})_{h_{0}}(\gamma)=\xi_{0}.$
It follows that, for $\varepsilon>0$ we have 
\[
S_{l}(\alpha\sqcup(h_{0}+\varepsilon\cdot\gamma))=a\otimes S_{l}(h_{0}+\varepsilon\cdot\gamma).
\]
By differentiation with respect to $\varepsilon$ at $\varepsilon=0$, we obtain that
\[
(dS_{l})_{\alpha\sqcup h_{0}}(0\sqcup\gamma)=(dL_{a})_{S_{l}(h_{0})}\circ(dS_{l})_{h_{0}}(\gamma)=(dL_{a})_{g_{0}}(\xi_{0})=\xi.
\]
Therefore, $(dS_{l})_{\alpha\sqcup h_{0}}|_{\bar{{\cal H}}_{0}}$
is surjective. Since $a$ is arbitrary, we conclude that if ${\cal P}$
is non-empty, then ${\cal P}=G$.

To complete the proof, it remains to show that $\cal P\neq\emptyset$. This will be a simple consequence of Sard's theorem. Indeed, for each $n\geq1$,  define
\begin{align}\label{def: fn}
f_{n}\triangleq S_{l}\circ H_{n}:(\mathbb{R}^{d})^{n}\rightarrow G,
\end{align}
where we recall that $H_n(y)$ is the piecewise linear path obtained by concatenating $y_1,..., y_n$.
The map $f_n$ is simply given by \[
f_{n}(y_{1},\ldots,y_{n})=\exp(y_{1})\otimes\cdots\otimes\exp({y_{n}}),
\]
where we recall that the exponential maps is defined by \eqref{eq:def-exp-on-T-l}. It is readily checked that $f_n$ is a smooth map.
According to Sard's theorem (cf. Theorem \ref{thm: Sard's theorem}), the set of critical values of $f_n$, denoted as $E_n$, is a Lebesgue null set in $G$. It follows that $E\triangleq\cup_{n=1}^{\infty}E_{n}$ is also a Lebesgue null set in $G$. We have thus obtained that, \[
G\backslash E=\left(\bigcup_{n=1}^{\infty}f_{n}((\mathbb{R}^{d})^{n})\right)\backslash E\neq\emptyset,
\]
where the first equality is due to the fact that $S_l(\bch_0)=G$ by the Chow-Rashevskii theorem. 
Pick any element $g\in G\backslash E$. Then for some $n\geq1$, we have $g\in f_n((\mathbb{R}^d)^n)\backslash E_n$. In particular, there exists $y\in (\mathbb{R}^d)^n$ such that $f_n(y)=g$ and $(df_n)_y$ is surjective.  We claim that $g\in\cal P$ with $h\triangleq H_n(y)\in\bar{\cal H}_0$ being the associated path.  Indeed, it is apparent that $S_l(h)=g$. In addition, let $\xi\in T_g G$ and $w\in (\mathbb{R}^d)^n$ be such that $(df_n)_y(w)=\xi$. The existence of $w$ follows from the surjectivity of $(df_n)_y$. Since $H_n$ is linear, we obtain that \begin{align*}
(dS_{l})_{h}(H_{n}(w)) & =\left.\frac{d}{d\varepsilon}\right|_{\varepsilon=0}S_{l}(H_{n}(y)+\varepsilon\cdot H_{n}(w))
  =\left.\frac{d}{d\varepsilon}\right|_{\varepsilon=0}S_{l}(H_{n}(y+\varepsilon\cdot w))\\
 & =\left.\frac{d}{d\varepsilon}\right|_{\varepsilon=0}f_{n}(y+\varepsilon\cdot w)
  =(df_{n})_{y}(w)
  =\xi.
\end{align*}Therefore, the pair $(h,g)$ satisfies property (\textbf{P}) and thus $\cal P$ is non-empty.
\end{proof}

\begin{rem}
Some mild technical care is needed in the above proof which we have postponed until now so not to distract the reader from getting the key idea in the proof. One point is that, Theorem 1.4 in \cite{BNOT16} was stated for SDEs in which the vector fields are of class $C_b^\infty$. Nevertheless, that theorem relies on properties of the skeleton of a non-degenerate random variable $F$, which are clearly satisfied for the truncated signature path $\Gamma_t$ defined by \eqref{def: signature of B}.\end{rem}

\begin{rem}
  Another point is that, when $H>1/2$ it is not clear whether $\bar{\cal H}$ contains the space of piecewise linear paths. It is though obvious from $\bar{\cal H}=I_{0+}^{H+1/2}(L^2([0,1]))$ that it contains all smooth paths. One simple way to fix this issue is to reparametrize the piecewise linear path $y_1\sqcup\cdots\sqcup y_n$ in a way depending only on $n$, so that the resulting path is smooth but the trajectory remains unchanged. This will not change the truncated signature as it is invariant under reparametrization. For instance, one can define $H_n(y)$ in a way that on $[(i-1)/n,i/n]$ it is given by \[
H_{n}(y)_{t}=y_{1}+\cdots y_{i-1}+\left(\int_{0}^{t}\eta_{i}(s)ds\right)y_{i},
\]where $\eta_i$ is a positive smooth function supported on $[\frac{i-1}{n}+\frac{1}{3n},\frac{i}{n}-\frac{1}{3n}]$ with $\int_{\frac{i-1}{n}}^{\frac{i}{n}}\eta_{i}(t)dt=1$.   See also \cite{BFO19} for a direct strategy.
\end{rem}

Essentially the same amount of effort allows us to adapt the argument in the proof of Lemma \ref{lem: positivity of rho} to establish the general positivity result for hypoelliptic SDEs as stated in Theorem \ref{prop: positivity-intro} which is of independent interest. This complements the result of \cite[Theorem 1.4]{BNOT16} by affirming that Hypothesis 1.2 in that theorem is always verified under hypoellipticity.  

\begin{proof}[Proof of Theorem \ref{prop: positivity-intro}]
Without loss of generality we only consider $t=1$.  Continuing to denote by $\Phi_t(x;h)$ the skeleton of equation \eqref{eq: hypoelliptic SDE}, defined by \eqref{eq: skeleton ODE}, 
let $F:\bar{{\cal H}}\rightarrow\mathbb{R}^{N}$ be the end point
map defined by $F(h)\triangleq\Phi_{1}(x;h)$.  As in the proof of Lemma \ref{lem: positivity of rho}, we wish to check the assumptions of \cite[Theorem 1.4]{BNOT16}. Recall that this means that we should prove that the Malliavin covariance matrix of $X_1$ admits an inverse in $L^p(\Omega)$, and that \eqref{nondegenerate map} holds for the map $F$. Furthermore, under our standing assumptions, the fact that the Malliavin covariance matrix of $X_1$ is in $L^p(\Omega)$ is already proved in \cite{CHLT15}. We will thus focus on an equivalent of condition \eqref{nondegenerate map} in the remainder of the proof. Summarizing our considerations so far, we wish to prove that for any $y\in\mr^N$ there exists $h\in\bch$ such that 
\begin{align}\label{surjectivity general case}
F(h)=y,\quad\text{and}\ \ \ (dF)_h: \bch\to\mr^N \ \ \textnormal{is surjective.}
\end{align}
Along the same lines as in the proof of Lemma \ref{lem: positivity of rho}, we define ${\cal P}$ to be the set of points $y\in\mr^N$ satisfying \eqref{surjectivity general case} for some $h\in\bar{\cal H}$. We first show that ${\cal P}$ is non-empty which then implies $\mathcal{P}=\mathbb{R}^N$ again by a translation argument.

To show that $\cal P$ is non-empty, we first define $H_{n}:(\mathbb{R}^{d})^{n}\rightarrow\bar{{\cal H}}$
and $\bar{{\cal H}}_{0}\subseteq\bar{{\cal H}}$ in the same way as
in the proof of Lemma \ref{lem: positivity of rho}. Also define a map $F_n$ by
\[
F_{n}\triangleq F\circ H_{n}:(\mathbb{R}^{d})^{n}\rightarrow\mathbb{R}^{N}.
\]
According to Sard's theorem, the set of critical values of $F_{n},$
 again denoted as $E_{n}$, is a Lebesgue null set in $\mathbb{R}^{N},$
and so is $E\triangleq\cup_{n}E_{n}$. 

Next consider a given $q\in\mr^N$. Thanks to the hypoellipticity assumption \eqref{eq:unif-hypo-assumption}, we can equip a neighborhood $U_q$ of $q$ with a sub-Riemannian metric, {by requiring that a certain subset of $\{V_1,...,V_d\}$ is an orthonormal frame near $q$}. Then according to the Chow-Rashevskii theorem (cf. \cite{Montgomery02},
Theorem 2.1.2), every point in $U_{q}$ is reachable from $q$ by
a horizontal path. And if one examines the proof of the theorem in
Section 2.4 of \cite{Montgomery02} carefully, this horizontal path
is controlled by a piecewise linear path in $\mathbb{R}^{d}$, i.e.
$U_{q}\subseteq\cup_{n}\Phi_{1}(q;H_{n}((\mathbb{R}^{d})^{n}))$.
Now for given $y\in\mathbb{R}^{N},$ choose an arbitrary continuous
path $\gamma$ joining $x$ to $y$. By compactness, we can cover
the image of $\gamma$ by finitely many open sets of the form $U_{q_{i}}$
such that $U_{q_{i}}\cap U_{q_{i+1}}\neq\emptyset$ for all $i$ where
$q_{i}\in{\rm Im}(\gamma).$ It follows that $y$ can be reached from
$x$ by a horizontal path controlled by a piecewise linear path in
$\mathbb{R}^{d}.$ In other words, we have $y\in F_{n}((\mathbb{R}^{d})^{n})$
for some $n.$ This establishes the property that $\mathbb{R}^{N}=\cup_{n}F_{n}((\mathbb{R}^{d})^{n})$.

Now the same argument as in the proof of Lemma \ref{lem: positivity of rho} allows us to conclude that 
\[
\mathbb{R}^{N}\backslash E=\bigcup_{n=1}^{\infty}F_{n}((\mathbb{R}^{d})^{n})\backslash E\subseteq\mathcal{P},
\]showing that $\cal P$ is non-empty since $E$ is a Lebesgue null set.

Finally, we show that $\mathcal{P}=\mathbb{R}^N$. To this aim, first note that, for any $h_0,\gamma,\alpha\in\bar{\cal H}$ and $\varepsilon>0$, we have \[
\Phi_{1}(x;(h_{0}+\varepsilon\cdot\gamma)\sqcup\alpha)=\Phi_{1}\left(\Phi_{1}(x;h_{0}+\varepsilon\cdot\gamma);\alpha\right),
\]where paths are always assumed to be parametrized on $[0,1]$. Therefore, by differentiating with respect to $\varepsilon $ at $\varepsilon=0$, we obtain that \[
(dF)_{h_{0}\sqcup\alpha}(\gamma\sqcup0)=J_{1}(F(h_{0});\alpha)\circ(dF)_{h_{0}}(\gamma),
\]where recall that $J_t(\cdot;\cdot)$ is the Jacobian of the flow $\Phi_t$. This shows that 
\begin{equation}\label{eq:Surj}
(dF)_{h_{0}\sqcup\alpha}=J_{1}(F(h_{0});\alpha)\circ(dF)_{h_{0}}.
\end{equation}

Now pick any fixed $y_0\in\cal P$ with an associated $h_0\in\bar{\cal H}$ satisfying (\ref{surjectivity general case}). For any $\eta\in\mathbb{R}^N$, choose $\alpha\in\bar{\cal H}$ such that $F(\alpha)=\eta.$ Then $F(h_0\sqcup\alpha)=y+\eta$ and the surjectivity of $(dF)_{h_0\sqcup\alpha}$ follows from (\ref{eq:Surj}), the surjectivity of $(dF)_{h_0}$ and the invertibility of the Jacobian. In particular, $y+\eta\in\cal P$. Since $\eta$ is arbitrary, we conclude that $\mathcal{P}=\mathbb{R}^N$.

\end{proof}

\begin{rem}
A general support theorem for hypoelliptic SDEs allows one to show that the support of the density $p_t(x,y)$ is dense. In the diffusion case, together with the semigroup property
 \[
p(s+t,x,y)=\int_{\mathbb{R}^{N}}p(s,x,z)p(t,z,y)dz
\]one immediately sees that $p(t,x,y)$ is everywhere strictly positive. This argument clearly breaks down in the fractional Brownian setting.
\end{rem}

{Finally, we present the main result in this part which gives a precise local lower estimate for the density $\rho_t(u)$. }In order to get this estimate a first idea wold be to use the stochastic differential equation for  $U_t$, which is obtained by taking logarithm in relation \eqref{eq: SDE for signature}. Instead of following this strategy, we will resort to some more elementary scaling properties, which stems from the left invariance of the vector fields $\tilde{W}_\alpha$ in \eqref{eq: SDE for signature}.
This is why dealing with $U_t$ is considerably easier than studying the solution to the general SDE (\ref{eq: hypoelliptic SDE}).

\begin{prop}\label{prop: step one}
For each $M>0,$ define $\beta_{M}\triangleq\inf\left\{ \rho_{1}(u):\|u\|_{\textsc{CC}}\leq M\right\}$. Then $\beta_M$ is strictly positive and for all $(u,t)\in\mathfrak{g}^{(l)}\times(0,1)$ with $\|u\|_{\textsc{CC}}\leq Mt^{H},$ 
 we have 
\begin{align}\label{signature density lower bound}
\rho_{t}(u)\geq\beta_{M}t^{-H\nu},
\end{align}
where the constant $\nu$ is given by $\nu\triangleq\sum_{k=1}^{l}k\dim{\cal L}_{k},$ and ${\cal L}_k$ is introduced in Definition \ref{def:free-algebra}.
\end{prop}
\begin{proof}
First observe that the strict positivity of $\beta_M$ is an easy consequence of Lemma \ref{lem: positivity of rho} plus the fact that the set $\{u\in\frak{g}^{(l)}; \|u\|_{\textsc{CC}}\leq M\}$ is compact. Next recall that if $\tilde{W}$ is a left invariant vector field on $G^{(l)}$, the push-forward of $\tilde{W}$ by $\delta_\lambda$ (denoted by $(\delta_\lambda)_*\tilde{W}$) is defined by
\begin{align*}
\left[(\delta_\lambda)_*\tilde{W}\right](\delta_\lambda u)=\left.\frac{d}{d\varepsilon}\right|_{\varepsilon=0}\delta_\lambda (u\otimes\exp(\varepsilon\cdot\tilde{W}(\mathbf{1}))).
\end{align*}
Using this definition, it is readily checked that 
\begin{align}\label{scaling property W}
(\delta_{\lambda})_{*}\tilde{W}=\lambda\cdot \tilde{W}.
\end{align}
Therefore, applying a change of variable formula to \eqref{eq: SDE for signature} with $f(\Gamma_t)=\delta_\lambda \Gamma_t$ and resorting to \eqref{scaling property W} we obtain
\begin{align}
d(\delta_{\lambda}\Gamma_{t}) & =\sum_{\alpha=1}^{d}\left((\delta_{\lambda})_{*}\tilde{W}_{\alpha}\right)(\delta_{\lambda}\Gamma_{t})dB_{t}^{\alpha}
  =\lambda\sum_{\alpha=1}^{d}\tilde{W}_{\alpha}(\delta_{\lambda}\Gamma_{t})dB_{t}^{\alpha}.\label{eq: d.e. for dilated path}
\end{align}
On the other hand, for $\lambda>0$ set $\Gamma_t^\lambda\triangleq\Gamma_{\lambda^{1/H}t}$. Then we have the following series of identities:
$$d\Gamma^\lambda_t=\lambda^{1/H}d\Gamma_{\lambda^{1/H}t}=\lambda^{1/H}\sum_{\alpha=1}^d\tilde{W}_\alpha(\Gamma^\lambda_t)dB^\alpha_{\lambda^{1/H}t}.$$
Therefore, setting $B_t^{\alpha,\lambda}=B^\alpha_{\lambda^{1/H}t}$, we get
\begin{equation}
d\Gamma^\lambda_{t}  =\sum_{\alpha=1}^{d}\tilde{W}_{\alpha}(\Gamma^\lambda_{t})d(B_{t}^{\alpha,\lambda})  
=
\lambda\sum_{\alpha=1}^{d}\tilde{W}_{\alpha}(\Gamma^\lambda_{t})d\left(\lambda^{-1} B^{\alpha,\lambda}_{t}\right)_t.\label{eq: d.e. for rescaled path}
\end{equation}
Now observe that the usual scaling for the fractional Brownian motion yields
$$\{\lambda^{-1}B^\lambda_t; t\geq 0\}\stackrel{d}{=}\{B_t; t\geq 0\}.$$
We have thus obtained that $\Gamma^{\lambda}\stackrel{d}{=}{\hat{\Gamma}}^\lambda$, where $\hat{\Gamma}^\lambda$ solves the system
\begin{align}
d\hat{\Gamma}^\lambda_{t}  =\lambda\sum_{\alpha=1}^{d}\tilde{W}_{\alpha}(\Gamma^\lambda_{t})dB^{\alpha}_t.\label{eq for gamma hat}
\end{align}
Comparing \eqref{eq for gamma hat} and \eqref{eq: d.e. for dilated path},
we conclude that $\delta_{\lambda}\Gamma_{\cdot}\stackrel{{\rm law}}{=}\Gamma_{\lambda^{\frac{1}{H}}\cdot}.$
If we define $Q_{t}$ to be the law of $U_{t}$ on $\mathfrak{g}^{(l)}$,
it follows that $Q_{s}\circ\delta_{\lambda}^{-1}=Q_{\lambda^{\frac{1}{H}}s}$ for all $s>0$.
In particular, by setting $s=1$ and $\lambda=t^{H},$ we obtain that
\begin{equation}\label{eq: scaling property of U_t}
Q_{t}=Q_{1}\circ\delta_{t^{H}}^{-1}.
\end{equation}

Now suppose that $\rho_{t}(u)$ is the density of $Q_{t}$ with respect
to the Lebesgue measure $du$ on $\frak{g}^{(l)}$. Then for any $f\in C_{b}^{\infty}(\mathfrak{g}^{(l)}),$
we have 
\begin{align*}
\int_{\mathfrak{g}^{(l)}}f(u)\rho_{t}(u)du & =\int_{\mathfrak{g}^{(l)}}f(u)Q_{t}(du)
  =\int_{\mathfrak{g}^{(l)}}f(\delta_{t^{H}}u)Q_{1}(du)\\
 & =\int_{\mathfrak{g}^{(l)}}f(\delta_{t^{H}}u)\rho_{1}(u)du
  =\int_{\mathfrak{g}^{(l)}}t^{-H\nu}f(u)\rho_{1}(\delta_{t^{-H}}u)du,
\end{align*}
where the equality follows from the change of variables $u\leftrightarrow\delta_{t^{-H}}u$
and the fact that $du\circ\delta_{t^{H}}^{-1}=t^{-H\nu}du$ (cf. relation \eqref{eq:dilation-lebesgue-on-cal-G}). Therefore,
we conclude that 
\begin{equation}\label{eq: formula for rho_t}
\rho_{t}(u)=t^{-H\nu}\rho_{1}(\delta^{-1}_{t^{H}}u),\ \ \text{for\ all}\ \ (u,t)\in\mathfrak{g}^{(l)}\times(0,1),
\end{equation}
from which our result \eqref{signature density lower bound} follows.
\end{proof}

\subsection{Step two: local lower estimate for the density of the Taylor approximation process.}\label{sec: step two}

Recall that  according to our definition \eqref{eq:def-rough-eq}, $X_l(t,x)=x+F_l(U_t^{(l)},x)$ is the Taylor approximation process of order $l$ for the actual solution of the SDE (\ref{eq: hypoelliptic SDE}). Due to hypoellipticity, it is natural to expect that when $l\geq l_0$, $F_l$ is "non-degenerate" in certain sense. In addition, $X_l(t,x)$ should have a density, and a precise local lower estimate for the density should follow from Proposition \ref{prop: step one} in Step One,  combined with the "non-degeneracy" of $F_l$. Here the main subtlety and challenge lies in finding a way of respecting the fractional Brownian scaling and Cameron-Martin structure so that the estimate we obtain on $X_l(t,x)$ is sharp. In this part, we always fix $l\geq l_0$.

\paragraph*{I. Non-degeneracy of $F_l$ and a disintegration formula.}
\label{sec: non-deg}
\addcontentsline{toc}{subsubsection}{\nameref{sec: non-deg}}

$\ $\\
\\
We first review a basic result in \cite{KS87} on the (local) non-degeneracy of $F_l$, which then allows us to obtain a formula for the (localized) density of $X_l(t,x)$ by disintegration. This part is purely analytic and does not rely on the structure of the underlying process.

Let $JF_l(u,x):\mathfrak{g}^{(l)}\rightarrow\mathbb{R}^N$ be the Jacobian of $F_l$ with respect to $u$. Since $\mathfrak{g}^{(l)}$ has a canonical Hilbert structure induced from $T^{(l)}(\mathbb R^d)$, we can also consider the adjoint map $JF_l(u,x)^*:\mathbb{R}^N\rightarrow\mathfrak g^{(l)}$.  The non-degeneracy of $JF_l$ is summarized in the following lemma.

\begin{lem}\label{lem: nondegeneracy of F}Let $F_l$ be the approximation map given in Definition \ref{def: Taylor approximation function} and Let $JF_l(u,x): \frak{g}^{(l)}\to \mr^N$ be its Jacobian. Then there exists a constant $c>0$ depending only on $l_0$ and the vector fields, such that \[
JF_l(0,x)\cdot JF_l(0,x)^{*}\geq c\cdot\mathrm{Id}_{\mathbb{R}^{N}}
\]for all $l\geq l_0$ and $x\in \mathbb R^N$.
\end{lem}

\begin{proof}[Sketch of proof]
This is Lemma 3.13 of \cite{KS87}. Because of its importance, we outline the idea of the proof so that one may see how the hypoellipticity property comes into play.
Recall the definitions of $\mathrm{e}_{(\alpha)},\mathrm{e}_{[\alpha]},V_{(\alpha)}$ from equation \eqref{eq:def-basis-algebras} and $V_{[\alpha]}$ from Section \ref{sec: main results}.
Define a linear map $\Xi:T_0^{(l)}\rightarrow C^\infty_b(\mathbb{R}^N;\mathbb{R}^N)$ by setting $\Xi(\mathrm {e}_{(\alpha)})\triangleq V_{(\alpha)}$ for each $\alpha\in\mathcal{A}_1(l)$. A crucial property  is that $\Xi$ respects Lie brackets, i.e. 
\begin{equation}\label{eq: Xi respects Lie brackets}\Xi(\mathrm{e}_{[\alpha]})=V_{[\alpha]},\ \ \ \forall\alpha\in\mathcal{A}_1(l).
\end{equation}
Now let $\{u_\mu:1\leq\mu\leq m_l\}$ be an orthonormal basis of $\mathfrak{g}^{(l)}$, where $m_l\triangleq\dim\mathfrak{g}^{(l)}$, and set $V_\mu=\Xi(u_\mu)$. Based on (\ref{eq: Xi respects Lie brackets}), it is not hard to show that 
\[
{\rm Span}\{V_{[\alpha]}(x):\alpha\in{\cal A}_{1}(l)\}={\rm Span}\{V_{\mu}(x):1\leq\mu\leq m_{l}\},
\]
for each $x\in\mathbb{R}^N$. 
Let us now relate these notions to the non-degeneracy of $JF_l$. To this aim, taking Definition \ref{def: Taylor approximation function} into account, it is easily seen that 
$$JF_l(0,x)(u)=\sum_{\alpha\in{\cal A}_1(l)}V_{(\alpha)}(x)u^\alpha.$$
In particular we have $JF_l(0,x)(e_{(\alpha)})=V_{(\alpha)}(x)$. Hence invoking relation \eqref{eq: Xi respects Lie brackets} we end up with
$$JF_l(0,x)(e_{[\alpha]})=V_{[\alpha]}(x).$$
By definition of our orthonormal basis $\{u_\mu; 1\leq \mu\leq m_l\}$ we thus get
 \[
JF_l(0,x)\cdot JF_l(0,x)^{*}=\sum_{\mu=1}^{m_{l}}V_{\mu}(x)\otimes V_{\mu}(x).
\] Therefore, the non-degeneracy of $JF_l(0,x)$ follows from the hypoellipticity assumption \eqref{eq:unif-hypo-assumption} of the vector fields $V_\alpha$. \end{proof}

An immediate corollary of Lemma \ref{lem: nondegeneracy of F} is the following.

\begin{cor}\label{cor: local submersion}
Given $l\geq l_0,$ there exists $r>0$ depending on $l$ and the vector fields, such that $\det(JF_l(u,x)\cdot JF_l(u,x)^*)$ is uniformly positive on $\{u\in\mathfrak{g}^{(l)}:\|u\|_\mathrm{HS}< r\}\times\mathbb{R}^N$. In particular, the map
\[
\{u\in\mathfrak{g}^{(l)}:\|u\|_{\textsc{HS}}<r\}\rightarrow\mathbb{R}^{N},\qquad u\mapsto x+F_{l}(u,x),
\]is a submersion in the sense of differential geometry. 
\end{cor}

\begin{rem}\label{rem: consistent choice of r}
Note that the map $F_l$ and the constant $r$ in Corollary \ref{cor: local submersion} depend on $l$. For technical reasons, we will assume that $r$ is chosen (still depending on $l$) so that for all $l_0\leq l'\leq l$, the map $JF_{l'}(\pi^{(l')}(u),x)$ has full rank whenever $(u,x)\in\mathfrak{g}^{(l)}\times\mathbb{R}^N$ with $\|u\|_{\mathrm{HS}}<r$, where $\pi^{(l')}:\mathfrak{g}^{(l)}\rightarrow\mathfrak{g}^{(l')}$ is the canonical projection. This property will be used in the proof of Lemma \ref{lem: upper estimate for approximating density} in Step Three below.
\end{rem}

Now let $r$ be the constant given in Remark \ref{rem: consistent choice of r} . It is standard from differential geometry that for each $x\in\mathbb{R}^N$ and $y\in\{x+F_l(u,x):\|u\|_\mathrm{HS}<r\}$, the "bridge space"
\begin{align}\label{Mxy}M_{x,y}\triangleq \{u\in\mathfrak{g}^{(l)}:\|u\|_\mathrm{HS}<r\ \mathrm{and}\ x+F_l(u,x)=y\}\end{align} is a submanifold of $\{u\in\mathfrak{g}^{(l)}:\|u\|_{\mathrm{HS}}<r\}$ with dimension $\dim \mathfrak{g}^{(l)}-N$. In addition, since both of $\mathfrak{g}^{(l)}$ and $\mathbb{R}^N$ are oriented Riemannian manifolds, we know from differential topology that $M_{x,y}$ carries a natural orientation and hence a volume form which we denote as $m_{x,y}$. The following result is the standard disintegration formula in Riemannian geometry (cf. Appendix for a proof).

\begin{prop}\label{th: disintegration}For any $\varphi\in C_c^\infty(\{u\in\mathfrak{g}^{(l)}:\|u\|_\mathrm{HS}<r\})$, we have 
\begin{equation}\label{eq: disintegration}
\int_{\mathfrak{g}^{(l)}}\varphi(u)du=\int_{\mathbb{R}^{N}}dy\int_{M_{x,y}}K(u,x)\varphi(u)m_{x,y}(du),
\end{equation}where the kernel $K$ is given by
\begin{align}\label{kernel in disintegration}
K(u,x)\triangleq\left(\det(JF_{l}(u,x)\cdot JF_{l}(u,x)^{*})\right)^{-\frac{1}{2}},
\end{align}
and we define $m_{x,y}\triangleq0$ if $M_{x,y}=\emptyset.$
\end{prop}

The disintegration formula (\ref{eq: disintegration}) immediately leads to a formula for the (localized) density of the Taylor approximation process $X_l(t,x)$. We summarize this fact in the folloiwng proposition.

\begin{prop}\label{prop: representation of density} Let $\eta\in C_c^\infty(\{u\in\mathfrak{g}^{(l)}:\|u\|_\mathrm{HS}<r\})$ be a bump function so that $0\leq\eta\leq1$ and $\eta=1$ when $\|u\|_\mathrm{HS}<r/2$, where $r$ is the constant featuring in Proposition \ref{th: disintegration}. Define $\mathbb{P}^\eta_l(t,x,\cdot)$ to be the measure \[
\mathbb{P}_{l}^{\eta}(t,x, A)\triangleq\mathbb{E}\left[\eta(U_{t}){\bf 1}_{\{X_{l}(t,x)\in A\}}\right],\ \ \ A\in{\cal B}(\mathbb{R}^{N}),
\]where $U_t=\log \Gamma_t$, $\Gamma_t$ is defined by \eqref{eq: SDE for signature} and $X_{l}(t,x)=x+F_{l}(U_{t},x)$ is the approximation given by \eqref{eq:def-rough-eq}.
The measure $\mathbb{P}_{l}^{\eta}(t,x,\cdot)$ is absolutely continuous with
respect to the Lebesgue measure, and its density is given by 
\begin{equation}\label{eq: formula for p^eta_l}
p_{l}^{\eta}(t,x,y)\triangleq\int_{M_{x,y}}\eta(u)K(u,x)\rho_{t}(u)m_{x,y}(du).
\end{equation}
where $\rho_t$ is the density of $U_t$ alluded to in Lemma \ref{lem: positivity of rho} and $K$ is given by \eqref{kernel in disintegration}.
\end{prop}
\begin{proof}
Since $X_l((t,x)=F_l(U_t,x)$, we have
\[
\mathbb{P}_l^\eta(t,x,A)=\mathbb{E}\left[\eta(U_{t}){\bf 1}_{\{F_l(U_t,x)\in A\}}\right],
\]
and one can recast this expression in terms of the density of $U_t$, namely
\[
\mathbb{P}_l^\eta(t,x,A)=\int_{\frak{g}^{(l)}}\eta(u){\bf 1}_{\{F_l(u,x)\in A\}}\rho_t(u)du.
\]
Then our conclusion \eqref{eq: formula for p^eta_l} stems from a direct application of \eqref{eq: disintegration}.
\end{proof}

\paragraph*{II. Estimating the volume form $m_{x,y}$.}
\label{sec: volume}
\addcontentsline{toc}{subsubsection}{\nameref{sec: volume}}

$\ $\\
\\
To obtain a sharp lower estimate on $p_{l}^{\eta}(t,x,y)$ from 
formula \eqref{eq: formula for p^eta_l} and the lower estimate of $\rho_t(u)$ given by \eqref{signature density lower bound}, one needs to estimate the volume form $m_{x,y}$ precisely.
For this purpose, we apply a change of variables $S_{x,y}:M_{x,y}\rightarrow M_{x,x}$
introduced in \cite{KS87}. In contrast to the Brownian motion case, our main challenge lies in respecting
the Cameron-Martin structure in order to obtain sharp estimates. In this section, we will
use the technique implemented in the proof of Theorem \ref{thm: local comparison}
and pathwise estimates for Cameron-Martin paths to achieve this.

The construction of the function $S_{x,y}$ alluded to above is based on the simple idea that in order to transform a loop $\alpha$ from $x$ to $x$ into a path from $x$ to $y$, we just concatenate a generic path from $x$ to $y$ to the loop $\alpha$. However, we are looking at this construction from the Taylor approximation point of view. 
More specifically, we define the operation $\times$ to be the multiplication induced from $G^{(l)}$ through the exponential map, namely
\begin{equation}\label{multiplication on Lie algebra}
v\times u\triangleq\log( \exp(v)\otimes \exp(u)),\ \ \ v,u\in\mathfrak{g}^{(l)}.
\end{equation}
As mentioned above, we would ideally like the operation $\times$ to transform elements $v\in M_{x,x}$ into elements $v\times u\in M_{x,y}$ where $u\in M_{x,y}$ is fixed. However, { due to the fact that $F_{l}$ is only an approximation of the flow $\Phi$}, this property will in general not be fulfilled. Nevertheless, if $\|u\|_{\text{HS}}$ and $\|v\|_{\text{HS}}$ are small enough the product $v\times u$ is close to an element of $M_{x,y}$, so that the function $\Psi_l$ of Lemma~\ref{lem: the Psi function} can be applied. We summarize those heuristic considerations in the following lemma (cf. \cite[Lemma 3.23]{KS87}), which gives the precise
construction of the change of variables $S_{x,y}:M_{x,y}\rightarrow M_{x,x}$.

\begin{lem}\label{lem: construction of S_xy}
{For $x\in\mr^N$ and $h\in\bch$, we set $y=\Phi_1(x;h)$ where recall that $\Phi$ is the flow defined by \eqref{eq: skeleton ODE}. Let $r$ be the constant arising in Lemma \ref{lem: the Psi function}. For $u\in M_{x,y}$ and $v\in M_{x,x}$, recall that $v\times u$ is defined by (\ref{multiplication on Lie algebra}). Then the following holds true:}

\noindent(i) There exist $\varepsilon,\rho_{1}>0$
and $\rho_{2}\in(0,r)$, such that for any given $x\in\mathbb{R}^{N}$
and $h\in\bar{{\cal H}}$ with $\|h\|_{\bar{{\cal H}}}<\rho_{1},$
the map 
\[
\tilde{\Psi}_{x,h}(v)\triangleq\Psi_{l}\left(v\times u,x, y-x-F_{l}(v\times u,x)+F_{l}(v,x)\right)
\]
defines a diffeomorphism from an open neighbourhood $V_{x,h}\subset\frak{g}^{(l)}$ of $0$ containing the ball $\{v\in\mathfrak{g}^{(l)}:\|v\|_{\textsc{HS}}<\varepsilon\}$ onto 
$W\triangleq\{w\in\mathfrak{g}^{(l)}:\|w\|_{\textsc{HS}}<\rho_{2}\}$,
where recall that the function $\Psi_l$ is defined in Lemma~\ref{lem: the Psi function}.
In addition,
$$v\in V_{x,h}\cap M_{x,x}\qquad  \text{iff}\qquad w\triangleq\tilde{\Psi}_{x,h}(v)\in W\cap M_{x,y}.$$
(ii) Recall that $d(x,y)$ is the control distance
function associated with the SDE (\ref{eq: hypoelliptic SDE}) and that the set $\Pi_{x,y}$ is defined by~\eqref{eq: Pi_xy}. Given $x,y\in\mathbb{R}^{N}$ with $d(x,y)<\rho_{1}/2$, choose
$h\in\Pi_{x,y}$ satisfying 
\begin{align}\label{picking h}
d(x,y)\leq\|h\|_{\bar{{\cal H}}}\leq2d(x,y)<\rho_{1},
\end{align}
and define 
\begin{align}\label{def of Sxy}
S_{x,y}\triangleq\left.\tilde{\Psi}_{x,h}^{-1}\right|_{W\cap M_{x,y}}:W\cap M_{x,y}\rightarrow V_{x,h}\cap M_{x,x}.
\end{align}
Then there exist a constant $C>0$$,$ such that 
\begin{align}\label{compare vol mxx and mxy}
\frac{1}{C}\cdot m_{x,x}(\cdot)\leq m_{x,y}\circ S_{x,y}^{-1}(\cdot)\leq C\cdot m_{x,x}(\cdot)
\end{align}
on $V_{x,h}\cap M_{x,x}.$
\end{lem}
%
The previous lemma sets the stage for a useful change of variable in \eqref{eq: formula for p^eta_l}. Our next step is to provide some useful bounds for the integral in~\eqref{eq: formula for p^eta_l}. We first need the following crucial lemma.
\begin{lem}\label{lem: d respect cc}
\label{lem: d < CC}There exist constants $C,\kappa>0$ such that
for any $u\in\mathfrak{g}^{(l)}$ with $\|u\|_{\textsc{HS}}<\kappa$,
we have 
\begin{align}\label{d respect cc}
d(x,x+F_{l}(u,x))\leq C\|u\|_{\textsc{CC}}.
\end{align}
\end{lem}
\begin{proof}
We only consider the case when $H>1/2$, as the other case follows
from Lemma \ref{lem: continuous embedding when H<1/2} and the result
for the diffusion case proved in \cite{KS87}. When $H>1/2,$ the
argument is an adaptation of the proof of Theorem \ref{thm: local comparison}.
We use the same notation as in that proof, except for the fact that $l_{0}$ is replaced
by a general $l\geq l_{0}$ exclusively.

We set up an inductive procedure as in the proof of Theorem \ref{thm: local comparison}. Namely,  denote $u_{1}\triangleq u,$ $x_{1}\triangleq x$,  and $y\triangleq x+F_{l}(u,x).$
Choose $\kappa_{1}>0$ so that 
\begin{align}\label{cc small imply x close to y}
\|u\|_{\textsc{CC}}<\kappa_{1}\implies|y-x|<\delta,
\end{align}
where $\delta$ is again the constant arising in the proof of Theorem \ref{thm: local comparison}. By constructing successively
elements $u_{m}\in\mathfrak{g}^{(l)}$ and intervals $I_{m}$, we obtain exactly as in \eqref{d1} that
\begin{equation}
|I_{m}|=\|u_{m}\|_{\textsc{CC}}\leq C_{V,l}\|u_{m-1}\|^{1+\frac{1}{l}}=C_{V,l}|I_{m-1}|^{1+\frac{1}{l}} .\label{eq: recursive estimate of I_m}
\end{equation}
In addition, along the same lines as (\ref{eq: estimating d by I_m}) and \eqref{d2}, we have 
\begin{align}\label{bound d by pieces}
d(x,x+F_{l}(u,x))^{2}\leq 
C_{H,l}
\lim_{m\to\infty}
\left(\sum_{k=1}^{m}|I_{k}|\right)^{2H}\left(\sum_{k=1}^{m}k^{2H+1}|I_{k}|^{2(1-H)}\right).
\end{align}
We will now bound the right hand side of \eqref{bound d by pieces}.

Let us set $\alpha\triangleq1+1/l$. By iterating (\ref{eq: recursive estimate of I_m}),
we obtain that 
\[
|I_{m}|\leq\left(C_{V,l}|I_{1}|\right)^{\alpha^{m-1}},\ \ \ \forall m\geq1.
\]
Therefore, we can bound the two terms on the right hand-side of \eqref{bound d by pieces} as follows:
\begin{align}\label{first term in bound d by pieces}
\sum_{k=1}^{m}|I_{k}|\leq C_{V,l}|I_{1}|\cdot\left(\sum_{k=1}^{m}(C_{V,l}|I_{1}|)^{\alpha^{k-1}-1}\right)
\end{align}
and 
\begin{align}\label{second term in bound d by pieces}
\sum_{k=1}^{m}k^{2H+1}|I_{k}|^{2(1-H)}\leq(C_{V,l}|I_{1}|)^{2(1-H)}\cdot\left(\sum_{k=1}^{m}k^{2H+1}(C_{V,l}|I_{1}|)^{2(1-H)(\alpha^{k-1}-1)}\right).
\end{align}
Let us now bound the term $|I_1|$ in \eqref{first term in bound d by pieces} and \eqref{second term in bound d by pieces}. To this aim, recall that we have chosen $u_1=u$. Therefore one can choose $\kappa_{2}>0$ so that 
\begin{align}\label{u cc norm small imply I1 small}
\|u\|_{\textsc{CC}}<\kappa_{2}\implies C_{V,l}|I_{1}|=C_{V,l}\|u_1\|_{\textsc{CC}}\leq\frac{1}{2}.
\end{align}
We will assume that both \eqref{cc small imply x close to y} and \eqref{u cc norm small imply I1 small} are satisfied in the sequel, under the condition that $\|u\|_{\textsc{CC}}<\kappa$ with $\kappa=\kappa_1\wedge\kappa_2$.
In order to bound \eqref{first term in bound d by pieces} and \eqref{second term in bound d by pieces} above, also
observe that the series 
\begin{align}\label{two series}
\sum_{m=1}^{\infty}\left(\frac{1}{2}\right)^{\alpha^{m-1}-1},\ \quad\text{and}\ \quad \sum_{m=1}^{\infty}m^{2H+1}\left(\frac{1}{2}\right)^{2(1-H)(\alpha^{m-1}-1)}
\end{align}
are convergent. Therefore plugging \eqref{two series} and \eqref{u cc norm small imply I1 small}, and then \eqref{first term in bound d by pieces} and \eqref{second term in bound d by pieces} into \eqref{bound d by pieces}, we get that 
for $\|u\|_{\textsc{CC}}<\kappa$ we have
\begin{equation*}
d(x,x+F_{l}(u,x))^{2}  \leq C_{H,V,l}|I_{1}|^{2H}\cdot|I_{1}|^{2(1-H)}
 =C_{H,V,l}|I_{1}|^{2}
 =C_{H,V,l}\|u\|_{\textsc{CC}}^{2}.
\end{equation*}
Therefore, our result \eqref{d respect cc} follows.
\end{proof}
Lemma \ref{lem: d respect cc} yields the following two-sided estimate on the transformation $S_{x,y}$.

\begin{lem}
\label{lem: estimating S_xy}Keeping the same notation as in Lemma
\ref{lem: construction of S_xy}, let $x,y\in\mr^N$ be such that $d(x,y)<\rho_1/2$ and let $h\in\Pi_{x,y}$ fulfilling condition \eqref{picking h}. Then the function $S_{x,y}$ defined by \eqref{def of Sxy} satisfies the following inequality for all $v\in V_{x,h}\cap M_{x,x}$,
\begin{equation}
\frac{1}{\Lambda}\cdot\left(\|v\|_{\textsc{CC}}+d(x,y)\right)\leq\|S_{x,y}^{-1}(v)\|_{\textsc{CC}}\leq \Lambda\cdot\left(\|v\|_{\textsc{CC}}+d(x,y)\right)\label{eq: estimating S_xy}
\end{equation}
{for  some constant $\Lambda>0.$}

\end{lem}
\begin{proof}
We prove the two inequalities in our claim separately. 

\smallskip

\noindent{\it Step 1: proof of the upper bound in \eqref{eq: estimating S_xy}}.   To this aim, we
set $w\triangleq S_{x,y}^{-1}(v)$ and $u\triangleq\log S_{l}(h)$. Note that according to the definition of $S_{x,y}$ given in Lemma \ref{lem: construction of S_xy} we have  $w=\tilde{\Psi}_{x,h}(v)$. We wish to upper bound $\|w\|_{\textsc{CC}}$ in order to get the second part of \eqref{eq: estimating S_xy}.

According to Lemma \ref{lem: construction of S_xy}, $\|w\|_{\textsc{CC}}$ is close to $\|v\times u\|_{\textsc{CC}}$. Specifically, 
Proposition \ref{prop: ball-box estimate} entails 
\[
\left|\|w\|_{\textsc{CC}}-\|v\times u\|_{\textsc{CC}}\right|\leq\rho_{\textsc{CC}}(w,v\times u)\leq C_{1}\|w-v\times u\|_{\textsc{HS}}^{\frac{1}{l}}.
\]
Let $\gamma\in C^{1-{\rm var}}([0,1];\mathbb{R}^{d})$ be such that
$v=\log (S_{l}(\gamma))$ and $\|v\|_{\textsc{CC}}=\|\gamma\|_{1-{\rm var}}.$
By the definition of $\tilde{\Psi}_{x,h}$, the fact that $w=\tilde{\Psi}_{x,h}(v)$ and Lemma \ref{lem: the Psi function}
(ii), we have 
\begin{align}
 & \|w-v\times u\|_{\textsc{HS}}\nonumber\\
 & \leq A|y-x-F_{l}(v\times u,x)|\nonumber\\
 & \leq A\left(|y-\Phi_{1}(x;\gamma\sqcup h)|+|\Phi_{1}(x;\gamma\sqcup h)-x-F_{l}(v\times u,x)|\right)\triangleq A(Q_1+Q_2).\label{Q1 and Q2}
\end{align}
We will now bound the terms $Q_1$ and $Q_2$ respectively.  

In order to estimate the term $Q_1$ in \eqref{Q1 and Q2}, let us recall that $y=\Phi_1(x;h)$. Therefore using the flow property of $\Phi$ we get
$$Q_1=|\Phi_1(x;h)-\Phi_1(\Phi_1(x,\gamma);h)|.$$
We can now use some standard estimate on the flow of differential equations (cf. \cite[Theorem 10.26]{FV10}) plus the fact that $\|h\|_{\bch}$ is bounded by $\rho_1$ (cf. condition \eqref{picking h}) in order to get
$$Q_1\leq C_2|\Phi_1(x;\gamma)-x|.$$
In addition, $v$ is assumed to be an element of $M_{x,x}$. According to \eqref{Mxy}, this means in particular that $F_l(v,x)=0$. Hence we obtain
$$Q_1\leq C_2|\Phi_1(x;\gamma)-x-F_l(v,x)|.$$
Thanks to the Euler estimate of \cite[Corollary 10.15]{FV10}, and the fact that $v=\log (S_l(\gamma))$, we end up with
\begin{align}\label{estimate Q1}Q_1\leq C_2 C_{V,l}\|\gamma\|_{1-\rm var}^{\bar{l}}=C_2C_{V,l}\|v\|_{\textsc{CC}}^{\bar{l}},\end{align}
where $\bar{l}$ is any given number in $(l,l+1)$.

In order to handle the term $Q_2$ in \eqref{Q1 and Q2}, let us observe that whenever $H>1/4$, Proposition \ref{prop: variational embedding} entails that $h\in C^{q-\rm var}$ for $q\in[1,2)$, for all $h\in\bch$. Moreover, the following inequality holds true,
\begin{equation}
\|h\|_{q-{\rm var}}\leq C_{H}\|h\|_{\bar{{\cal H}}}\label{eq: variational estimate for CM path}.
\end{equation}
Therefore, since $v\times u=\log (S_l(\gamma\sqcup h))$, the rough path estimates in \cite[Corollary 10.15]{FV10} yield
\begin{align*}
Q_2\leq C_{V,l}\|\gamma\sqcup h\|_{q-{\rm var}}^{\bar{l}}\leq C_{V,l}2^{\bar{l}-\frac{\bar{l}}{q}}\left(\|\gamma\|_{1-{\rm var}}+\|h\|_{q-{\rm var}}\right)^{\bar{l}},
\end{align*}
where the last inequality stems from the simple relation

\[
\|\gamma\sqcup h\|_{q-{\rm var}}\leq2^{1-\frac{1}{q}}\left(\|\gamma\|_{q-{\rm var}}+\|h\|_{q-{\rm var}}\right),\quad\mathrm{and}\quad \|\gamma\|_{q-\rm var}\leq \|\gamma\|_{1-\rm var}.
\]
Furthermore, since we have chosen $v$ such that $v=\log(S_l(\gamma))$ and $\|v\|_{\textsc{CC}}=\|\gamma\|_{1-\rm var}$, we get
\begin{align}\label{estimate Q2}
Q_2\leq C_{H,V,l}(\|v\|_{\textsc{CC}}+\|h\|_{\bch})^{\bar{l}},
\end{align}
where we have also  invoked \eqref{eq: variational estimate for CM path} in order to upper bound $\|h\|_{q-\rm var}$.

Summarizing our considerations so far, if we plug our estimate \eqref{estimate Q1} on $Q_1$ and our bound \eqref{estimate Q2} on $Q_2$ into relation \eqref{Q1 and Q2}, we obtain the following inequality
\begin{equation}
\left|\|w\|_{\textsc{CC}}-\|v\times u\|_{\textsc{CC}}\right|\leq C_{3}\left(\|v\|_{\textsc{CC}}+\|h\|_{\bar{{\cal H}}}\right)^{1+\frac{\bar{l}-l}{l}}.\label{estimate of w-vu}
\end{equation}

Next, we claim that $\|u\|_{\textsc{CC}}\leq C_{H,l}\|h\|_{\bar{{\cal H}}}.$
Indeed, recall that $u=\log(S_l(h))$ and set $S_l(h)=g$. This means that $g=\exp(u)$. Since $h\in C^{q-\text{var}}$ with $q\in[1,2)$, Lyons' extension theorem  (cf. \cite[Theorem 2.2.1]{Lyons98}) implies that for all $i=1,...,l$ we have
\[
\|g_{i}\|_{\textsc{HS}}\leq C_{H,l}\|h\|_{q-{\rm var}}^{i}\, ,
\]
where $g_{i}$ is the $i$-th component of $g.$ If we define the
homogeneous norm $\interleave\cdot\interleave$ on $G^{(l)}$ by 
\[
\interleave\xi\interleave\triangleq\max_{1\leq i\leq l}\|\xi_{i}\|_{\textsc{HS}}^{\frac{1}{i}},\ \ \ \xi\in G^{(l)},
\]
we get the following estimate:
\begin{align}\label{bound u cc norm by qvar of h}
\interleave g\interleave\leq C_{H,l}\|h\|_{q-\text{var}},\quad\text{and}\quad \|u\|_{\textsc{CC}}\leq C_{H,l}\interleave g\interleave,
\end{align}
where the second inequality stems from the equivalence of homogeneous norms in $G^{(l)}$ (cf. \cite[Theorem 7.44]{FV10}).
Now combining the two inequalities in \eqref{bound u cc norm by qvar of h} and relation \eqref{eq: variational estimate for CM path}, we end up with
\begin{align}\label{bound cc norm of u by H bar norm of h}
\|u\|_{\textsc{CC}}\leq C_{H,l}\|h\|_{q-{\rm var}}\leq C_{H,l}\|h\|_{\bch}.
\end{align}

Let us now go back to (\ref{estimate of w-vu}), from which we easily deduce
\begin{align*}
\|w\|_{\textsc{CC}} & \leq\|v\times u\|_{\textsc{CC}}+C_{3}\left(\|v\|_{\textsc{CC}}+\|h\|_{\bar{{\cal H}}}\right)^{1+\frac{\bar{l}-l}{l}}\\
 & \leq\|v\|_{\textsc{CC}}+\|u\|_{\textsc{CC}}+C_{3}\left(\|v\|_{\textsc{CC}}+\|h\|_{\bar{{\cal H}}}\right)^{1+\frac{\bar{l}-l}{l}}
 \end{align*}
 Plugging \eqref{bound cc norm of u by H bar norm of h} into this inequality and resorting to the fact that $v\in V_{x,h}$ and $h$ satisfies~\eqref{picking h}, we end up with
 \begin{align*}
\|w\|_{\textsc{CC}} & \leq\|v\|_{\textsc{CC}}+C_{H,l}\|h\|_{\bar{{\cal H}}}+C_{3}\left(\|v\|_{\textsc{CC}}+\|h\|_{\bar{{\cal H}}}\right)^{1+\frac{\bar{l}-l}{l}}\\
 & \leq C_{4}\left(\|v\|_{\textsc{CC}}+\|h\|_{\bar{{\cal H}}}\right)\\
 & \leq C_{4}\left(\|v\|_{\textsc{CC}}+2d(x,y)\right).
\end{align*}
Recalling that we have set $w=S_{xy}^{-1}(v)$, this proves the upper bound in (\ref{eq: estimating S_xy}).

\smallskip

\noindent{\it Step 2: proof of the lower bound in \eqref{eq: estimating S_xy}.} We start from inequality \eqref{estimate of w-vu}, which yields
\begin{align*}
\|w\|_{\textsc{CC}} & \geq\|v\times u\|_{\textsc{CC}}-C_{3}\left(\|v\|_{\textsc{CC}}+\|h\|_{\bar{{\cal H}}}\right)^{1+\frac{\bar{l}-l}{l}}\\
 & \geq\|v\|_{\textsc{CC}}-\|u\|_{\textsc{CC}}-C_{3}\left(\|v\|_{\textsc{CC}}+\|h\|_{\bar{{\cal H}}}\right)^{1+\frac{\bar{l}-l}{l}}
 \end{align*}
 Furthermore, thanks to \eqref{bound cc norm of u by H bar norm of h} we obtain
 \begin{align}\label{lower bound w cc norm 1}
\|w\|_{\textsc{CC}} & \geq\|v\|_{\textsc{CC}}-C_{H,l}\|h\|_{\bar{{\cal H}}}-C_{3}\left(\|v\|_{\textsc{CC}}+\|h\|_{\bar{{\cal H}}}\right)^{1+\frac{\bar{l}-l}{l}}.
\end{align}
We now invoke the fact that $\|v\|_{\textsc{HS}}<r$ and $\|h\|_{\bar{{\cal H}}}<\rho_{1},$
thus by possibly shrinking $r$ and $\rho_{1},$ we may assume that 
\[
C_{3}\left(\|v\|_{\textsc{CC}}+\|h\|_{\bar{{\cal H}}}\right)^{\frac{\bar{l}-l}{l}}<\frac{1}{2}.
\]
{Putting this information into \eqref{lower bound w cc norm 1}, we thus obtain }
\begin{align}\label{lower bound w cc norm 2}
\|w\|_{\textsc{CC}}\geq\frac{1}{2}\|v\|_{\textsc{CC}}-\left(\frac{1}{2}+C_{H,l}\right)\|h\|_{\bar{{\cal H}}}\geq\frac{1}{2}\|v\|_{\textsc{CC}}-(1+2C_{H,l})d(x,y).
\end{align}
On the other hand, according to Lemma \ref{lem: d < CC} (one may
further shrink $r$ so that $\|w\|_{\textsc{CC}}<\kappa$, where $\kappa$ is the constant featuring in Lemma \ref{lem: d respect cc}), we have 
\begin{align}\label{d bound by w cc norm}
d(x,y)=d(x,x+F_{l}(w,x))\leq C\|w\|_{\textsc{CC}}.
\end{align}
Putting together inequalities \eqref{lower bound w cc norm 2} and \eqref{d bound by w cc norm}, we easily get the lower bound in \eqref{eq: estimating S_xy}, which finishes our proof.
\end{proof}
Before proceeding to analyze the density $p_l^\eta(t,x,y)$, with all the preparations above we take a short detour to prove the local equivalence of controlling distances claimed in Theorem \ref{thm:equiv-distances}. 

\begin{proof}[Proof of Theorem \ref{thm:equiv-distances}] For each fixed $x\in\mr^N$, define a function on  $\mr^N$ by
$$g(x,y)=\inf\{\|u\|_{\textsc{CC}}: u\in\frak{g}^{(l)}\ \text{and}\ x+F_l(u,x)=y\}.$$
Observe that $g(x,y)$ is an intrinsic quantity that does not depend on $H$. In order to prove Theorem \ref{thm:equiv-distances}, it suffices to show that $d(x,y)$ is equivalent to $g(x,y)$ in a Euclidean neighborhood of $x$. 

To this aim, first note that by Corollary \ref{cor: local submersion} the map $G(\cdot): u\mapsto x+F_l(u,x)$ is a submersion in a neighborhood of $0\in \frak{g}^{(l)}$. Recall that the set $W\cap M_{x,y}$ is introduced in Lemma~\ref{lem: construction of S_xy}. We can thus choose a small enough $\delta>0$ such that $\{G^{-1}(y): |x-y|\leq \delta\}\subset W\cap M_{x,y}$ and that both Lemma \ref{lem: d respect cc} and Lemma \ref{lem: estimating S_xy} can be applied. 

Now we fix  such a choice of $\delta$.  For any $y\in \mr^N$ with $|x-y|\leq\delta$, we can first apply Lemma \ref{lem: d respect cc} to conclude that
$$d(x,y)\leq Cg(x,y).$$
Next, we use the second inequality in Lemma \ref{lem: estimating S_xy} for $v=0$ to conclude that 
$$g(x,y)\leq \Lambda d(x,y).$$
The proof is thus completed.
\end{proof}

\bigskip
Now we come back to the main goal of this part. Namely starting from Proposition \ref{prop: representation of density}, we will apply a change of variables involving $S_{x,y}$ and express the density $p_l^\eta$ in terms of the measure $m_{x,x}$ not depending on $y$.

\begin{lem}\label{lem: lower estimate in terms of m_xx}
\label{lem: lower estimate in terms of m_xx}Let $p_l^\eta(t,x,y)$ be the density defined by \eqref{eq: formula for p^eta_l}, and recall that the exponent $\nu$ is defined by $\nu\triangleq\sum_{k=1}^{l}k\dim{\cal L}_{k}.$ Then there exist constants
$C,\tau>0$, such that for all $x,y,t$ with $d(x,y)\leq t^{H}$
and $0<t<\tau$, we have 
\begin{equation}
p_{l}^{\eta}(t,x,y)\geq Ct^{-H\nu}m_{xx}\left(\{v\in M_{x,x}:\|v\|_{\textsc{CC}}\leq t^{H}\}\right),\label{eq: lower estimate in terms of m_xx}
\end{equation}
where $m_{x,x,}$ is the volume form on $M_{x,x}$ given by \eqref{Mxy}
\end{lem}
\begin{proof} Lemma \ref{lem: construction of S_xy} asserts that there exists $\rho_1>0$, such that if $d(x,y)<\rho_1/2$, then $S_{x,y}$ given by \eqref{def of Sxy} defines a change of variables (i.e. a diffeomorphism) for (\ref{eq: formula for p^eta_l}). Specifically we have
\begin{align*}
p_{l}^{\eta}(t,x,y) & \geq\int_{M_{x,y}\cap W}\eta(u)K(u,x)\rho_{t}(u)m_{x,y}(du)\\
 & =\int_{M_{x,x}\cap V_{x,h}}\eta(S_{x,y}^{-1}v)K(S_{x,y}^{-1}v,x)\rho_{t}(S_{x,y}^{-1}v)
 \, m_{x,y}\circ S_{x,y}^{-1}(dv).
 \end{align*}
In addition, since $V_{x,h}$ contains the ball $\{v\in\mathfrak{g}^{(l)}:\|v\|_{\mathrm{HS}}<\varepsilon\}$, owing to relation \eqref{compare vol mxx and mxy} and thanks to the fact that $K$ defined by \eqref{kernel in disintegration} is bounded below, we obtain
 \begin{align*}
 p_l^\eta(t,x,y) \geq C_{H,V,l}\int_{M_{x,x}\cap\{v\in\mathfrak{g}^{(l)}:\|v\|_{\textsc{HS}}<\varepsilon\}}\rho_{t}(S_{x,y}^{-1}v) \, m_{x,x}(dv).
\end{align*}
Now choose $\tau<(\rho_{1}/2)^{\frac{1}{H}}$ to be such that 
\[
0<t<\tau\implies\{v\in\mathfrak{g}^{(l)}:\|v\|_{\textsc{CC}}\leq t^{H}\}\subseteq\left\{ v\in\mathfrak{g}^{(l)}:\|v\|_{\textsc{HS}}<\varepsilon\right\} .
\]
We will thus lower bound $p_l^\eta(t,x,y)$ as follows
 \begin{align}\label{lower bound p_l^eta 1}
 p_l^\eta(t,x,y) \geq C_{H,V,l}\int_{M_{x,x}\cap\{v\in\mathfrak{g}^{(l)}:\|v\|_{\textsc{CC}}<t^{H}\}}
 \rho_{t}(S_{x,y}^{-1}v)m_{x,x}(dv).
\end{align}

Next, according to the second inequality of (\ref{eq: estimating S_xy}),
if $d(x,y)\leq t^{H}$ and $t<\tau$ (so that $d(x,y)<\rho_{1}/2$),
then 
\[
\|S_{x,y}^{-1}v\|_{\textsc{CC}}\leq2Ct^{H},
\]
provided that $v\in M_{x,x}$ with $\|v\|_{\textsc{CC}}\leq t^{H}.$
For such $x,y,t,v$, by Proposition \ref{prop: step one} we have 
\[
\rho_{t}(S_{x,y}^{-1}v)\geq\beta_{2C}t^{-H\nu}.
\]
Plugging this inequality into \eqref{lower bound p_l^eta 1}, we arrive at 
\[
p_{l}^{\eta}(t,x,y)\geq C_{H,V,l}\beta_{2C}t^{-H\nu}m_{x,x}\left(\{v\in M_{x,x}:\|v\|_{\textsc{CC}}\leq t^{H}\}\right),
\]
which is our claim \eqref{eq: lower estimate in terms of m_xx}.

\end{proof}
The next result tells us that the right hand side of (\ref{eq: lower estimate in terms of m_xx})
is comparable with the inverse volume of $B_{d}(x,t^{H}).$ This seems
to be surprising as the first quantity does not capture the Gaussian
structure at all while the second quantity relies crucially on the
Cameron-Martin structure. The key reason behind this lies in the precise
two-sided estimate (\ref{eq: estimating S_xy}) of $S_{x,y}$ in terms
of $d(x,y)$, which is also the key point leading to the local equivalence of all the control distance functions as we just proved.
\begin{lem}\label{lem: estimate vol m_xx}Let $M_{x,x}$ be the set defined by \eqref{Mxy} and recall that $m_{x,x}$ is the volume measure on $M_{x,x}$.
There exist constants $C,\tau>0,$ such that 
\begin{align}\label{estimate vol m_xx}
\frac{1}{C|B_{d}(x,t^{H})|}\leq t^{-H\nu}m_{x,x}\left(\{v\in M_{x,x}:\|v\|_{\textsc{CC}}\leq t^{H}\}\right)\leq\frac{C}{|B_{d}(x,t^{H})|}
\end{align}
for all $x\in\mathbb{R}^{N}$ and $0<t<\tau$.
\end{lem}
\begin{proof}
The upper and lower bounds in \eqref{estimate vol m_xx} follow the same pattern, therefore we focus on the proof of the lower bound. To this aim, let $\psi\in C_{c}^{\infty}((-1,1);\mathbb{R}^{1})$
be such that $0\leq\psi\leq1$ and $\psi(\xi)=1$ when $|\xi|\leq1/2.$
We further localize the measure $P^\eta_l$ defined in Proposition \ref{prop: representation of density} by considering the following measure
\[
P_{l}^{\eta,\psi}(t,x,\Gamma)\triangleq\mathbb{E}\left[\eta(U_{t})\psi\left(\frac{\Lambda\|U_{t}\|_{\textsc{CC}}}{t^{H}}\right)\mathbf{1}_{\left\{ X_{l}(t,x)\in\Gamma\right\} }\right],\ \ \ \Gamma\in\mathcal{B}(\mathbb{R}^{N}),
\]
where $\Lambda$ is the constant appearing in Lemma \ref{lem: estimating S_xy}.
Along the same lines as for the disintegration formula \eqref{eq: disintegration}, the measure $P_{l}^{\eta,\psi}(t,x,\cdot)$
has a density given by 
\begin{align*}
p_{l}^{\eta,\psi}(t,x,y) & =\int_{M_{x,y}}\eta(u)\psi\left(\frac{\Lambda\|u\|_{\textsc{CC}}}{t^{H}}\right)K(u,x)\rho_{t}(u)m_{x,y}(du)\\
 & =\int_{M_{x,y}\cap W}\eta(u)\psi\left(\frac{\Lambda\|u\|_{\textsc{CC}}}{t^{H}}\right)K(u,x)\rho_{t}(u)m_{x,y}(du),
\end{align*}
where the kernel $K$ is given by \eqref{kernel in disintegration}  and provided that $\tau$ further satisfies
\[
0<t<\tau\implies\left\{ u\in\mathfrak{g}^{(l)}:\Lambda\|u\|_{\textsc{CC}}\leq t^{H}\right\} \subseteq W.
\]
As in the proof of Lemma \ref{lem: lower estimate in terms of m_xx}, we now apply the change of variables $S_{x,y}u=v$ and the fact that $\psi(\xi=0)$ if $|\xi|\ge 1$ in order to get
\begin{multline*}
 p_{l}^{\eta,\psi}(t,x,y)
 =\int_{V_{x,h}\cap\left\{ v\in M_{x,x}:\Lambda\|S_{x,y}^{-1}v\|_{\textsc{CC}}\leq t^{H}\right\} }
 \eta(S_{x,y}^{-1}v) \, \psi\left(\frac{\Lambda\|S_{x,y}^{-1}v\|_{\textsc{CC}}}{t^{H}}\right)  \\
 \times K(S_{x,y}^{-1}v,x) \, \rho_{t}(S_{x,y}^{-1}v) \, m_{x,y}\circ S_{x,y}^{-1}(dv).
 \end{multline*}
 Furthermore, due to the fact that $\psi$ is supported on $(-1,1)$, $K$ is bounded owing to~\eqref{kernel in disintegration} and according to the upper bound in \eqref{compare vol mxx and mxy}, we get
 \begin{align*}
 p_l^{\eta,\psi}(t,x,y) \leq C_{H,V,l}\int_{V_{x,h}\cap\left\{ v\in M_{x,x}:\Lambda\|S_{x,y}^{-1}v\|_{\textsc{CC}}\leq t^{H}\right\} }\rho_{t}(S_{x,y}^{-1}v) \, m_{x,x}(dv).
 \end{align*}
 Therefore identity (\ref{eq: formula for rho_t}) yields
 \begin{align*}
 p_l^{\eta,\psi}(t,x,y) \leq C_{H,V,l}t^{-H\nu}m_{x,x}\left(\{v\in M_{x,x}:\Lambda\|S_{x,y}^{-1}v\|_{\textsc{CC}}\leq t^{H}\}\right).
 \end{align*}
 Finally the lower bound on $\|S_{x,y}^{-1}(v)\|_{\textsc{CC}}$ in \eqref{eq: estimating S_xy} implies that whenever $d(x,y)\leq t^H$ and $0<t<\tau$ we have
 \begin{align}
 p_l^{\eta,\psi}(t,x,y) \leq C_{H,V,l}t^{-H\nu}m_{x,x}\left(\{v\in M_{x,x}:\|v\|_{\textsc{CC}}\leq t^{H}\}\right). \label{p^psi < m_xx}
\end{align}

We now lower bound the density $p_l^{\eta,\psi}$ on the ball $B_d(x,t^H)$. To this aim we first write
\begin{align}\label{lower bound density mid step 1}
 & \int_{B_{d}(x,t^{H})}p_{l}^{\eta,\psi}(t,x,y)dy\nonumber \\
 & =\mathbb{E}\left[\eta(U_{t})\psi\left(\frac{\Lambda\|U_{t}\|_{\textsc{CC}}}{t^{H}}\right)\mathbf{1}_{\left\{ d(x,x+F(U_{t},x))<t^{H}\right\} }\right] \nonumber\\
 & \geq\mathbb{P}\left(d(x,x+F(U_{t},x))<t^{H},\ \|U_{t}\|_{\textsc{CC}}\leq\frac{t^{H}}{2\Lambda},\ \|U_{t}\|_{\textsc{HS}}<\kappa\wedge\frac{r}{2}\right),
\end{align}
where the second inequality stems from the fact that in Proposition \ref{prop: representation of density} we have assumed that $\eta=1$ where $\|u\|_{\textsc{HS}}<r/2$ and we also have $\psi(r)=1$ if $|r|\leq 1/2$. Next we resort to Lemma \ref{lem: d respect cc}, which can be rephrased as follows: there exist $C, \kappa, \tau>0$ such that if $t<\tau$ and $\|u\|_{\textsc{CC}}<\gamma t^H$ with $\gamma\triangleq(\max\{C,2\Lambda\})^{-1},$ then we have
\begin{align}\label{lower bound density mid step 2}
\|u\|_{\textsc{CC}}\leq\gamma t^{H}\implies d(x,x+F_{l}(u,x))<t^{H},\ \|u\|_{\textsc{CC}}\leq\frac{t^{H}}{2\Lambda},\ \|u\|_{\textsc{HS}}<\kappa\wedge\frac{r}{2}.
\end{align}
Plugging \eqref{lower bound density mid step 2} into \eqref{lower bound density mid step 1}, we thus get that for $t<\tau$ we have
\begin{align*}
  \int_{B_{d}(x,t^{H})}p_{l}^{\eta,\psi}(t,x,y)dy
 \geq\mathbb{P}\left(\|U_{t}\|_{\textsc{CC}}\leq\gamma t^{H}\right)
  =\mathbb{P}\left(\|\delta_{t^{H}}^{-1}U_{t}\|_{\textsc{CC}}\leq\gamma\right).
 \end{align*}
 Eventually, owing to the scaling property of $U_t$ alluded to in \eqref{eq: scaling property of U_t}, we end up with
 \begin{align}
 \int_{B_{d}(x,t^{H})}p_{l}^{\eta,\psi}(t,x,y)dy 
   \ge
 \int_{\{u\in\mathfrak{g}^{(l)}:\|u\|_{\textsc{CC}}\leq\gamma\}}\rho_{1}(u)du
  \triangleq C_{\gamma,l}.\label{eq: int p^psi>1}
\end{align}
Now the lower bound in \eqref{estimate vol m_xx} follows from integrating both sides of (\ref{p^psi < m_xx}) over
$B_{d}(x,t^{H})$ and (\ref{eq: int p^psi>1}).
\end{proof}
Summarizing the content of Lemma \ref{lem: lower estimate in terms of m_xx} and Lemma \ref{lem: estimate vol m_xx},  we have obtained the following lower bound on $p_l^\eta(t,x,y)$, which finishes the second step of the main strategy.

\begin{cor}Let $p_l^\eta(t,x,y)$ be the density given by \eqref{eq: formula for p^eta_l}, and recall the notations of Lemma \ref{lem: lower estimate in terms of m_xx}. Then there exist constants  $C,\tau>0$
depending only on $H,l$ and the vector fields, such that 
\begin{equation}\label{eq: lower estimate for approximating density}
p_{l}^{\eta}(t,x,y)\geq\frac{C}{|B_{d}(x,t^{H})|}
\end{equation}
for all $x,y,t$ satisfying $d(x,y)\leq t^{H}$ and $0<t<\tau$.
\end{cor}

\subsection{Step three: comparing approximating and actual densities.}\label{sec: step three}

The last step towards the proof of Theorem \ref{thm: local lower estimate}
will be to show that the approximating density $p_{l}^{\eta}(t,x,y)$
and the actual density $p(t,x,y)$ of $X_t^x$ are close to each other when $t$
is small. For this part, we combine the Fourier transform approach
developed in \cite{KS87} with general estimates for Gaussian rough differential
equations. As we will see, there is a quite subtle point related to
the uniformity in $l$ (the degree of approximation) when obtaining upper
bound of $p_{l}^{\eta}(t,x,y)$ which is the main challenge for this
part. In our modest opinion, we believe that there is a gap in the argument in~\cite{KS87}
for the diffusion case, and we therefore propose an alternative proof in the
fractional Brownian setting which also covers the diffusion result. As before,
we assume that $l\geq l_{0}.$

Recall that the Fourier transform of a function $f(y)$ on $\mathbb{R}^{N}$
is defined by 
\[
\cf f(\xi)=\hat{f}(\xi)\triangleq\int_{\mathbb{R}^{N}}f(y){\rm e}^{2\pi i\langle\xi,y\rangle} \, dy,\ \ \ \xi\in\mathbb{R}^{N},
\]
where we highlight the fact that $\cf f$ and $\hat{f}$ are used indistinctly to designate our Fourier transform.
In the sequel we will consider the Fourier transform $\hat{p}(t,x,\xi)$ (respectively, $\hat{p}_{l}^{\eta}(t,x,\xi)$)
of the density $p(t,x,y)$ (respectively, $p_{l}^{\eta}(t,x,y)$)
with respect to the $y$-variable. We will invoke the following trivial bound on $p-p_{l}$ in terms of $\hat{p}$ and $\hat{p}_{l}^{\eta}$:
\begin{equation}
|p(t,x,y)-p_{l}^{\eta}(t,x,y)|\leq\int_{\mathbb{R}^{N}}\left|\hat{p}(t,x,\xi)-\hat{p}_{l}^{\eta}(t,x,\xi)\right|d\xi.\label{eq: Fourier estimate}
\end{equation}
Therefore our aim in this section will be to estimate the right hand side of \eqref{eq: Fourier estimate} by considering two regions $\{|\xi|\leq R\}$ and $\{|\xi|>R\}$ separately in the integral, where $R$ is some large number to
be chosen later on.

\paragraph*{I. Integrating relation \eqref{eq: Fourier estimate} in a neighborhood of the origin.}
\label{sec: integrate-small}
\addcontentsline{toc}{subsubsection}{\nameref{sec: integrate-small}}

$\ $\\
\\
We first integrate our Fourier variable $\xi$ in \eqref{eq: Fourier estimate} over the region $\{|\xi|\leq R\}$.
In this case, we make use of a
tail estimate for the error of the Taylor approximation of $X_{t}^{x}$ which is provided below. 
\begin{lem}
\label{lem: tail estimate for Taylor error}
Let $X_{t}^{x}$
be the solution to the SDE (\ref{eq: hypoelliptic SDE}) and consider its approximation $X_{l}(t,x)$ of order $l\ge l_{0}$, as given in \eqref{eq:def-X-l}. Fix $\bar{l}\in(l,l+1)$ and assume that the vector fields $V_\alpha$ are $C_b^\infty$.
There exist constants $C_{1},C_{2}$ depending only on $H,l$ and
the vector fields, such that for all $t\in(0,1]$ and $x,y\in\R^{N}$ we have
\begin{align}\label{eq: tail estimate of difference}
\mathbb{P}\left(|X_{t}^{x}-X_{l}(t,x)|\geq\lambda\right)\leq C_{1}\exp\left(-\frac{C_{2}\lambda^{\frac{2}{\bar{l}}}}{t^{2H}}\right),\ \ \ \text{ for all }\lambda>0.
\end{align}
\end{lem}
\begin{proof}
According to \cite[Corollary 10.15]{FV10}, we have the
following almost sure pathwise estimate
\begin{align}\label{eq: rough path estimate Euler}
|X(t,x)-X_{l}(t,x)|\leq C\cdot\|{\bf B}\|_{p-{\rm var};[0,t]}^{\bar{l}},
\end{align}
with $C=C_{H,V,l}>0$, and where ${\bf B}$ is the rough path
lifting of $B$ alluded to in Proposition \ref{prop:fbm-rough-path}. In equation \eqref{eq: rough path estimate Euler}, the parameter $p$ is any number greater that $1/H$ and the $p$-variation norm is defined with respect to the CC-norm. It follows
from \eqref{eq: rough path estimate Euler} that for any $\lambda>0$ and $\eta>0,$
we have 
\begin{align*}
  \mathbb{P}\left(|X_{t}^{x}-X_{l}(t,x)|\geq\lambda\right)
 \leq\mathbb{P}\left(\|{\bf B}\|_{p-{\rm var};[0,t]}^{\bar{l}}\geq\lambda/C\right).
 \end{align*}
 In addition, the fBm signature satisfies the identity in law
 $$
 ({\bf B}_s)_{0\leq s\leq t}
 \stackrel{d}{=}
 \left(\delta_{t^H}\circ{\bf B}_{\frac{s}{t}}\right)_{0\leq s\leq t}.
 $$
 Owing to the scaling properties of the CC-norm, we thus get that for an arbitrary $\zeta>0$ we have
 \begin{align}
 \mathbb{P}(|X^x_t-X_l(t,x)\geq \lambda)& 
 \leq 
 \mathbb{P}\left(\|{\bf B}\|_{p-{\rm var};[0,1]}\geq\frac{(\lambda/C)^{\frac{1}{\bar{l}}}}{t^{H}}\right)
 \nonumber \\
 & \leq
 \exp\left(-\frac{\zeta(\lambda/C)^{\frac{2}{\bar{l}}}}{t^{2H}}\right)\cdot
 \mathbb{E}\left[{\rm e}^{\zeta\|{\bf B}\|_{p-{\rm var};[0,1]}^{2}}\right],\label{eq: tail estimate for Taylor error}
\end{align}
where we have simply involved Markov's inequality for the last inequality. Now notice that a  Fernique type estimate holds for the fractional Brownian rough path (cf. \cite[Theorem 15.33]{FV10}), namely there exists $\zeta=\zeta_{H}>0$ such that
\[
\mathbb{E}\left[{\rm e}^{\zeta\|{\bf B}\|_{p-{\rm var};[0,1]}^{2}}\right]<\infty.
\]
Plugging this inequality into \eqref{eq: tail estimate for Taylor error}, our conclusion \eqref{eq: tail estimate of difference} is easily obtained. 
\end{proof}

We are now ready to derive a Fourier transform estimate for small values of $\xi$.
\begin{lem}\label{lem: bounded region part}
Keep the same notation and hypothesis as in Lemma \ref{lem: tail estimate for Taylor error}, and also assume that the uniform hypoellipticity condition \eqref{eq:unif-hypo-assumption} is fulfilled. Let $p(t,x,y)$ be the density of the random variable $X_t^x$ and denote by $p_l^\eta$ the approximating density defined by \eqref{eq: formula for p^eta_l}. Then the Fourier transforms $\hat{p}=\cf p$ and $\hat{p}_l^\eta=\cf p_l^\eta$ satisfy the following inequality over the region $\{|\xi|\leq R\}$,
\begin{align}\label{eq: bounded region part}
 \left|\hat{p}(t,x,\xi)-\hat{p}_{l}^{\eta}(t,x,\xi)\right|\leq C_{H,V,l}(1+|\xi|)t^{H\bar{l}},
\end{align}provided that $t<\tau_1$ for some constant $\tau_1$ depending on $H,l$ and the vector fields.
\end{lem}
\begin{proof}Notice that according to our definition \eqref{eq: formula for p^eta_l} of $p^\eta_l$, we have
\begin{align*}
\hat{p}(t,x,\xi)=\mathbb{E}\left[e^{2\pi i\langle \xi, X_t^x\rangle}\right],\quad\text{and}\quad \hat{p}_l^\eta(t,x,\xi)=\mathbb{E}\left[\eta(U_t)e^{2\pi i\langle \xi, X_l(t,x)\rangle}\right].
\end{align*}
Hence it is easily seen that
\begin{eqnarray}
  \left|\hat{p}(t,x,\xi)-\hat{p}_{l}^{\eta}(t,x,\xi)\right| 
 &\leq&
 \mathbb{E}\left[\left|{\rm e}^{-2\pi i\langle\xi,X_{t}^{x}\rangle}-{\rm e}^{-2\pi i\langle\xi,X_{l}(t,x)\rangle}\right|\right]+\mathbb{E}[1-\eta(U_{t})]\nonumber\\
 & \leq&
 2\pi|\xi|\cdot\mathbb{E}\left[|X_{t}^{x}-X_{l}(t,x)|\right]+\mathbb{E}[1-\eta(U_{t})].\label{Fourier difference bound mid step1}
\end{eqnarray}
Now in order to bound the right hand-side of \eqref{Fourier difference bound mid step1}, we first invoke Lemma \ref{lem: tail estimate for Taylor error}. This yields
\begin{eqnarray}
  \mathbb{E}\left[|X_{t}^{x}-X_{l}(t,x)|\right]
 & =&
 \int_{0}^{\infty}\mathbb{P}\left(|X_{t}^{x}-X_{l}(t,x)|\geq\lambda\right)d\lambda\nonumber\\
 & \leq& 
 C_{1}\int_{0}^{\infty}\exp\left(-\frac{C_{2}\lambda^{\frac{2}{\bar{l}}}}{t^{2H}}\right)d\lambda
  =
 C_{3}t^{H\bar{l}}.\label{Fourier difference bound mid step2}
\end{eqnarray}
On the other hand, using a similar argument to the proof of Lemma
\ref{lem: tail estimate for Taylor error}, there exists a strictly positive exponent $\al_{H,l}$ such that
\begin{align}\label{Fourier difference bound mid step3}
\mathbb{E}[1-\eta(U_{t})]\leq\mathbb{P}\left(\|U_{t}\|_{\textsc{HS}}\geq\frac{r}{2}\right)\leq C_{4}\cdot{\rm e}^{-\frac{C_{5}}{t^{\alpha_{H,l}}}}.
\end{align}
Therefore taking $t$ small enough, we can make the right hand-side of \eqref{Fourier difference bound mid step3} smaller than $C_{6}t^{H\bar{l}}$.  Hence there exists $\tau_1>0$ such that if $t\leq \tau_1$ we have
\begin{align}\label{Fourier difference bound mid step4}
\mathbb{E}[1-\eta(U_{t})]\leq C_{6}t^{H\bar{l}}.
\end{align}
Now combining \eqref{Fourier difference bound mid step2} and \eqref{Fourier difference bound mid step4}, we easily get our conclusion \eqref{eq: bounded region part}.
\end{proof}

\paragraph*{II. Integrating relation \eqref{eq: Fourier estimate} for large Fourier modes.}
\label{sec: integrate-large}
\addcontentsline{toc}{subsubsection}{\nameref{sec: integrate-large}}

$\ $\\
\\
We now integrate the Fourier variable $\xi$ over the region $\{|\xi|>R\}$.
In this case, we make use of certain upper estimates
for $p(t,x,y)$ and $p_{l}^{\eta}(t,x,y)$. We start with a bound on the density of $X_{t}^{x}$ which is also of independent
interest. The main ingredients of the proof are basically known in
the literature, but to our best knowledge the result (for the hypoelliptic
case) has not been formulated elsewhere.
\begin{prop}
\label{prop: Gaussian upper bound for p}Let $p(t,x,y)$ be the density of the random variable $X_t^x$ and as in Lem\-ma~\ref{lem: bounded region part} we assume that the uniform hypoellipticity condition  \eqref{eq:unif-hypo-assumption} is satisfied. Then for each $n\geq1,$
there exist constants $C_{1,n},C_{2,n},\nu_{n}>0$ depending on $n,H$
and the vector fields such that 
\begin{align}\label{general density upper bound for hypoellitpic SDE}
|\partial_{y}^{n}p(t,x,y)|\leq C_{1,n}t^{-\nu_{n}}\exp\left(-\frac{C_{2,n}|y-x|^{2\wedge(2H+1)}}{t^{2H}}\right),
\end{align}
for all $(t,x,y)\in(0,1]\times\mathbb{R}^{N}\times\mathbb{R}^{N}$, where $\partial_y^n$ denotes the $n$-th order derivative operator with respect to the $y$ variable.
\end{prop}

\begin{proof}
Elaborating on the integration by parts invoked for example in \cite[Relation (24)]{BNOT16}, there exist exponents $\alpha, \beta, p,q>1$ such that 
\begin{align}\label{IBP density bound}
|\partial_{y}^{n}p(t,x,y)|\leq C_{1,n}\mathbb{P}(|X_t^x-x|\geq|y-x|)^{\frac{1}{2}}\cdot\|\gamma^{-1}_{X_t^x}\|_{\alpha,p}^\alpha\cdot\|\mathbf{D}X_t^x\|_{\beta,q}^\beta,
\end{align}
where the Malliavin covariance matrix $\gamma_{X_t^x}$ is defined by \eqref{malmat} and the Sobolev norm $\|\cdot\|_{k, p}$ is introduced in \eqref{eq:malliavin-sobolev-norm}.  Then with \eqref{IBP density bound} in hand, we proceed in the following way:\\
\\
(i) An exponential tail estimate for $X_t^x$ yield the exponential term in \eqref{general density upper bound for hypoellitpic SDE}. This step is achieved as in \cite[Relation (25)]{BNOT16}.\\
(ii) The Malliavin derivatives of $X_t^x$ are estimated as in \cite[Lemma 3.5 (1)]{BOZ15}. This produces some positive powers of $t$ in \eqref{general density upper bound for hypoellitpic SDE}.\\
(iii) The inverse of the Malliavin covariance matrix is bounded as in \cite[Lemma 3.5 (2)]{BOZ15}. It gives some negative powers of $t$ in \eqref{general density upper bound for hypoellitpic SDE}.\\
\\
For the sake of conciseness, we will not detail the steps outlined as above. We refer the reader to \cite{BNOT16, BOZ15} for the details. 
\end{proof}
We now state a lemma which parallels Proposition \ref{prop: Gaussian upper bound for p} for the approximation $X_l$. Its proof is somehow delicate and is thus postponed to a separate paragraph.
\begin{lem}
\label{lem: upper estimate for approximating density}Assume the same hypothesis as in Proposition \ref{prop: Gaussian upper bound for p}. Recall that the approximating density $p_l^\eta$ is defined by \eqref{eq: formula for p^eta_l}.  { Fix $l\geq l_0$}. Then
for each $n\geq1$ there exists constants $C_n=C_n(H,l)$ and $\gamma_n=\gamma_n(H,l_0)$ such that for all $(t,x)\in(0,1]\times\mathbb{R}^{N}$ the following bound holds true
\begin{align}\label{eq: upper bound approximating density}
\|\partial_{y}^{n}p_{l}^{\eta}(t,x,\cdot)\|_{C_{b}^{n}(\mathbb{R}^{N})}\leq C_{n}\cdot t^{-\gamma_{n}}.
\end{align}
Moreover, the function $\partial_{y}^{n}p_{l}^{\eta}(t,x,\cdot)$ is compactly supported in $\R^{N}$.
\end{lem}

\begin{rem}
Let us highlight the fact that $\gamma_n$ in \eqref{eq: upper bound approximating density} depends on $l_0$ instead of $l$. This subtle technical point is crucial and requires a non-trivial amount of analysis, which is carried out in the next paragraph.
\end{rem}

We now take Lemma \ref{lem: upper estimate for approximating density} for granted and we come back to the estimate (\ref{eq: Fourier estimate}) for
the region $|\xi|>R$.  We are able to state the following result.

\begin{lem}\label{lem: large xi part}
Using the same notation and hypothesis  as in Lemma \ref{lem: bounded region part}, the Fourier transforms $\hat{p}$ and $\hat{p}_l^\eta$ are such that for all $|\xi| > R$ we have
\begin{equation}
|\xi|^{N+2}\left(\left|\widehat{p}(t,x,\xi)\right|+\left|\widehat{p}_{l}^{\eta}(t,x,\xi)\right|\right)\leq C\cdot t^{-\mu},
\label{eq: outer ball part}
\end{equation}
for some strictly positive constants $C=C_{N,H,V,l}$ and $\mu=\mu_{N,H,V,l_{0}}$.
\end{lem}

\begin{proof}According to standard compatibility rules between Fourier transform and differentiation, we have (recall that $\mathcal{F}f$ and $\hat{f}$ are both used to designate the Fourier transform of a function $f$):
\begin{align*}
&|\xi|^{N+2}(|\hat{p}(t,x,\xi)|+|\hat{p}^\eta_l(t,x,\xi)|)\\
&\leq C_N\left(|\mathcal{F}(\partial_y^{N+2}p(t,x,y)|+|\mathcal{F}(\partial_y^{N+2}p^\eta_l(t,x,y)|\right).
\end{align*}
Plugging \eqref{general density upper bound for hypoellitpic SDE} and \eqref{eq: upper bound approximating density} into this relation and using the fact that $\partial_y^{N+2}p^\eta_l(t,x,\cdot)$ is compactly supported, our claim \eqref{eq: outer ball part} is easily proved.
\end{proof}

\paragraph*{III. Comparison of the densities.}
\label{sec: com-den}
\addcontentsline{toc}{subsubsection}{\nameref{sec: com-den}}

$\ $\\
\\
Combining the previous preliminary results on Fourier transforms we get the following uniform bound on the difference $p-p^\eta_l$.
\begin{prop}
We still keep the same notation and assumptions of Lemma \ref{lem: bounded region part}. Then there exists $\tau>0$ such that for all $t\leq \tau$ and $x,y\in\mr^N$ we have
\begin{align}\label{eq: comparison estimate}
|p(t,x,y)-p_l^\eta(t,x,y)|\leq C_{H,V,l} \,t.
\end{align}
\end{prop}

\begin{proof}
Thanks to \eqref{eq: Fourier estimate} we can write 
\begin{align*}
|p(t,x,y)-p_{l}^{\eta}(t,x,y)|\leq\left(\int_{|\xi|\leq\mr}+\int_{|\xi|>\mr}\right)\left|\hat{p}(t,x,\xi)-\hat{p}_{l}^{\eta}(t,x,\xi)\right|d\xi.
\end{align*}
Next we invoke the bounds (\ref{eq: bounded region part}) and (\ref{eq: outer ball part}), which allows to write
\begin{align*}
  \left|p(t,x,y)-p_{l}^{\eta}(t,x,y)\right|
  \leq C_{1}\left(R^{N+1}t^{H\bar{l}}+t^{-\mu}\int_{|\xi|>R}|\xi|^{-N-2}d\xi\right),
\end{align*}
whenever $t\in(0,1]$, and where we recall that $\bar{l}$ is a fixed number in $[l,l+1]$ introduced in Lemma \ref{lem: tail estimate for Taylor error}. Now an elementary change of variable yields
\begin{align}
 & \left|p(t,x,y)-p_{l}^{\eta}(t,x,y)\right|\nonumber\\
 & \leq C_{1}\left(R^{N+1}t^{H\bar{l}}+t^{-\mu}R^{-1}\int_{|\xi|\geq1}|\xi|^{-(N+1)}d\xi\right)\nonumber\\
 & \leq C_{2}\left(R^{N+1}t^{H\bar{l}}+t^{-\mu}R^{-1}\right).\label{eq: bound difference of densities mid step1}
\end{align}
We can easily optimize expression \eqref{eq: bound difference of densities mid step1} with respect to $R$ by choosing  $R=t^{-(\mu+1)}$ . It follows that
\begin{align}\label{eq: bound difference of densities mid step 2}
\left|p(t,x,y)-p_{l}^{\eta}(t,x,y)\right|\leq C_{2}t^{-(N+1)(\mu+1)+H\overline{l}}+t,
\end{align}
for all  $t\in(0,1]$.  In addition, recall that a crucial point in our approach is that the exponent $\mu$ in \eqref{eq: bound difference of densities mid step 2} does not depend on $l$. Therefore  we can choose
$l\geq l_{0}$ large enough, so that 
\[
-(N+1)(\mu+1)+H\overline{l}\geq1.
\]
For this value of $l$, the upper bound \eqref{eq: comparison estimate} is easily deduced from \eqref{eq: bound difference of densities mid step 2}.


\end{proof}
\paragraph*{IV. A new proof of Lemma \ref{lem: upper estimate for approximating density}.}
\label{sec: uniform}
\addcontentsline{toc}{subsubsection}{\nameref{sec: uniform}}

$\ $\\
\\
Before proving Lemma \ref{lem: upper estimate for approximating density},
we mention that the independence on $l$ for the exponent $\gamma_{n}$
was already observed in \cite{KS87} for the
diffusion case. However, in our modest opinion we believe that there
is a gap in the argument. We explain the reason as follows. Recall that for a
differentiable random vector $Z=(Z^{1},\ldots,Z^{n})$ in the sense
of Malliavin, we use the notation $\gamma_{Z}\triangleq(\langle DZ^{i},DZ^{j}\rangle_{\cal H}){}_{1\leq i,j\leq n}$
to denote its Malliavin covariance matrix. By the definition \eqref{eq:def-X-l} of
$X_{l}(t,x),$ it is immediate that 
\begin{align}\label{equ: Malliavin matrix X_l}
\gamma_{X_{l}(t,x)}=JF_{l}(U_{t}^{(l)},x)\cdot \gamma_{U_{t}^{(l)}}\cdot JF_{l}(U_{t}^{(l)},x)^{*}.
\end{align}
Next, let $\pi_{l,l_{0}}:\mathfrak{g}^{(l)}\rightarrow\mathfrak{g}^{(l_{0})}\subseteq\mathfrak{g}^{(l)}$
be the canonical orthogonal projection. The matrix form of $\pi_{l,l_0}$ as a linear function on $\frak{g}{(l)}$ is given by 
\[
\pi_{l,l_{0}}=\left(\begin{array}{cc}
{\rm Id}_{\mathfrak{g}^{(l_{0})}} & 0\\
0 & 0
\end{array}\right).
\]
 In \cite[Page 420]{KS87}, it was asserted that 
\begin{equation}
JF_{l}(U_{t}^{(l)},x)\cdot \gamma_{U_{t}^{(l)}}\cdot JF_{l}(U_{t}^{(l)},x)^{*}\geq JF_{l}(U_{t}^{(l)},x)\cdot\pi_{l,l_{0}}\cdot \gamma_{U_{t}^{(l)}}\cdot\pi_{l,l_{0}}\cdot JF_{l}(U_{t}^{(l)},x)^{*},\label{eq: KS inequality for M matrices}
\end{equation}
which we believe was crucial for proving the $l$-independence in the argument. Now if we write 
\[
\gamma_{U_{t}^{(l)}}=\left(\begin{array}{cc}
\gamma_{U_{t}^{(l_{0})}} & P\\
Q & R
\end{array}\right),
\]
then it is readily checked that (\ref{eq: KS inequality for M matrices}) is equivalent to 
\[
JF_{l}(U_{t}^{(l)},x)\cdot\left(\begin{array}{cc}
0 & P\\
Q & R
\end{array}\right)\cdot JF_{l}(U_{t}^{(l)},x)^{*}\geq0.
\]
However, we do not see a reason why this nonnegative definiteness property
should hold even if we know that the Malliavin covariance matrices are always nonnegative
definite. Therefore the considerations below  are devoted to an alternative proof of
Lemma \ref{lem: upper estimate for approximating density} in the
fractional Brownian setting, which also covers the diffusion case.

In view of the decomposition \eqref{equ: Malliavin matrix X_l}, we first need the following lemma from \cite{BFO19},
which gives an estimate of the Malliavin covariance matrix of $U_t^{(l)}$. 

\begin{lem}\label{lem: estimating eigenvalue of M matrix for signature}
Given $l\geq1,$ let $U_t^{(l)}$ be the truncated log-signature defined by \eqref{def: log signature of B}. We consider the Malliavin covariance matrix $\gamma_{U_t^{(l)}}$ of the random variable $U_t^{(l)}$, and denote by  $\mu_{t}^{(l)}$  the smallest eigenvalue. Then for any $q>1,$ we have 
\begin{align}\label{eq: Malliavin matrix U_t^k order}
\sup_{t\in(0,1]}\left\Vert \frac{t^{2Hl}}{\mu_{t}^{(l)}}\right\Vert _{q}<\infty.
\end{align}
\end{lem}

Now we are able to give the proof of Lemma \ref{lem: upper estimate for approximating density}.

\begin{proof}[Proof of Lemma \ref{lem: upper estimate for approximating density}]

As mentioned earlier, the uniform upper bound for the derivatives
of $p_{l}^{\eta}(t,x,y)$ follows from the same lines as in the proof
of Proposition \ref{prop: Gaussian upper bound for p} (with the same three
main ingredients (i)-(ii)-(iii)), based on the integration by parts formula. In the
remainder of the proof, we show that the exponent $\gamma_{n}$ can be chosen depending
only on $l_{0}$ but not on $l$ (note, however, that it also depends on $H,n$
and the vector fields). We now divide the proof in several steps.

\noindent{\it Step 1: A decomposition based on lowest eigenvalue.} As recalled in \eqref{IBP density bound} and the strategy of proof of Proposition \ref{prop: Gaussian upper bound for p}, the exponent $\gamma_{n}$ in \eqref{eq: upper bound approximating density} comes from integrability estimates for
the inverse of the Malliavin covariance matrix of $X_{l}(t,x)$. To prove the claim,
by the definition of $\mathbb{P}_{l}^{\eta}(t,x,\cdot),$ it is sufficient
to establish the following property: for each $q>1,$ we have
\begin{equation}
\sup_{t\in(0,1]}\mathbb{E}\left[\left|\frac{t^{2Hl_{0}}}{\lambda_{t}^{(l)}}\right|^{q};\|U_{t}^{(l)}\|_{\textsc{HS}}<r\right]<\infty,\label{eq: uniform sharp bound on inverse Malliavin matrix}
\end{equation}
where the random variable $\lambda_t^{(l)}$ is defined by
\[
\lambda_{t}^{(l)}\triangleq\inf_{\eta\in S^{N-1}}\langle\eta,\gamma_{X_{l}(t,x)}\eta\rangle_{\mathbb{R}^{N}},
\]
that is, $\lambda_t^{(l)}$ is the smallest eigenvalue of $\gamma_{X_{l}(t,x)}.$  In order to lower bound $\lambda_t^{(l)}$, we write $X_{l}(t,x)=X_{l_{0}}(t,x)+R_{t},$ where by the definition \eqref{eq:def-X-l}
of $X_{l}(t,x)$ we have 
\begin{align}\label{eq: R_t}
R_{t}\triangleq\sum_{l_{0}<|\alpha|\leq l}V_{(\alpha)}(x)(\exp U_{t}^{(l)})^{\alpha}.
\end{align}
Then for every $\eta\in S^{N-1},$ we have
\begin{align*}
  \langle\eta,\gamma_{X_{l}(t,x)}\eta\rangle_{\mathbb{R}^{N}}
 & =\left\Vert D\left(\langle\eta,X_{l}(t,x)\rangle_{\mathbb{R}^{N}}\right)\right\Vert _{\bar{{\cal H}}}^{2}\\
 & =\left\Vert D\left(\langle\eta,X_{l_{0}}(t,x)\rangle_{\mathbb{R}^{N}}\right)+D\left(\langle\eta,R_{t}\rangle_{\mathbb{R}^{N}}\right)\right\Vert _{\bar{{\cal H}}}^{2}.
\end{align*}
By invoking the definition \eqref{malmat} of $\gamma_{X_{l_0}}$, we get
\begin{align*}
  \langle\eta,\gamma_{X_{l}(t,x)}\eta\rangle_{\mathbb{R}^{N}} & \geq\frac{1}{2}\left\Vert D\left(\langle\eta,X_{l_{0}}(t,x)\rangle_{\mathbb{R}^{N}}\right)\right\Vert _{\bar{{\cal H}}}^{2}-\left\Vert D\left(\langle\eta,R_{t}\rangle_{\mathbb{R}^{N}}\right)\right\Vert _{\bar{{\cal H}}}^{2}\\
 & =\frac{1}{2}\langle\eta,\gamma_{X_{l_{0}}(t,x)}\eta\rangle_{\mathbb{R}^{N}}-\langle\eta,\gamma_{R_{t}}\eta\rangle_{\mathbb{R}^{N}}\\
 & \geq\frac{1}{2}\langle\eta,\gamma_{X_{l_{0}}(t,x)}\eta\rangle_{\mathbb{R}^{N}}-\|\gamma_{R_{t}}\|_{{\rm F}},
\end{align*}
where $\|\cdot\|_\mathrm{F}$ denotes the Frobenius norm of a matrix, and we have used the simple inequality $\|a+b\|_{E}^{2}\geq\frac{1}{2}\|a\|_{E}^{2}-\|b\|_{E}^{2}$
which is valid in any Hilbert space $E.$ It follows that 
\begin{equation}
\lambda_{t}^{(l)}\geq\frac{1}{2}\lambda_{t}^{(l_{0})}-\|M_{R_{t}}\|_{{\rm F}}.\label{eq: lower bound on smallest eigenvalue}
\end{equation}

In order to go from \eqref{eq: lower bound on smallest eigenvalue} to our desired estimate \eqref{eq: uniform sharp bound on inverse Malliavin matrix}, consider $q>1$ and the following decomposition:
\begin{align}\label{eq: decomposition I_t+J_t}
\mathbb{E}\left[\left|\frac{t^{2Hl_{0}}}{\lambda_{t}^{(l)}}\right|^{q};\|U_{t}^{(l)}\|_{\textsc{HS}}<r\right]=I_t+J_t,
 \end{align}
 where $I_t$ and $J_t$ are respectively defined by
  \begin{align*}
 & I_t=\mathbb{E}\left[\left|\frac{t^{2Hl_{0}}}{\lambda_{t}^{(l)}}\right|^{q};\frac{1}{2}\lambda_{t}^{(l_{0})}-\|M_{R_{t}}\|_{{\rm F}}\geq\frac{1}{4}\lambda_{t}^{(l_{0})},\ \|U_{t}^{(l)}\|_{\textsc{HS}}<r\right],\\
 & J_t=\mathbb{E}\left[\left|\frac{t^{2Hl_{0}}}{\lambda_{t}^{(l)}}\right|^{q};\frac{1}{2}\lambda_{t}^{(l_{0})}-\|M_{R_{t}}\|_{\mathrm{F}}<\frac{1}{4}\lambda_{t}^{(l_{0})},\ \|U_{t}^{(l)}\|_{\textsc{HS}}<r\right].
\end{align*}
Now we estimate $I_{t}$ and $J_{t}$ separately. 

\bigskip

\noindent{\it Step 2: Upper bound for $I_t$.} To estimate $I_{t}$, observe that according to (\ref{eq: lower bound on smallest eigenvalue}) we have
\begin{align}\label{eq: upper bound I_t 1}
I_{t}\leq\mathbb{E}\left[\left|\frac{4t^{2Hl_{0}}}{\lambda_{t}^{(l_{0})}}\right|^{q};\|U_{t}^{(l)}\|_{\textsc{HS}}<r\right].
\end{align}
Furthermore, since $X_{l_{0}}(t,x)=x+F_{l_{0}}(U_{t}^{(l_{0})},x)$, we know that
\[
\gamma_{X_{l_{0}}(t,x)}=JF_{l_{0}}(U_{t}^{(l_{0})},x)\cdot \gamma_{U_{t}^{(l_{0})}}\cdot JF_{l_{0}}(U_{t}^{(l_{0})},x)^{*}.
\]
Therefore, for each $\eta\in S^{N-1},$
\begin{align}
 & \langle\eta,\gamma_{X_{l_{0}}(t,x)}\eta\rangle_{\mathbb{R}^{N}}\nonumber \\
 & =\eta^{*}\cdot JF_{l_{0}}(U_{t}^{(l_{0})},x)\cdot \gamma_{U_{t}^{(l_{0})}}\cdot JF_{l_{0}}(U_{t}^{(l_{0})},x)^{*}\eta\nonumber\\
 & \geq\mu_{t}^{(l_{0})}\cdot\left(\eta^{*}\cdot JF_{l_{0}}(U_{t}^{(l_{0})},x)\cdot JF_{l_{0}}(U_{t}^{(l_{0})},x)^{*}\cdot\eta\right),\label{eq: eigenvalue of JF}
\end{align}
where recall that $\mu_{t}^{(l_{0})}$ denotes the smallest eigenvalue
of $\gamma_{U_{t}^{(l_{0})}}.$ In addition, we choose the constant $r$ in \eqref{eq: upper bound I_t 1} as in Corollary \ref{cor: local submersion} and Remark \ref{rem: consistent choice of r}. We hence know that the matrix
\[
JF_{l_{0}}(\pi^{(l_{0})}(u),x)\cdot JF_{l_{0}}(\pi^{(l_{0})}(u),x)
\]
is uniformly positive definite on $\{(u,x)\in\mathfrak{g}^{(l)}\times\mathbb{R}^{N}:\|u\|_{\textsc{HS}}<r\}.$
In particular, there exists a constant $c_{V,l}>0,$ such that on
the event $\{\|U_{t}^{(l)}\|_{\textsc{HS}}<r\}$, we have
\[
\eta^{*}\cdot JF_{l_{0}}(U_{t}^{(l_{0})},x)\cdot JF_{l_{0}}(U_{t}^{(l_{0})},x)^{*}\cdot\eta\geq c_{V,l}|\eta|^{2},\ \ \forall\eta\in\mathbb{R}^{N},
\]
Therefore, according to (\ref{eq: eigenvalue of JF}), we conclude
that on the event $ \{\|U_{t}^{(l)}\|_{\textsc{HS}}<r\}$ we have
\begin{align}\label{eq: lower bound lambda_t^l_0}
\lambda_{t}^{(l_{0})}=\inf_{\eta\in S^{N-1}}\langle\eta,M_{X_{l_{0}}(t,x)}\eta\rangle_{\mathbb{R}^{N}}\geq c_{V,l}\mu_{t}^{(l_{0})},
\end{align}
where we recall that $\mu_t^{(l_0)}$ is the smallest eigenvalue of $\gamma_{U_t^{(l_0)}}$.  Putting \eqref{eq: lower bound lambda_t^l_0} into \eqref{eq: upper bound I_t 1} it follows that
\[
I_{t}\leq\mathbb{E}\left[\left|\frac{4t^{2Hl_{0}}}{c_{V,l}\mu_{t}^{(l_{0})}}\right|^{q}:\|U_{t}^{(l)}\|_{\textsc{HS}}<r\right].
\]
Hence a direct application of Lemma \ref{lem: estimating eigenvalue of M matrix for signature} yields
\[
\sup_{t\in(0,1]}I_{t}<\infty.
\]

\noindent{\it Step 3: Upper bound for $J_t$.} 
To estimate $J_{t}$, according to H\"older's inequality, we have 
\begin{align}
J_{t} & \leq
\left(\mathbb{E}\left[\left|\frac{t^{2Hl_{0}}}{\lambda_{t}^{(l)}}\right|^{2q}\right];\|U_{t}^{(l)}\|_{\textsc{HS}}<r\right)^{\frac{1}{2}}\cdot
\left(\mathbb{P}\left(\frac{1}{2}\lambda_{t}^{(l_{0})}-\|M_{R_{t}}\|_{\mathrm{F}}<\frac{1}{4}\lambda_{t}^{(l_{0})}
;\|U_{t}^{(l)}\|_{\textsc{HS}}<r
\right)\right)^{\frac{1}{2}}\nonumber\\
 & =\left(\mathbb{E}\left[\left|\frac{t^{2Hl_{0}}}{\lambda_{t}(l)}\right|^{2q}\right];\|U_{t}^{(l)}\|_{\textsc{HS}}<r\right)^{\frac{1}{2}}\cdot\left(\mathbb{P}\left(\|M_{R_{t}}\|_{\mathrm{F}}>\frac{1}{4}\lambda_{t}^{(l_{0})},\ \|U_{t}^{(l)}\|_{\textsc{HS}}<r\right)\right)^{\frac{1}{2}}.\label{eq: upper bound J_t 1}
\end{align}
On the one hand, according to (\ref{eq: lower bound lambda_t^l_0}) applied to general $l$ and Lemma \ref{lem: estimating eigenvalue of M matrix for signature},
we have
\[
C_{1,q,l}\triangleq\sup_{t\in(0,1]}\mathbb{E}\left[\left|\frac{t^{2Hl}}{\lambda_{t}^{(l)}}\right|^{2q}:\|U_{t}^{(l)}\|_{\textsc{HS}}<r\right]<\infty.
\]
It follows that 
\begin{align}
\mathbb{E}\left[\left|\frac{t^{2Hl_{0}}}{\lambda_{t}^{(l)}}\right|^{2q}:\|U_{t}^{(l)}\|_{\textsc{HS}}<r\right] & =\frac{1}{t^{4qH(l-l_{0})}}\mathbb{E}\left[\left|\frac{t^{2Hl}}{\lambda_{t}^{(l)}}\right|^{2q}:\|U_{t}^{(l)}\|_{\textsc{HS}}<r\right]\nonumber \\
 & \leq\frac{C_{1,q,l}}{t^{4qH(l-l_{0})}}.\label{eq: first bit}
\end{align}
In order to bound the right hand-side of \eqref{eq: upper bound J_t 1}, we also write 
\begin{equation*}
  \mathbb{P}\left(\|M_{R_{t}}\|_{\mathrm{F}}>\frac{1}{4}\lambda_{t}^{(l_{0})},\ \|U_{t}^{(l)}\|_{\textsc{HS}}<r\right)
  =\mathbb{P}\left(t^{-2Hl_{0}}\|M_{R_{t}}\|_{\mathrm{F}}\geq\frac{1}{4}t^{-2Hl_{0}}\lambda_{t}^{(l_{0})},\ \|U_{t}^{(l)}\|_{\textsc{HS}}<r\right).
\end{equation*}
Now a crucial observation is that $R_{t}$ is defined in terms of
signature components of order at least $l_{0}+1$ for the fractional
Brownian motion, as easily seen from \eqref{eq: R_t}. According to the scaling property of the signature, if we define $\xi_{t}\triangleq t^{-2H(l_{0}+1)}\|M_{R_{t}}\|_{\mathrm{F}}$,
then $\xi_{t}$ has moments of all orders uniformly in $t\in(0,1].$
It follows that 
\begin{align}
 & \mathbb{P}\left(\|M_{R_{t}}\|_{\mathrm{F}}>\frac{1}{4}\lambda_{t}^{(l_{0})},\ \|U_{t}^{(l)}\|_{\textsc{HS}}<r\right)\nonumber \\
 & =\mathbb{P}\left(t^{2H}\xi_{t}\geq\frac{1}{4}t^{-2Hl_{0}}\lambda_{t}^{(l_{0})},\ \|U_{t}^{(l)}\|_{\textsc{HS}}<r\right)\nonumber \\
 & \leq\mathbb{E}\left[\left|\frac{4t^{2H}\xi_{t}}{t^{-2Hl_{0}}\lambda_{t}^{(l_{0})}}\right|^{2q(l-l_{0})};\ \|U_{t}^{(l)}\|_{\textsc{HS}}<r\right]\nonumber \\
 & =t^{4qH(l-l_{0})}\mathbb{E}\left[\left|\frac{4\xi_{t}t^{2Hl_{0}}}{\lambda_{t}(l_{0})}\right|^{2q(l-l_{0})};\ \|U_{t}^{(l)}\|_{\textsc{HS}}<r\right]\nonumber \\
 & \leq t^{4qH(l-l_{0})}\left(\mathbb{E}[|4\xi_{t}|^{4q(l-l_{0})}]\right)^{\frac{1}{2}}\cdot\left(\mathbb{E}\left[\left|\frac{t^{2Hl_{0}}}{\lambda_{t}(l_{0})}\right|^{4q(l-l_{0})}\right]\right)^{\frac{1}{2}}\nonumber \\
 & \leq C_{2,q,l}t^{4qH(l-l_{0})}.\label{eq: second bit}
\end{align}
Plugging (\ref{eq: first bit}) and (\ref{eq: second bit}) into \eqref{eq: upper bound J_t 1}, we arrive
at 
\[
J_{t}\leq\sqrt{C_{1,q,l}\cdot C_{2,q,l}}<\infty,\ \ \text{for\ all}\  t\in(0,1].
\]

\bigskip

\noindent{\it Step 4: Conclusion.}  Putting together our estimates on $I_t$ and $J_t$ and inserting them into~\eqref{eq: decomposition I_t+J_t}, our claim \eqref{eq: uniform sharp bound on inverse Malliavin matrix} is readily proved. 

\end{proof}

\subsection{Completing the proof of Theorem \ref{thm: local lower estimate}.}

Finally, we are in a position to complete the proof of Theorem \ref{thm: local lower estimate}.
Indeed, recall that (\ref{eq: lower estimate for approximating density})
and (\ref{eq: comparison estimate}) assert that for $x$ and $y$ such that $d(x,y)\leq t^H$ and $t<\tau$ we have
\begin{align}\label{f1}
p^\eta(t,x,y)\geq \frac{C_{1}}{|B_d(x,t^H)|},\quad\text{and}\quad |p(t,x,y)-p^\eta_l(t,x,y)|\leq C_{H,V,l}t.
\end{align}
In addition, owing to \eqref{eq:local-comparison}, for small $t$ we get
\begin{equation}\label{f2}
\frac{1}{|B_d(x,t^H)|}\geq\frac{C}{t^{HN/l_0}}.
\end{equation}
Putting together \eqref{f1} and \eqref{f2}, it is thus easily seen that when $t$ is small enough we have
\[
p(t,x,y)\geq\frac{C_{2}}{|B_{d}(x,t^{H})|}.
\]
The proof of Lemma \ref{lem: upper estimate for approximating density} is thus complete.

\begin{appendices}
\section{A disintegration formula on Riemannian manifolds.}

Since we are not aware of a specific reference in the literature, for completeness we include a proof of a general disintegration formula on Riemannian manifolds, which  is used for Proposition \ref{th: disintegration}.

Recall that, if $V$ is an $m$-dimensional real inner product space, then for
each $0\leqslant p\leqslant m,$ the $p$-th exterior power $\Lambda^{p}V$
of $V$ carries an inner product structure defined by 
\[
\langle v_{1}\wedge\cdots\wedge v_{p},w_{1}\wedge\cdots\wedge w_{p}\rangle_{\Lambda^{p}V}\triangleq\det\left(\langle v_{i},w_{j}\rangle_{1\leqslant i,j\leqslant p}\right).
\]
In particular, if $M$ is a Riemannian manifold, then for each $p$,
the space of differential $p$-forms carries a canonical pointwise
inner product structure induced from the Riemannian structure of $M$.
A norm on differential $p$-forms is thus defined pointwisely on $M$.

Let $N$ be an oriented $n$-dimensional Riemannian manifold. Suppose
that $F:\ M\triangleq\mathbb{R}^{m}\rightarrow N$ ($m\geqslant n$)
is a non-degenerate $C^{\infty}$-map in the sense that $(dF)_{p}$
is surjective everywhere. Then we know that for each $q\in F(\mathbb{R}^{m}),$
$F^{-1}(q)$ is a closed submanifold of $\mathbb{R}^{m},$ which
carries a canonical Riemannian structure induced from $\mathbb{R}^{m}$. From
differential topology we also know that $F^{-1}(q)$ carries a natural
orientation induced from the ones on $\mathbb{R}^{m}$ and
$N$. In particular, the volume form on $F^{-1}(q)$ is well-defined
for every $q$. 

Now we have the following disintegration formula.
\begin{thm}
Let $\mathrm{vol}_{N}$ be the volume form on $N$. Then for every
$\varphi\in C_{c}^{\infty}(\mathbb{R}^{m}),$ we have 
\begin{equation}
\int_{\mathbb{R}^{m}}\varphi(x)dx=\int_{q\in N}\mathrm{vol}_{N}(dq)\int_{x\in F^{-1}(q)}\frac{\varphi(x)}{\|F^*\mathrm{vol}_{N}\|}\mathrm{vol}_{F^{-1}(q)}(dx),\label{eq: general disintegration}
\end{equation}
where $\mathrm{vol}_{F^{-1}(q)}\triangleq 0$ if $F^{-1}(q)=\emptyset.$
\end{thm}

\begin{proof}
By a partition of unity argument, it suffices to prove the formula
locally under coordinate charts on $N$. Fix $p\in\mathbb{R}^{m}$ and
$q\triangleq F(p)\in N.$ Let $(V;y^{i})$ be a chart around $q$.
Then the Jacobian matrix $\frac{\partial y}{\partial x}$ has full
rank (i.e. rank $n$) at $p.$ Without loss of generality, we may
assume that 
\[
\frac{\partial y}{\partial x_{1}}\triangleq\left(\frac{\partial y^{i}}{\partial x^{j}}\right)_{1\leqslant i,j\leqslant n}
\]
is non-degenerate, where we write $x_{1}=(x^{1},\cdots,x^{n})$ and
$x_{2}=(x^{n+1},\cdots,x^{m})$. Define a map $\overline{F}:\ \mathbb{R}^{m}\rightarrow N\times\mathbb{R}^{m-n}$
by $\overline{F}(x_{1},x_{2})\triangleq(F(x_{1},x_{2}),x_{2}).$ It
follows that locally around $p$$,$ we have 
\[
\frac{\partial\overline{F}}{\partial x}=\left(\begin{array}{cc}
\frac{\partial y}{\partial x_{1}} & \frac{\partial y}{\partial x_{2}}\\
0 & \mathrm{I}_{m-n}
\end{array}\right).
\]
In particular, $\frac{\partial\overline{F}}{\partial x}$ is non-denegerate
at $p$. Therefore, $\overline{F}$ defines a local diffeomorphism
between $p\in U$ and $W=V'\times(a,b)$ for some $U$, $V'\subseteq V$
and $(a,b)\subseteq\mathbb{R}^{m-n}.$ We use $(y,z)\in V'\times(a,b)$
to denote the new coordinates on $U\subseteq\mathbb{R}^{m}.$ Note
that every slice $\{(y,z)\in W:\ y=y_{0}\}$ ($y_{0}\in V'$) defines
a parametrization of the fiber $F^{-1}(y_{0})\cap U$.

By change of coordinates from $x=(x_{1},x_{2})$ to $(y,z),$ we have
\begin{equation}
dx=\frac{1}{\det\left(\frac{\partial\overline{F}}{\partial x}\right)}dy\wedge dz=\frac{1}{\det\left(\frac{\partial y}{\partial x_{1}}\right)}dy\wedge dz.\label{eq: computing dx in disintegration}
\end{equation}
Since for each $y\in V'$, $z\in(a,b)\mapsto\overline{F}^{-1}(y,z)\in F^{-1}(y)$
defines of parametrization of the fiber $F^{-1}(y)\cap U$, we know
that 
\[
dz=\frac{\mathrm{vol}_{F^{-1}(y)}}{\sqrt{\det\left(\left(\langle\partial_{i}z,\partial_{j}z\rangle\right)_{n+1\leqslant i,j\leqslant m}\right)}}
\]
for each fixed $y\in V'$, where the inner product is defined by the
induced Riemannian structure on $F^{-1}(y).$ But we know that $\{\partial_{i}x:\ 1\leqslant i\leqslant m\}$
is an orthonormal basis of $T_{x}\mathbb{R}^{m}$ for every $x$.
Therefore, 
\begin{align*}
\langle\partial_{i}z,\partial_{j}z\rangle & =\sum_{\alpha,\beta=1}^{m}\frac{\partial x^{\alpha}}{\partial z^{i}}\frac{\partial x^{\beta}}{\partial z^{j}}\langle\partial_{\alpha}x,\partial_{\beta}x\rangle\\
 & =\sum_{\alpha=1}^{m}\frac{\partial x^{\alpha}}{\partial z^{i}}\frac{\partial x^{\alpha}}{\partial z^{j}}.
\end{align*}
It follows that 
\[
\left(\langle\partial_{i}z,\partial_{j}z\rangle\right)_{n+1\leqslant i,j\leqslant m}=\left(\frac{\partial x}{\partial z}\right)^{*}\cdot\frac{\partial x}{\partial z}.
\]

On the other hand, we know that 
\[
\left(\begin{array}{cc}
\frac{\partial x_{1}}{\partial y} & \frac{\partial x_{1}}{\partial z}\\
\frac{\partial x_{2}}{\partial y} & \frac{\partial x_{2}}{\partial z}
\end{array}\right)\cdot\left(\begin{array}{cc}
\frac{\partial y}{\partial x_{1}} & \frac{\partial y}{\partial x_{2}}\\
0 & \mathrm{I}_{m-n}
\end{array}\right)=\mathrm{I}_{m}.
\]
By comparing components, we get
\[
\begin{cases}
\frac{\partial x_{1}}{\partial y}=\left(\frac{\partial y}{\partial x_{1}}\right)^{-1},\\
\frac{\partial x_{1}}{\partial z}=-\left(\frac{\partial y}{\partial x_{1}}\right)^{-1}\cdot\frac{\partial y}{\partial x_{2}},\\
\frac{\partial x_{2}}{\partial y}=0,\\
\frac{\partial x_{2}}{\partial z}=\mathrm{I}_{m-n}.
\end{cases}
\]
Therefore, 
\[
\frac{\partial x}{\partial z}=\left(\begin{array}{c}
-\left(\frac{\partial y}{\partial x_{1}}\right)^{-1}\cdot\frac{\partial y}{\partial x_{2}}\\
\mathrm{I}_{m-n}
\end{array}\right),
\]
and
\begin{align*}
\det\left(\left(\langle\partial_{i}z,\partial_{j}z\rangle\right)_{n+1\leqslant i,j\leqslant m}\right) & =\left(\begin{array}{cc}
-\left(\frac{\partial y}{\partial x_{2}}\right)^{*}\left(\frac{\partial y}{\partial x_{1}}\right)^{*-1} & \mathrm{I}_{m-n}\end{array}\right)\cdot\left(\begin{array}{c}
-\left(\frac{\partial y}{\partial x_{1}}\right)^{-1}\cdot\frac{\partial y}{\partial x_{2}}\\
\mathrm{I}_{m-n}
\end{array}\right)\\
 & =\det\left(\left(\frac{\partial y}{\partial x_{2}}\right)^{*}\left(\frac{\partial y}{\partial x_{1}}\right)^{*-1}\left(\frac{\partial y}{\partial x_{1}}\right)^{-1}\frac{\partial y}{\partial x_{2}}+\mathrm{I}_{m-n}\right)\\
 & =\det\left(\left(\frac{\partial y}{\partial x_{1}}\right)^{*-1}\left(\frac{\partial y}{\partial x_{1}}\right)^{-1}\frac{\partial y}{\partial x_{2}}\left(\frac{\partial y}{\partial x_{2}}\right)^{*}+\mathrm{I}_{n}\right),
\end{align*}
where in the last equality we have used Sylvester's determinant identity
(i.e. $\det(\mathrm{I}_{m}+AB)=\det(\mathrm{I}_{n}+BA)$ if $A,B$
are  $m\times n$ and $n\times m$ matrices respectively).

Consequently, according to (\ref{eq: computing dx in disintegration}),
we get 
\begin{align*}
dx & =\frac{1}{\det\left(\frac{\partial y}{\partial x_{1}}\right)\cdot\sqrt{\det\left(\left(\langle\partial_{i}z,\partial_{j}z\rangle\right)_{n+1\leqslant i,j\leqslant m}\right)}}dy\wedge\mathrm{vol}_{F^{-1}(y)}\\
 & =\frac{1}{\sqrt{\det\left(\frac{\partial y}{\partial x_{1}}\left(\frac{\partial y}{\partial x_{1}}\right)^{*}\right)\cdot\det\left(\left(\frac{\partial y}{\partial x_{1}}\right)^{*-1}\left(\frac{\partial y}{\partial x_{1}}\right)^{-1}\frac{\partial y}{\partial x_{2}}\left(\frac{\partial y}{\partial x_{2}}\right)^{*}+\mathrm{I}_{n}\right)}}dy\wedge\mathrm{vol}_{F^{-1}(y)}\\
 & =\frac{1}{\sqrt{\det\left(\frac{\partial y}{\partial x_{1}}\left(\frac{\partial y}{\partial x_{1}}\right)^{*}+\frac{\partial y}{\partial x_{2}}\left(\frac{\partial y}{\partial x_{2}}\right)^{*}\right)}}dy\wedge\mathrm{vol}_{F^{-1}(y)}\\
 & =\frac{1}{\sqrt{\det\left(\frac{\partial y}{\partial x}\left(\frac{\partial y}{\partial x}\right)^{*}\right)}}dy\wedge\mathrm{vol}_{F^{-1}(y)}.
\end{align*}
But we also know that 
\begin{align*}
\|F^{*}dy\| & =\sqrt{\langle dy^{1}\wedge\cdots\wedge dy^{n},dy^{1}\wedge\cdots\wedge dy^{n}\rangle}\\
 & =\sqrt{\mathrm{det}\left(\left(\langle dy^{i},dy^{j}\rangle\right)_{1\leqslant i,j\leqslant n}\right)}\\
 & =\sqrt{\mathrm{det}\left(\frac{dy}{dx}\left(\frac{dy}{dx}\right)^{*}\right)}.
\end{align*}
Therefore, we arrive at 
\begin{align*}
dx & =\frac{1}{\|F^{*}dy\|}dy\wedge\mathrm{vol}_{F^{-1}(y)}\\
 & =\frac{1}{\|F^{*}\mathrm{vol}_{N}\|}\mathrm{vol}_{N}\wedge\mathrm{vol}_{F^{-1}(y)}
\end{align*}
on $U$, where in the last equality we have used the fact that $F^{*}$
is a linear map.

Now the proof of the theorem is complete.
\end{proof}
\begin{rem}
Note that the disintegration formula (\ref{eq: general disintegration})
is intrinsic, i.e. it does not depend on coordinates over $N$.
However, the formula is not true when $M$ is not flat. Indeed, the
left hand side of the formula depends on the entire Riemannian structure
of $M$ since the volume form is defined in terms of the Riemannian
metric on $M$. However, the right hand side of the formula depends
only on the Riemannian structure of $N$ and of those fibers. In general,
the volume form on $M$ cannot be recovered intrinsically from the
geometry of $N$ and the geometry of those fibers.
\end{rem}
A particularly useful case of the disintegration formula is
when $N=\mathbb{R}^{n}.$ In this case, the formula reads
\begin{equation*}
\int_{\mathbb{R}^{m}}\varphi(x)dx=\int_{y\in\mathbb{R}^{m}}dy\int_{F^{-1}(y)}\frac{\varphi(x)}{\sqrt{\det\left(\frac{\partial y}{\partial x}\cdot\left(\frac{\partial y}{\partial x}\right)^{*}\right)}}\mathrm{vol}_{F^{-1}(y)}(dx).
\end{equation*}

\end{appendices}

\end{document}